\documentclass[sn-mathphys,Numbered]{sn-jnl}% Math and Physical Sciences Reference Style
\usepackage{amsthm, mathtools,bbm,xcolor,soul}
\usepackage{hyperref}
\hypersetup{
    colorlinks,
    citecolor=black,
    filecolor=black,
    linkcolor=blue,
    urlcolor=black
}
\usepackage{graphicx}%
\usepackage{multirow}%
\usepackage{amsmath,amssymb,amsfonts}%
\usepackage{amsthm}%
\usepackage{mathrsfs}%
\usepackage[title]{appendix}%
\usepackage{xcolor}%
\usepackage{textcomp}%
\usepackage{manyfoot}%
\usepackage{booktabs}%
\usepackage{algorithm}%
\usepackage{algorithmicx}%
\usepackage{algpseudocode}%
\usepackage{listings}%
\usepackage{comment}
\usepackage{tikz-cd}
\usetikzlibrary{shapes.geometric}
\theoremstyle{definition}
\newtheorem*{defn*}{Definition}

\newtheorem{thm}{Theorem}[section]

\theoremstyle{definition}
\newtheorem{defn}[thm]{Definition}

\numberwithin{equation}{section}

\newtheorem{cor}[thm]{Corollary}

\newtheorem*{cor*}{Corollary}

\newtheorem{prop}[thm]{Proposition}

\newtheorem*{prop*}{Proposition}

\newtheorem{lem}[thm]{Lemma}

\newcommand{\R}{\mathbb{R}}

\newcommand{\Z}{\mathbb{Z}}

\newcommand{\ZZ}{\mathcal{Z}}

\newcommand{\ZZZ}{\mathsf{Z}}

\newcommand{\Q}{\mathcal{Q}}

\newcommand{\QQ}{\mathsf{Q}}

\newcommand{\N}{\mathbb{N}}

\newcommand{\NN}{\mathcal{N}}

\newcommand{\E}{\mathbb{E}}

\newcommand{\F}{\mathcal{F}}

\newcommand{\var}{\text{var}}

\newcommand{\Var}{\text{Var}}

\def\P{\mathbb {P}}

\newcommand{\id}{\mathbbm{1}}

\newcommand{\D}{\mathbf{D}}

\newcommand{\C}{\mathbf{C}}
                
\newcommand{\M}{\mathcal{M}}

\newcommand{\0}{\mathbf{0}}

\newcommand{\SP}{\mathcal{S}}

\newcommand{\K}{\mathcal{K}}

\begin{document}

\title[Article Title]{Heavy Traffic Scaling Limits for Shortest Remaining Processing Time Queues with Light Tailed Processing Time Distributions}


%\author*[1]{\fnm{Chunxu} \sur{Ji}}\email{ji003@csusm.edu}
\author[1]{\fnm{Chunxu} \sur{Ji}}\email{ji003@csusm.edu}

\author[1]{\fnm{Amber L.} \sur{ Puha}}\email{apuha@csusm.edu}
\equalcont{Research supported in part by NSF Grant DMS-2054505.}


\affil[1]{\orgdiv{Department of Mathematics}, \orgname{California State University San Marcos}, \orgaddress{\street{333 S.\ Twin Oaks Valley Road}, \city{San Marcos}, \postcode{92096}, \state{CA}, \country{USA}}}

% \affil[2]{\orgdiv{Department}, \orgname{Organization}, \orgaddress{\street{Street}, \city{City}, \postcode{10587}, \state{State}, \country{Country}}}

% \affil[3]{\orgdiv{Department}, \orgname{Organization}, \orgaddress{\street{Street}, \city{City}, \postcode{610101}, \state{State}, \country{Country}}}


% \abstract{The abstract serves both as a general introduction to the topic and as a brief, non-technical summary of the main results and their implications. Authors are advised to check the author instructions for the journal they are submitting to for word limits and if structural elements like subheadings, citations, or equations are permitted.}
\abstract{
We prove a heavy traffic scaling limit for a shortest remaining processing time queue. We are interested in the case where the processing time distribution has a tail that decays rapidly, i.e., has light tails. In particular, we revisit the work in Puha (2015), which shows that the diffusion scaled queue length process multiplied by a processing time distribution dependent factor that tends to infinity converges to a nontrivial reflecting Brownian motion, under the condition that this distribution dependent factor is slowly varying and obeys a certain rate of convergence condition. Here, we prove that the rate of convergence condition is not needed and the result holds more generally. We further show convergence of a sequence of nonstandardly scaled measure valued state descriptors to a point mass at one such that the total mass fluctuates randomly in accordance with the diffusion limit for the workload process. This is a sharp concentration result which shows that, under this nonstandard scaling, there are a very small number of tasks in the system and the remaining work for each such task is large and of the same order of magnitude as that of other tasks. This is due to the prioritization of the task with the least remaining work, and is in contrast to the case of heavy tailed processing times studied in Banerjee, Budhiraja, and Puha (2022). There it is shown that, while there is some concentration, the remaining times of the very small number of tasks in the system spread out over the nonnegative real line according to a random profile under this nonstandard scaling. Thus, this work completes the description of the two fundamentally different behaviors of SPRT by characterizing it in the case of light tailed processing time distributions.
}




\maketitle

\section{Introduction}\label{sec1}

% The Introduction section, of referenced text \cite{bib1} expands on the background of the work (some overlap with the Abstract is acceptable). The introduction should not include subheadings.

% Springer Nature does not impose a strict layout as standard however authors are advised to check the individual requirements for the journal they are planning to submit to as there may be journal-level preferences. When preparing your text please also be aware that some stylistic choices are not supported in full text XML (publication version), including coloured font. These will not be replicated in the typeset article if it is accepted. 

We study the behavior of a single-server queue operating under the shortest remaining processing time (SRPT) scheduling policy. For this tasks arrive to the system one at a time. Each such task has associated with it an amount of work for the server. This work is measured in the time units it will take for the server to process the task and so is called the processing time. In an SRPT queue, at each instant in time, the server processes the work associated with the task in the system that has the shortest remaining processing time. This is done with preemption such that when a task arrives with a processing time that is shorter than that remaining for the task in service, the task in service is placed on hold while the server commences processing the work associated with the new arrival. SRPT enjoys an optimality property in that it minimizes the number of tasks in the system  (see \cite{Old Optimal,New Optimal}), and therefore it is naturally of interest (e.g. \cite{Chen}). However, due to its complex dynamics, a closed form or exact performance analysis is not possible. Therefore, tractable approximations are of value.

There is a body of work that concerns functional law of large numbers approximations or fluid limits.  For this, we make note that in order to track the system state as it evolves over time, it is necessary to keep track of the remaining processing times of all tasks in the system at each time. In \cite{Down, Down_Sig}, a measure value process that at each time has a unit atom at the remaining processing time of each task in the system is used for this. The authors of \cite{Down, Down_Sig} develop a fluid approximation for this measure valued process that gives rise to a fluid analog of the response times of tasks in the system at time zero as a function of their remaining processing time.  In the sequel  \cite{KrukSoko}, a fluid limit for a multiclass SRPT queue is developed that fully justifies the fluid response time approximation in \cite{Down, Down_Sig}.  More general measure valued fluid limits for a range of priority queues such as SRPT with time vary arrivals and service rates are developed in \cite{Atar}. 
Steady state mean response times are also of interest.  The survey paper \cite{Schreiber} contains a summary of the work in this direction through the early 1990's. More recently, the work in \cite{Lin} lets the system load approach one and finds a growth rate that is consistent with that of the fluid sojourn times for critically loaded fluid model in \cite{Down, Down_Sig}. Multiserver versions are analyzed in \cite{Grosof} and \cite{Dong}, with the later focusing on the overloaded setting with task abandonment.  

Here we consider a sequence of systems in heavy traffic with a common light tailed processing time distribution that is continuous and has unbounded support.   We apply distribution dependent scaling to the sequence of measure valued state descriptors and prove a heavy traffic limit theorem that gives rise to a measure valued diffusion approximation that is an atom at one with the total mass randomly fluctuating in accordance with the limiting workload process.  The concentration of mass at one in the limit reflects that all tasks in the system have remaining processing times are asymptotically large and of the same order of magnitude.  Our main results are Theorems \ref{W convergence} and \ref{mv convergence}.  These results are complementary to those in \cite{heavy tails}, which apply to heavy tailed processing time distributions, and complete the characterization of the dichotomy of behaviors in the two regimes.

To describe this more fully, let $F$ denote the cumulative distribution function of the processing times and $\overline{F}=1-F$.
In this work, we assume that for all $t>1$,
\begin{equation}\label{lt}
\lim_{x\to\infty}\frac{\overline{F}(tx)}{\overline{F}(x)}=0.
\end{equation}
In particular, $\overline{F}$ is said to be rapidly varying with index $-\infty$.   The Weibull distribution, which includes the exponential distribution as a special case, satisfies \eqref{lt}.  In general, distributions that satisfy \eqref{lt} have tail probabilities that decay faster than the tail probabilities of any power law distribution.

We consider a sequence of systems indexed by members of a sequence $R$ of positive real numbers that increase to infinity. We scale the measure valued state descriptors using standard diffusion scaling combine with a distribution dependent function determined by $F$ that boosts the mass and shrinks the size of the remaining processing time of each task in the system.  Specifically, let $v$ denote a random variable whose distribution is equal to that of the common light-tailed processing times, $\E$ denote the expectation operator and $\id_A$ denote the indicator function of the Borel measurable subset $A$ of $\R_+$.  For $x\in\R_+$, let 
\begin{equation}\label{S}
S(x)=\frac{1}{\E[v\id_{[v>x]}]}.
\end{equation}
Then $S:\R_+\to\R_+$ is nondecreasing and $\lim_{x\to\infty}S(x)=\infty$.
For $r\in\R_+$, let
\begin{equation}\label{Sinv}
S^{-1}(r)=\inf\{x\in\R_+:S(x)>r\}.
\end{equation}
Then $\lim_{r\to\infty}S^{-1}(r)=\infty$.
Given $r\in R$, $S^{-1}(r)$ is the supplemental distribution dependent scale factor for the $r$-th system.  In particular, we apply standard diffusion scaling, multiplying time by $r^2$ and dividing space by $r$.  We further multiply space by $S^{-1}(r)$ to boost the mass of each task.  Finally, we relocate the atom associated with each task in the system by dividing its remaining processing time by $S^{-1}(r)$ and moving the atom to this location. This boosting and relocation is necessary to prevent the total mass from tending to zero and shifting to infinity respectively.  Its specific form is motivated in part by the fluid model introduced in \cite{Down, Down_Sig}.  This connection is explained more fully in Section \ref{DD Scaling}.

The distribution dependent scaling described above is necessary in order to obtain a non-trivial limit since, as shown in \cite{Gromoll}, standard diffusion scaling results in a limiting queue length process that is identically equal to zero under typical asymptotic assumptions.  One might interpret this as saying that under diffusion scaling the number of tasks in the system is very small, so small that it vanishes in the limit.  Yet, with SRPT, the server works constantly unless there are no tasks
in the system, and so, the classical work \cite{MultipleChannel} says the diffusion scaled workload process converges to a non-degenerate reflecting Brownian motion $W^*$.  Thus, these very few tasks in the system must have excessively large remaining processing times.  These two results are restated here in Section \ref{FCLTs} in preparation for stating the main results in this work.

The work in \cite{Puha} explored the order of magnitude difference between the queue length and workload processes uncovered in \cite{Gromoll} in the case where \eqref{lt} holds and $S^{-1}$ is assumed to obey a certain certain rate of convergence condition.
Theorem 3.1 in \cite{Puha} states that when this and other natural asymptotic conditions hold, the diffusion scaled queue length processes multiplied by the distribution dependent boosting factors converge to the limiting workload process.  The boosting factor for the system indexed by $r\in R$ is $S^{-1}(r)$.  Thus, the main result in \cite{Puha} says that all of the tasks in the $r$-th system embody remaining processing times that are of order $S^{-1}(r)$ as $r\to\infty$.  This suggests that a measure valued limit theorem of the type established in this work should hold and raises the question as to whether the rate of converge condition on $S^{-1}$ is necessary.  Theorem \ref{mv convergence} proved here develops the measure valued limit and does so without the rate of convergence condition.

The techniques used to prove Theorem \ref{mv convergence} in part involve adapting those developed in \cite{heavy tails} for the analysis of the heavy tailed processing time distributions to that of light tailed processing time distributions. In \cite{heavy tails}, the authors study a sequence of SRPT queues in heavy traffic with a common continuous processing time distribution that has regularly varying tails.  Specifically, in \cite{heavy tails}, it is
assumed that there exists a $p>1$ such that for all $t>0$
\begin{equation}\label{ht}
\lim_{x\to\infty}\frac{\overline{F}(tx)}{\overline{F}(x)}=t^{-p-1}.
\end{equation}
All power law distributions with greater than two finite moments satisfy \eqref{ht}.
Under this assumption and using the same distribution dependent scaling as in this work, they prove a limit theorem that gives rise to a measure valued diffusion model approximation that does not concentrate at one.
But rather, it is an absolutely continuous measure with a randomly fluctuating profile governed by a random field that arises as the limit of the so-called cutoff distribution dependent scaled workload processes (see \cite[Theorems 1-3]{heavy tails}).  This shows that heavy tailed processing times lead to mass spreading out in the limit, which is in sharp contrast to the concentration results proved here for light tailed processing time distributions.

Theorem \ref{W convergence} here concerns the convergence of the cutoff distribution dependent scaled workload processes in the light tailed case where \eqref{lt} holds.  For this, we note that integrating the measure valued state descriptor against the identity function $\chi$ yields the traditional workload process (see \eqref{defQandW}). A cutoff version is obtained by integrating against $\chi\id_{[0,a]}$ to compute the total work of tasks in the system with remaining processing time less or equal to $a$ for $a\in\R_+$.  We apply the cutoff after applying the distribution dependent scaling and find that cutting off below one ($a<1$) leads to a trivial limit of zero, while cutting off above one ($a>1$) leads to the limit $W^*$, the same limit that is obtained without any cutoff ($a=\infty$).  In Theorem \ref{W convergence}, joint convergence over $a\in\R_+$ to this step function is shown, which is fundamentally different from the continuous random field that arises when \eqref{ht} holds (see \cite[Theorem 1]{heavy tails}).

Theorem \ref{W convergence} is a precursor to Theorem \ref{mv convergence}, but it does not imply it, as it does not imply that the number of tasks with arbitrarily small remaining processing times vanish in the limit.  This is a subtle behavior and the proof of this given in \cite{Puha} in route to proving their Theorem 3.1 has a gap (see the discussion following the statement of Lemma \ref{a<1 convergence}).  Here, we state this fact in Theorem \ref{In between 0 and 1} and note that its proof is quite intricate and involves adapting the techniques used in \cite{heavy tails} to handle light tails.  This is explained in detail in Section \ref{scn:ProofofQconvergence}.  Thus another contribution of this work is to prove Theorem \ref{In between 0 and 1} in this light tailed setting where \eqref{lt} holds, which is crucial step toward proving Theorem \ref{mv convergence}.

The paper is organized as follows.
We begin by introducing some basic notation in Section \ref{scn:notation}.
In Section \ref{scn:model}, we introduce our SRPT queue model and its associated measure valued state descriptor. We also specify a sequence of such queues in the heavy traffic regime and summarize the prior works concerning
limits that arise under functional central limit theorem scaling.
In Section \ref{scn:main}, we introduce the distribution dependent scaling and thereafter, state the two main results in the paper.
The first result concerns convergence in distribution of the random field of cutoff distribution dependent scaled workload processes
(see Theorem \ref{W convergence}). The second result concerns convergence in distribution of the sequence of distribution dependent scaled to
measure valued state descriptors to a measure value diffusion model approximation (see Theorem \ref{mv convergence}).
In Section \ref{scn:SandSInv}, we develop some consequences of assumption \eqref{lt} that will be leveraged throughout the analysis.
In Sections \ref{scn:ProofofWconvergence} and \ref{scn:ProofofQconvergence}, we prove the Theorem \ref{W convergence} and Theorem \ref{mv convergence} respectively. 

\subsection{Notation}\label{scn:notation}
We let $\N$ denote the set of positive integers, $\Z$ denote the set of integers, $\Z_+$ denote the set of nonnegative integers, $\R$ denote the set of real numbers and $\R_+$ denote the set of nonnegative real numbers. Given $a,b\in\R$, we let $a\vee b$ (resp.\ $a\wedge b$) denote the maximum (resp.\ minimum) of $a$ and $b$.  For Borel measurable functions $f:\R_+\to(0,\infty)$ and $g:\R_+\to(0,\infty)$, $f=o(g)$ means $\lim_{x\to\infty}f(x)/g(x)=0$, whereas $f\sim g$ means $\lim_{x\to\infty}f(x)/g(x)=1$. For $f:\R_+ \to \R$ and $T\in[0,\infty)$, we let $\|f\|_T := \sup_{0\leq t\leq T}\lvert f(t)\rvert$ and $\|f\|_\infty := \sup_{0\leq t< \infty}\lvert f(t)\rvert$.  We let $\C(\R_+)$ denote the set of real-valued functions on $\R_+$ that are continuous and $\C_b(\R_+)$ consist of those members of $\C(\R_+)$ that are bounded. We let $0(\cdot)$ denote the function in $\C(\R_+)$ that is identically zero, i.e., $0(t)=0$ for all $t\geq0$. 
For a probability space $(\Omega,\mathcal{F},\P)$, if $A\in\mathcal{F}$, the indicator function $\id_A$ of $A$ is given by $\id_A(\omega)=1$ if  $\omega\in A$ and $\id_A(\omega)=0$ otherwise.
When $A=\R_+$, we denote $\id_A$ by $\id$.  For random variables $X$ and $Y$ that have the same distribution, we write $X\overset{d}{=}Y$.

In this paragraph, we fix an arbitrary Polish space $\SP$, i.e.,
a separable completely metrizable topological space.
A function $f:[0,\infty)\to\SP$ is an r.c.l.l.\ function if it is right continuous and has left limits in $\SP$.
We let $\D([0,\infty),\SP)$ be the set of r.c.l.l.\ functions of time taking values in $\SP$. We endow $\D([0,\infty),\SP)$ with Skorokhod $J_1$-topology. All $\SP$-valued stochastic processes considered throughout are r.c.l.l.\ with probability one. For a sequence $\{Z^n(\cdot)\}_{n\in\N}$ of $\SP$-valued stochastic processes and an $\SP$-valued stochastic process $Z(\cdot)$,
we write $Z_n(\cdot)\Rightarrow Z(\cdot)$ as $n\to\infty$ to denote convergence in distribution in $\D([0,\infty),\SP)$. For $\SP=\R$ and $f\in\C(\R)$ we say that $\{Z_n(\cdot)\}_{n\in\N}$ converges in probability to $f$ and we write $Z_n(\cdot)\overset{p}{\to} f(\cdot)$ as $n\to\infty$ if and only if $\lim_{n\to\infty}\P\left(\|Z_n-f\|_T>\theta\right)=0$ for all $T>0$ and $\theta\in(0,1)$. In this case, $Z_n(\cdot)\Rightarrow f(\cdot)$ as $n\to\infty$ if and only if $Z_n(\cdot)\overset{p}{\to}f(\cdot)$ as $n\to\infty$ (see p.24-25 in \cite{Billingsley}).

The one-dimensional Skorokhod map $\Gamma$, which we define here, plays a significant role in our analysis. For this, we let $\D_0([0,\infty),\R)$ denote the space of all $f\in\D([0,\infty),\R)$ with $f(0)\geq0$. For $f\in \D_0([0,\infty),\R)$, $\inf_{0\leq s\leq t}f(s)$ is well defined for each $t\geq0$, and we define 
\begin{equation}\label{skorokhod}
\Gamma[f](t)=f(t)-\inf_{0\leq s\leq t}f(s)\wedge 0,\quad t\geq0.
\end{equation}
Then $\Gamma[f](\cdot)\in\D_0([0,\infty),\R_+)$.
In Appendix \ref{ap:SM} we specify three useful properties of $\Gamma$ that are used in the proofs that come later.  See \cite[Section 13.5]{Whitt} for further details.

We let $\M$ be the set of finite, nonnegative Borel measures on $\R_+$. Given $\zeta\in\M$ and a Borel measurable function $f:\R_+\to\R$ that is integrable with respect to $\zeta$, we let $
\langle f,\zeta\rangle :=\int_{\R_+} f(x)\,\zeta(dx)$.
The set $\M$ is endowed with the topology of weak convergence such that $\zeta_n\xrightarrow{w}\zeta$ as $n\to\infty$ if and only if $\lim_{n\to\infty}\langle f,\zeta_n\rangle=\langle f,\zeta\rangle$ for any $f\in\C_b(\R_+)$. With this topology, $\M$ is a Polish space \cite{Prokhorov}. We let $\0$ denote the zero measure in $\M$ and $\0(\cdot)$ denote the element of $\D([0,\infty),\M)$ that is identically equal to $\0$. Also, for any $x\in\R_+$, we let $\delta_x^+$ be the Dirac measure on $\R_+$ such that for any Borel measurable $A\subseteq \R_+$, 
\begin{equation}\label{atom+}
\delta_x^+(A)=
\begin{cases}
1,\qquad\text{if $x\in A\setminus\{0\}$,}\\
0,\qquad\text{otherwise}.
\end{cases}
\end{equation}

%-----------------------Section 2----------------------------------------------------
\section{Model, sequence of systems, and prior results}\label{scn:model}
In Section \ref{The Shortest Remaining Processing Time Queue Model}, we define the stochastic model for an SRPT queue. In Section \ref{measure valued State Descriptor}, we introduce the key processes that track the system state and describe its performance. In Section \ref{sequence}, we introduce the sequence of SRPT queues in heavy traffic. In Section \ref{FCLTs}, we summarize some of the functional central limit theorems established in prior work, which motivate the need for nonstandard scaling.

%-----------------------Section 2.1----------------------------------------------------
\subsection{The Shortest Remaining Processing Time Queue Model}\label{The Shortest Remaining Processing Time Queue Model}
Suppose that there is a single server processing the work associated with incoming tasks. For $t\geq0$, we let $E(t)$ denote the number of tasks that arrive to the system in $(0,t]$. We assume that $E(\cdot)$ follows a delayed renewal process with positive rate $\lambda$. Then the times between consecutive arrivals, which we refer to as the interarrival times, are positive i.i.d.\ random variables with finite mean $1/\lambda$ and finite standard deviation $\sigma_A$. The delay refers to the time of the first arrival, which is assumed to be positive and independent of future interarrival times, but may have a different distribution. The process $E(\cdot)$ is an instance of a counting process that has jumps of size one.\\
\indent 
Each task that enters the system after time 0 has associated with it a processing time, which is the amount of time that the server must spend working to complete the work embodied in that task. We let the processing times of the incoming tasks be positive i.i.d.\ random variables $\{v_i\}_{i\in\N}$, where $v_i$ denotes the processing time of the $i$-th task to arrive to the system. We assume that $E(\cdot)$ and $\{v_i\}_{i\in\N}$ are mutually independent. Also, we let $v$ be a generic random variable whose distribution is equal to that of the processing time of each task and independent of $E(\cdot)$. We assume that $\P(v=0)=0$ and
that the expected value of $v$ denoted $\E[v]$ and the standard deviation of $v$ denoted $\sigma_S$ are positive and finite.\\
\indent We assume that there are $\QQ_0$ tasks in the system at time zero. Here $\QQ_0$ is a $\Z_+$ valued random variable. We index the tasks in the system at time zero using the integers $-\QQ_0+1,-\QQ_0+2,\dots,-1,0$. Let $\{v_i\}_{i=-\infty}^0$ be a sequence of strictly positive random variables. Then for $-\QQ_0+1\leq i\leq 0$ the task with index $i$ has remaining processing time $v_i$ at time zero. We assume that 
$\sum_{i=-\QQ_0+1}^0v_i<\infty$
almost surely. Finally, we assume that $(\QQ_0,\{v_i\}_{i = -\infty}^0)$ is mutually independent of $(E(\cdot), \{v_i\}_{i=1}^\infty)$.
\\
\indent Next we must specify a scheduling policy to determine when the server will process the work associated with each task. The SRPT scheduling policy is such that at every moment when the system is not empty, the task with the shortest remaining time to completion is served first. In particular, upon the server completing the work associated with a given task, that task departs the system and the task in the system with smallest remaining processing time enters service, if such a task exists. Otherwise, there are no other tasks in the system and the server idles until a new task arrives and that task enters service. The scheduling policy is done with preemption such that when a task arrives with a total processing time that is smaller than the remaining processing time of any other task in the system including that of the task in service, the task in service is placed on hold and the arriving task enters service. Otherwise, the arriving task waits in the system until its processing time is the smallest among all those in the system, as which time it enters service. \\
\indent We let $\rho=\lambda\E[v]$ be the traffic intensity of this system. Then $\rho$ represents the nominal load on the server. Specifically, the system is stable if $\rho<1$, critically loaded if $\rho=1$, and unstable if $\rho>1$. It is common for such a system to operate close to the critically loaded regime. Therefore it is of interest to understand system performance for $\rho\approx1$. We describe our approach for this in Section \ref{sequence} below after introducing the state descriptor and performance processes in Section \ref{measure valued State Descriptor}.

%------------------Section 2.2--------------------------
\subsection{State Descriptor and Performance Processes}\label{measure valued State Descriptor}
In order to describe the state of the system, one must keep track of the remaining processing time of each task in the system. Encoding these in a measure turns out to be effective. For this for $i\in\Z$, let $u_i=\inf\{t\geq0:E(t)\geq i\}$, so that $u_i$ denotes the time of $i$-th arrival. Then $u_i=0$ for $i\leq 0$ and $u_1>0$. Let $T_1 = u_1$ and $T_{i+1} = u_{i+1}-u_i$ for $i\in\N$. Then, for $i\in\N$, $u_i$ denotes the time at which task $i$ arrives to the system, and, for $i\geq2$, $T_i$ denotes the time that elapses between the $(i-1)$-th and $i$-th arrival. Then $\{T_i\}_{i=2}^\infty$ is an i.i.d.\ sequence of positive random variables that is independent of $T_1$. We remind the reader that $T_1$ does not necessarily have the same distribution as $T_i$ for $i\geq2$, as it corresponds to the initial delay. For convenience, we let $T$ denote a random variable that is equal in distribution to $T_2$ and independent of $v$. Note that $\E[T]=1/\lambda$ and $\Var[T]=\sigma_S$.

For each $t\geq0$ and $-\QQ_0+1\leq i\leq E(t)$, we let $v_i(t)$ be the remaining processing time at time $t$ of the $i$-th task to arrive by time $t$. Then, for $i\in\N$, $v_i(u_i)=v_i$ and, for convenience, we define $v_i(t)=v_i$ for $t\in[0,u_i)$.  In addition, for each $t\geq0$ and $i\ge -\QQ_0+1$, $v_i(t)=0$ if and only if the server has completed processing the work associated with task $i$, in which case task $i$ has departed the system by time $t$.  Then, for each $i\ge -\QQ_0+1$, $v_i(\cdot)$ is continuous, decreases at rate one when task $i$ is being served and remains constant otherwise.  For convenience for $i\ge -\QQ_0+1$, we often refer to $v_i(t)$ as the {\it size} of task $i$ at time $t$ or, when it is clear from context which time $t$ is being referenced, we simply say the {\it size} of task $i$.

The associated measure valued state descriptor $\Q(\cdot)\in\D([0,\infty),\mathcal{M})$ is such that, for each $t\geq0$,
\[
\Q(t)=\sum_{i=-\QQ_0+1}^{E(t)}\delta_{v_i(t)}^+.
\]
For this, recall the definition of $\delta_x^+$ in \eqref{atom+}. The $+$ superscript captures that tasks depart the system once their size hits zero. Thus, for all Borel measurable $A\subseteq \R_+$ and $t\geq0$, $\langle\id_A,\Q(t)\rangle=\sum_{i=-\QQ_0+1}^{E(t)}\langle\id_A,\delta^+_{v_i(t)}\rangle=\sum_{i=-\QQ_0+1}^{E(t)}\id_{A\setminus\{0\}}(v_i(t))$, which is the number of tasks in the system with remaining processing time in $A$.\\

\indent Next we define the {\it queue length} and the {\it workload} processes, which are denoted by $\QQ$ and $W$ respectively. To define these, let $\chi$ be the identity function on $\R_+$, i.e., $\chi(x)=x$ for all $x\in\R_+$. Then, for each $t\geq0$,
\begin{equation}\label{defQandW}
 \QQ(t)=\langle \id,\Q(t)\rangle\qquad\text{and}\qquad W(t)=\langle\chi,\Q(t)\rangle.   
\end{equation}
For any $t\geq0$, $\QQ(t)$ is the number of tasks in the system. We note that $\QQ(0) = \QQ_0<\infty$ almost surely by assumption, and for all $t\geq 0$, $\QQ(t)\leq \QQ(0)+E(t)<\infty$. Also, for each $t\geq0$, $W(t)=\sum_{i=-\QQ_0+1}^{E(t)}v_i(t)$, which is the total time that the server must spend working to complete the processing of all tasks in the system. Observe that $W(0) = \sum_{i=-\QQ(0)+1}^0v_i<\infty$ almost surely by assumption. Moreover, for all $t\geq 0$, $W(t)\leq W(0) + \sum_{i=1}^{E(t)}v_i<\infty$ almost surely. The processes $\QQ$ and $W$ are fundamental performance processes as they provide natural measures of congestion.

Finally, we define the {\it load}, {\it netput} and {\it idle time} processes, which are denoted by $V$, $X$ and $I$ respectively. For $t\geq0$, set
\begin{equation}
V(t)=\sum_{i=1}^{E(t)}v_i,\quad X(t)=W(0)+V(t)-t\quad\text{and}\quad I(t):=\int_0^t\id_{\{\QQ(s)=0\}}ds.    \label{definition of V and X}
\end{equation}
For $t\geq0$, $I(t)$ denotes the cumulative time in $[0,t]$ during which their are no tasks in the system and the server idles. Then $W(t)=W(0)+V(t)-t+I(t)=X(t)+I(t)$ for each $t\geq0$.
Due to Lemma 13.4.1 in \cite{Whitt},
$I(t)=-\inf_{0\leq s\leq t}X(s)\wedge 0$
for all $t\ge 0$.
Hence, for each $t\ge 0$,
\begin{equation}\label{I(t)}
W(t)=\Gamma[X](t),
\end{equation}
where $\Gamma$ is the one-dimensional Skorokhod map as defined in \eqref{skorokhod}.

%-----------------------------Section 2.3-------------
\subsection{Sequence of Systems Approaching Heavy Traffic}\label{sequence}
\indent We study the behavior of the system when $\rho\approx 1$. For this, we consider a sequence of SRPT queues approaching the critically loaded regime by allowing the arrival rates to approach $1/\E[v]$, so that the traffic intensity $\rho$ approaches 1. For this, we let $R=\{r_i\}_{i\in\N}$ be a real valued, increasing sequence such that $r_i\in(1/\E[v],\infty)$ for all $i\in\N$ and $\lim_{i\to\infty}r_i=\infty$. To avoid cluttering the notation, we suppress the subscript $i$ and just write $R=\{r\}$. We consider a sequence of SRPT queues indexed by $R$ such that the processing time distribution of each queue does not depend on its index.  We continue to denote the common cumulative distribution function of these $R$ indexed queues by $F$, the sequence of processing times by $\{v_i\}_{i\in\N}$ and the generic random variable with distribution function $F$ by $v$.

For each $r\in R$, we denote the initial condition, parameters, and primitive input, performance and other associated processes of the $r$-th system with a superscript $r$, i.e., $\QQ_0^r$, $\{v_i^r\}_{i=-\infty}^0$, $\lambda^r$, $\sigma_A^r$, $E^r(\cdot)$, $\{u_i^r\}_{i\in\N}$, $T_1^r$, $\{T_i^r\}_{i\geq2}$, $T^r$, $\rho^r=\lambda^r\E[v]$, $\Q^r$, $\QQ^r$, $W^r$, $V^r$, $X^r$ and $I^r$. We assume that for some $\kappa\in\R$, $\sigma_A\in(0,\infty)$, $C_1>0$, and $\ell>0$, 
\begin{align}
&\lim_{r\to\infty}r(\rho^r-1)=\kappa,     \label{r rho convergence}\\ &\lim_{r\to\infty}\sigma_A^r=\sigma_A,     \label{sigma convergence}\\
&\sup_{r\in R}{\mathbb E}[(T_1^r)^2]=C_1<\infty,\qquad\text{and}\qquad\sup_{r\in R}\E\left[(T^{r})^{2+\ell}\right]<\infty.\label{u convergence}
\end{align}
\noindent Since \eqref{r rho convergence} implies that $\lim_{r\to\infty}\rho^r=1$ the sequence of systems is approaching critical loading. In fact, \eqref{r rho convergence} gives that $\rho^r-1$ is of order $1/r$ as $r\to\infty$. With this, the sequence of systems is said to be in heavy traffic.  For convenience, we let $\lambda = 1/\E[v]$
and make note that due to \eqref{r rho convergence},
\begin{equation}\label{lambda}
\lim\limits_{r\to\infty}\lambda^r=\lambda.
\end{equation}
Together \eqref{r rho convergence}, \eqref{sigma convergence} and \eqref{u convergence} ensure that the arrival processes satisfy a functional central limit theorem (see \eqref{E hat convergence} below).

%----------------------------Section 2.4-----------------------------------------------
\subsection{Functional Central Limit Theorems}\label{FCLTs}
\noindent In this section, we review some prior results concerning functional central limit theorems for the workload and queue length processes.
This is of some importance as the limiting workload process appears in the statement of the the main results in this paper.  Henceforth, we fix a standard Brownian motion which we denote as $B(\cdot)$.

To begin, recall \eqref{definition of V and X} and, for $r\in R$ and $t\geq0$, define
\begin{alignat*}{3}
\overline{E}^r(t)&=\frac{E^r(r^2t)}{r^2},
&\widehat{E}^r(t)&=\frac{E^r(r^2t)-\lambda^rr^2t}{r},\\
\widehat{W}^r(t)&=\frac{W^r(r^2t)}{r},
&\widehat{V}^r(t)&=\frac{V^r(r^2t)-\rho^rr^2t}{r},\hbox{ and}\\
\widehat{X}^r(t)&=\frac{X^r(r^2t)}{r},\hbox{ so that }&\widehat{X}^r(t)&=\widehat{W}^r(0)+\widehat{V}^r(t)+(\rho^r-1)rt.
\end{alignat*}
Since time and space are scaled with the same factor, $\overline{E}^r(\cdot)$ is a functional law of large numbers scaling of $E^r(\cdot)$. Similarly processes denoted with a $\widehat{\quad}$ are under functional central limit or diffusion scaling. As a consequence of Corollary 7.3.1 in \cite{Whitt} and assumptions \eqref{r rho convergence} - \eqref{u convergence},
as $r\to\infty$, 
\begin{equation}
    \widehat{E}^r(\cdot)\Rightarrow E^*(\cdot):=\lambda^{3/2}\sigma_AB(\cdot)
    \qquad\hbox{and}\qquad
\overline{E}^r(\cdot)\Rightarrow\lambda(\cdot),
    \label{E hat convergence}
\end{equation}
where $\lambda$ is as in \eqref{lambda}
and $\lambda(t)=\lambda t$ for all $t\geq 0$. 

Recall that $\sigma_S^2=\Var[v]$. Define 
\begin{equation}
\sigma^2=\lambda(\sigma_S^2+\sigma_A^2).\label{sigma}
\end{equation}
The next two results are well known and will be leveraged later on in the analysis.  Proofs are given in Appendix \ref{ap:FCLTs} for the reader's convenience.

\begin{prop}\label{functional clt convergence}
As $r\to\infty$, we have $\widehat{V}^r(\cdot)\Rightarrow\sigma B(\cdot)$.
\end{prop}
-
\begin{prop}\label{Def of X*}
Suppose that there exists a nonnegative random variable $\widehat{W}^*(0)$ such that $\widehat{W}^r(0)\Rightarrow\widehat{W}^*(0)$ as $r\to\infty$. Then, as $r\to\infty$,
\begin{equation}\label{difusion of X}
\widehat{X}^r(\cdot)\Rightarrow X^*(\cdot):=W^*(0)+\sigma B(\cdot)+\kappa(\cdot),
\end{equation}
where $W^*(0)\overset{d}{=}\widehat{W}^*(0)$, $W^*(0)$ is independent of $B(\cdot)$ and $\kappa(t)=\kappa t$ for all $t\geq0$. In addition, as $r\to\infty$,
\begin{equation}\label{diffusion of W}
    \widehat{W}^r(\cdot)\Rightarrow W^*(\cdot):=\Gamma[X^*](\cdot).
\end{equation}
\end{prop}

The convergence stated as \eqref{diffusion of W} in Proposition  \ref{Def of X*} was originally proved in \cite{MultipleChannel} in a bit more generality as the processing time distributions can also vary with $r\in R$.  In addition, it holds for any single server queue obeying a scheduling policy that does not idle when there is work in the system, which includes SRPT.
To the contrary, establishing a functional central limit theorem for the measure valued state descriptor and, in particular, the queue length process requires leveraging the specifics of the SRPT scheduling policy. This was done in \cite{Gromoll}. We restate this result here as Proposition \ref{seminal result} below in the case where the processing time distributions do not vary with $r\in R$ and have unbounded support.  For $r\in R$ and $t\geq0$, let 
\[
\widehat{\Q}^r(t)=\frac{1}{r}\Q^r(r^2t)
\qquad\text{and}\qquad
\widehat{\QQ}^r(t) = \langle\id,\widehat{\Q}^r(t)\rangle.
\]

\begin{prop}\label{seminal result}
Suppose that $\P(v>x)>0$ for all $x\in\R_+$ and
$\widehat{\Q}^r(0)
\Rightarrow \0$ as $r\to\infty$. Then,  
$\widehat{\Q}^r(\cdot)\Rightarrow\0(\cdot)$ as $r\to\infty$.
Equivalently, 
$
\widehat{\QQ}^r(\cdot)\Rightarrow0(\cdot)$ as $r\to\infty$.
\end{prop}

In contrast to the convergence in \eqref{diffusion of W}, Proposition \ref{seminal result} implies that the queue length process of an SRPT queue
vanishes in the limit under diffusion scaling when the common processing time distribution has unbounded support. Thus, in this case, the diffusion scaled queue length process is orders of magnitude smaller than the diffusion scaled workload process. So, as $r\to\infty$, the work is being embodied in a relatively small number of tasks with very large processing times. The results in this paper characterize this order of magnitude difference more fully in the case where the processing time distribution satisfies \eqref{lt}.

%------------Section 3-------------------
\section{Main Results}\label{scn:main}
Henceforth, we assume that $F$ is continuous, $F(0)=0$ (so that $\P(v=0)=0$) and $\overline{F}(\cdot)=1-F(\cdot)$ is rapidly varying with index $-\infty$ as in \eqref{lt}.
Note that \eqref{lt} implies that $F(x)<1$ for all $x\in\R_+$ and that $\E[v]$ and $\sigma_S$ are positive and finite.  We continue to assume that the heavy traffic conditions \eqref{r rho convergence}, \eqref{sigma convergence} and \eqref{u convergence} hold throughout.  

%----------------New Section 3.1---------------
\subsection{Distribution Dependent Scaling}\label{DD Scaling}
\indent To obtain a nontrivial limit for the queue length process and more fully describe the order of magnitude of difference between the workload and queue length processes exhibited in Propositions \ref{Def of X*} and \ref{seminal result}, a distribution dependent scaling is
used. The function $S^{-1}$ defined in \eqref{Sinv} plays a critical role in this.
For convenience, for each $r\in R$, let
\begin{equation}
c^r:=S^{-1}(r).\label{cr is s inverse r}    
\end{equation}
For all $r\in R$, $c^r>0$ since $r>\lambda$. Also, $\lim_{r\to\infty}c^r=\infty$.  To illustrate the rate of growth, we consider Weibull distributed processing times.  If $\overline{F}(x)=\exp(-\lambda x)$ for $x\in\R_+$, then $S(x)=\frac{\lambda e^{\lambda x}}{\lambda x+ 1}$ for $x\in\R_+$ and so $c^r\sim \lambda^{-1}\log{r}$ as $r\to\infty$. More generally, if $\overline{F}(x)=\exp(-(\mu x)^\alpha)$ for $x\in\R_+$, for some $\mu,\alpha>0$, then $c^r\sim\mu^{-1}(\log{r})^{1/\alpha}$ as $r\to\infty$.  In general, $S^{-1}$ is slowly varying, which is shown in Section \ref{scn:SandSInv} where properties of $S$ and $S^{-1}$ are developed.

For $r\in R$ and $t\geq0$, we define 
\begin{equation}
\widetilde{\Q}^r(t)=\frac{c^r}{r}\sum_{i=-\QQ_0^r+1}^{E^r(r^2t)}\delta^+_{\frac{v_i^r(r^2t)}{c^r}}.    \label{Tilde Q}
\end{equation}
Thus, the ``tilde" corresponds to a nonstandard, distribution dependent scaling. Like difussion scaling time is scaled by $r^2$, but space is scaled by $c^r/r$ instead of $1/r$, which boosts the weight of each task to prevent the total mass from vanishing in the limit as it does under diffusion scaling (see Proposition \ref{seminal result}). Furthermore, mass at $x$ in the measure value state descriptor $\Q^r(\cdot)$ is relocated to $x/c^r$ in $\widetilde{\Q}^r(\cdot)$. This prevents mass from sliding out to infinity and being lost in the limit. This scaling takes the same form as that used in \cite{heavy tails}, where \eqref{ht} is assumed to hold so that the processing times in that work are heavy rather than light tailed.

This scaling is motivated by the fluid model developed in \cite{Down}, which has measure valued solutions.  The function $s$ defined in (18) of \cite{Down} is equal to the function $S$ defined here in \eqref{S} times the nondecreasing function $w_0:\R_+\to\R_+$, where for $x\in\R_+$ $w_0(x)$ is
the amount of work in the fluid model at time zero associated with tasks of size less or equal $x$ at time zero.  Then,  for $x\in\R_+$, $s(x)$ is the first time at which the fluid model has no mass in $[0,x]$, and the right continuous inverse of $s$ denoted as $s_r^{-1}$ in \cite{Down} tracks the time evolution of left edge of the support of the measure valued fluid model solution. See Theorem 3.1 in \cite{Down}.  Thus, one can regard the relocation of mass here through division by $S^{-1}(r)=c^r$ in the $r$-th system as centering about the order of magnitude suggested by the left edge of the support in the fluid model.  The multiplication of space by $S^{-1}(r)=c^r$ in the $r$-th system has the boosting effect that gives the appropriate weight for computing each task's contribution to the current work in system.

For $r\in R$ and $t\geq0$, we define 
\begin{equation}
\widetilde{\QQ}^r(t)=\langle\id,\widetilde{\Q}^r(t)\rangle\qquad\text{and}\qquad\widetilde{W}^r(t)=\langle\chi,\widetilde{\Q}^r(t)\rangle.    \label{definition of QQ and W}
\end{equation}
Then
$\widetilde{\QQ}^r(t)=c^r\widehat{\QQ}^r(t)$ and $\widetilde{W}^r(t)=\widehat{W}^r(t)$
for all $r\in R$ and $t\geq0$.  For $r\in R$ and $a\in\R_+$, we define the $a$-{\it cutoff queue length} and {\it workload} processes
\begin{align}\label{cutoff}
    {\QQ}_a^r(\cdot)=\langle\id_{[0,a]},{\Q}^r(\cdot)\rangle\qquad&\text{and}\qquad{W}_a^r(\cdot)=\langle\chi_a,{\Q}^r(\cdot)\rangle,
\end{align}
and the
$a$-{\it cutoff scaled queue length} and {\it workload} processes
\begin{align*}
    \widehat{\QQ}_a^r(\cdot)=\langle\id_{[0,a]},\widehat{\Q}^r(\cdot)\rangle\qquad&\text{and}\qquad\widehat{W}_a^r(\cdot)=\langle\chi_a,\widehat{\Q}^r(\cdot)\rangle,\\
    \widetilde{\QQ}_a^r(\cdot)=\langle\id_{[0,a]},\widetilde{\Q}^r(\cdot)\rangle\qquad&\text{and}\qquad\widetilde{W}_a^r(\cdot)=\langle\chi_a,\widetilde{\Q}^r(\cdot)\rangle.
\end{align*}
Then $\widetilde{\QQ}_a^r(t)=c^r\widehat{\QQ}_{ac^r}^r(t)$ and $\widetilde{W}_a^r(t)=\widehat{W}_{ac^r}^r(t)$
for all $r\in R$, $a\in\R_+$ and $t\geq0$.

%----------------New Section 3.2---------------
\subsection{Asymptotic conditions for the sequence of initial conditions}\label{initial}

As in Proposition \ref{Def of X*}, we assume that there exists a nonnegative random variable $\widehat{W}^*(0)$ such that $\widehat{W}^r(0)\Rightarrow\widehat{W}^*(0)$ as $r\to\infty$. Henceforth, $W^*$ as in Proposition \ref{Def of X*}. In particular, $W^*(0)$ is independent of $B(\cdot)$. Then, since $\widetilde{W}^r(\cdot) = \widehat{W}^r(\cdot)$ for all $r\in R$, it follows from \eqref{diffusion of W} that 
$\widetilde{W}^r(\cdot)\Rightarrow W^*(\cdot)$ as $r\to\infty$.
Let 
\[
W^*_a(0) = \begin{cases}
    0, & a<1,\\
    W^*(0), & a\geq1.
\end{cases}
\]
We assume that
\begin{align}
\left(\left(\widetilde{W}^r_a(0),a\in\R_+\right),\widetilde{W}^r(0)\right)\Rightarrow \left(\left(W_a^*(0), a\in\R_+\right),W^*(0)\right),\qquad\text{as } r\to\infty.   \label{W sub joint convergence}
\end{align}
Note that $\left(\left(W_a^*(0), a\in\R_+\right),W^*(0)\right)$ and $B(\cdot)$ are mutually independent since $W^*(0)$ and $B(\cdot)$ are mutually independent.
In addition, for all $a\in[0,1)$, $\widetilde{W}_a^r(0)\Rightarrow 0$ as $r\to\infty$, and so
\begin{equation} 
 \widetilde{W}^r_a(0)\overset{p}{\to}0,\qquad\text{as } r\to\infty.   \label{W tilde sub a r 0 convergence to 0 in prob}
\end{equation}
\begin{comment}
   We assume that $\E[W^*(0)]<\infty$ and $\{\widetilde{W}^r(0)\}_{r\in R}$ is uniformly integrable, i.e., that
\begin{equation}
\lim_{M\to\infty}\sup_{r\in R}\E[\widetilde{W}^r(0)\id_{\{\widetilde{W}^r(0)\geq M\}}] = 0.\label{limsup upper bound of tilde w}
\end{equation}
Then, 
$
\lim_{r\to\infty}\E[\widetilde{W}^r_a(0)] = \E[W_a^*(0)]$ and $\lim_{r\to\infty}\E[\widetilde{W}^r(0)] = \E[W^*(0)]$ for all $a\in \R_+$.
In particular, for $a\ge 1$,
\begin{equation}
\lim_{r\to\infty}\E[\widetilde{W}^r(0)-\widetilde{W}^r_a(0)] =0.\label{E W tilde r W tilde a}
\end{equation} 
\end{comment}
We further assume that there exists $\eta_0\in(0,1)$ and $M_0>1$ such that 
\begin{equation}
\limsup_{r\to\infty}\sup_{a\in[M_0/c^r,1]}a^{-(1+\eta_0)}\E[\widetilde{W}^r_a(0)]<\infty.\label{boundedness of tilde W by a}
\end{equation}
Finally, we assume that for all $a\in \R_+$, 
\begin{equation}
\widetilde{\Q}^r(0)\Rightarrow W^*(0)\delta_1^+,\qquad\text{as } r\to\infty.\label{q tilde initial 0}
\end{equation}

The condition \eqref{W sub joint convergence} is analogous to condition (2.14) in \cite{heavy tails}, although the limit in \cite{heavy tails} is a continuous process that increases from zero to $W^*(0)$ as $a$ tends to infinity, rather than a step function.
Condition \eqref{W sub joint convergence} is also in the same spirit as condition (3.9) in \cite{Puha}, but more straightforward to understand.
Condition \eqref{boundedness of tilde W by a} is equivalent to condition (2.16) in \cite{heavy tails}, but restated to avoid reference to $p$ since \eqref{lt} holds instead of \eqref{ht}.
Finally, condition \eqref{q tilde initial 0} takes the place of condition (2.19) in \cite{heavy tails} as it is necessary for our measure valued limit theorem and it implies that condition (2.19) in \cite{heavy tails} holds.  Conditions (2.15) and (2.17) in \cite{heavy tails} are not assumed to hold here as they are not needed due to the constant nature of the limit process $W_{\cdot}^*(0)$ above one.

%------------Section 3.3------------------
\subsection{Limit Theorems}\label{Limit Theorems}
 The first theorem describes the limiting behavior of the random field of cutoff distribution dependent scaled workload processes. Recall $X^*$ and $W^*=\Gamma[X^*]$ from Proposition \ref{Def of X*}. For $t\geq0$, 
\[
X_a^*(t):=
\begin{cases}
-\infty,& \text{if $0\leq a<1$,}\\
X^*(t)-\lambda t,&\text{if }a=1,\\
X^*(t),&\text{if }a>1.
\end{cases}
\hbox{ and}\quad
W^*_a(t):=
\begin{cases}
0,&\text{if $0\leq a<1$,}\\
\Gamma[X^*_1](t),&\text{if $a=1$,}\\
\Gamma[X^*](t),&\text{if }a>1.
\end{cases}.
\] 
\begin{thm}\label{W convergence}
For any $k\in\N$ and $0\leq a_1<\cdots<a_k<\infty$, as $r\to\infty$, 
\[
(\widetilde{W}_{a_1}^r(\cdot),\dots,\widetilde{W}_{a_k}^r(\cdot),\widetilde{W}^r(\cdot))\Rightarrow(W^*_{a_1}(\cdot),\dots,W^*_{a_k}(\cdot),W^*(\cdot)).
\]
\end{thm}

The result in Theorem \ref{W convergence} is the light tailed analog of Theorem 1 in \cite{heavy tails}. It states that for processing time distributions that satisfy \eqref{lt}, there is a sharp concentration around $1$ of the scaled remaining processing times that contribute to the workload. This is so much so that as $r\to\infty$, $\widetilde{W}_a^r(\cdot)\Rightarrow W^*(\cdot)$ for all $a>1$ and $\widetilde{W}_a^r(\cdot)\Rightarrow 0(\cdot)$ for all $0\leq a<1$. This is in contrast to the heavy tailed result \cite[Theorem 1]{heavy tails}, where the limit is a continuous random field driven by a common Brownian motion plus a spatially dependent drift.
Theorem \ref{W convergence} is proved in Section \ref{scn:ProofofWconvergence}.  The proof employs a continuous mapping approach to handle $a\ge 1$.  This technique breaks down for $a<1$, as certain drift terms diverge to negative infinity (see Lemma \ref{drift}).  To handle $a<1$, sharpened versions of the techniques used in \cite{Puha} are leveraged (see the discussion following Lemma \ref{a<1 convergence}).  

Theorem \ref{W convergence} shows a concentration around 1 of the task sizes that contribute to the workload under the distribution dependent scaling. However, tasks with small remaining processing times contribute very little to the work in the system. So it does not preclude a build up of tasks with remaining processing times near zero. Our second result implies that such a phenomena does not occur. It states that under the distribution dependent scaling the measure valued state descriptor converges to a point mass at 1, with the mass at 1 randomly varying in time according to $W^*(\cdot)$.
\begin{thm}\label{mv convergence}
As $r\to\infty$, $\widetilde{\Q}^r(\cdot)\Rightarrow W^*(\cdot)\delta_1^+$.
\end{thm}
 The concentration result in Theorem \ref{mv convergence} corresponds to a concentration of mass around $c^r$ under diffusion scaling with a boosting factor of $c^r$. This generalizes the main result in \cite{Puha} by stating it in terms of the sequence of rescaled measure valued processes and also eliminates a restrictive rate of convergence condition on $S^{-1}$ (see \cite[Section 2.2]{Puha}). It also provides a light tailed analog of Theorem 3 in \cite{heavy tails}, which shows that in the heavy tailed case when \eqref{ht} holds the total mass in the limit at time $t$ is $W^*(t)$, but rather than concentrating as a point mass at one, it spreads out over $\R_+$ in a manner determined by the random field that arises there as the limit of the cutoff distribution dependent scaled workload processes.
 
Theorem \ref{mv convergence} is proved in Section \ref{scn:ProofofQconvergence}.  As mentioned above, the main issue is to show that the mass near the origin vanishes in the limit, which is quite delicate.  Specifically, we prove Theorem \ref{a<1 queue mass to zero} stated below in Section \ref{scn:ProofofQconvergence}. For this, techniques from \cite{heavy tails} are adapted to handle the light tailed setting.  In addition, we provide an exposition that some might find more accessible to that in \cite{heavy tails}. See the discussions following the statements of Theorem \ref{a<1 queue mass to zero} and Lemma \ref{Lemma 18 of paper}.

%--------------------------Section 4
\section{Properties of $S$ and $S^{-1}$}\label{scn:SandSInv}

In this work, we consider SRPT queues with random processing times such that the tail distribution function decays sufficiently fast, i.e., satisfies \eqref{lt}.  Here we develop the properties of $S$ and $S^{-1}$ that result from \eqref{lt} and will be used throughout the analysis.  Lemmas \ref{T1 and r convergence}, \ref{S property a} and \ref{S property} are the principle results in this section.

As in \cite{Bingham}, a Borel measurable function $f:\R_+\to(0,\infty)$ is said to be {\it rapidly varying with index} $\infty$ (resp.\ $-\infty$) if for any $t>1$, we have
$\lim_{x\to\infty}f(tx)/f(x)=\infty\hbox{ (resp.\ 0)}$, and {\it slowly varying} if for any $t>0$, $\lim_{x\to\infty}f(tx)/f(x)=1$.
We let $\textbf{R}_\infty$, $\textbf{R}_{-\infty}$, and $\textbf{R}_0$ denote the set of strictly positive Borel measurable functions on $\R_+$ that are rapidly varying with index $\infty$, rapidly varying with index $-\infty$, and slowly varying respectively.  If $f\in\textbf{R}_\infty$ (resp.\ $\textbf{R}_{-\infty}$), then it is easy to show that
$\lim_{x\to\infty}f(tx)/f(x)=0$ (resp.\ $\infty$) for all $0<t<1$.  In addition, if $f\in\textbf{R}_0$, then Theorem 1.5.4 in \cite{Bingham} (which we restate next for the reader's convenience) states that $f$ has subpolynomial growth.
\begin{prop}\label{subpolynomial growth of slowly varying functions}
If $f\in\textbf{R}_0$, then $f(x) = o(x^\gamma)$ as $x\to\infty$ for all $\gamma>0$.
\end{prop}
The subpolynomial growth of slowly varying functions together with the assumption \eqref{u convergence} gives the following result concerning the arrival times of the first task.  
\begin{lem}\label{T1 and r convergence}
Suppose that $f\in\textbf{R}_0$ and $\gamma_1, \gamma_2>0$. As $r\to\infty$,
\[
\frac{\left(f(r)\right)^{\gamma_1}T_1^r}{r^{\gamma_2}}\overset{p}{\to}0.
\]
\begin{proof}
By Markov's inequality and \eqref{u convergence}, for all $\theta>0$ and $r\in R$, 
\[
\P\left(\frac{\left(f(r)\right)^{\gamma_1}T_1^r}{r^{\gamma_2}}>\theta\right)\leq\frac{1}{\theta^2}\E\left[\frac{\left(f(r)\right)^{2\gamma_1}(T_1^r)^2}{r^{2{\gamma_2}}}\right]\leq\frac{C_1}{\theta^2}\left(\frac{f(r)}{r^{\gamma_2/\gamma_1}}\right)^{2\gamma_1}.
\]
Letting $r\to\infty$ and using Proposition \ref{subpolynomial growth of slowly varying functions} completes the proof.
\end{proof}
\end{lem}

Next recall that the tail distribution function of a Weibull distributed random variable is an element of $\textbf{R}_{-\infty}$.  Logarithmic functions provide a canonical examples of a slowly varying functions.  By explicit computation, it can be shown that the $S$ and $S^{-1}$ associated with the Weibull distribution are members of $\textbf{R}_\infty$ and  $\textbf{R}_0$ respectively.  This holds more generally and we develop this now so that we may use these facts in the analysis that follows.  For this, for an unbounded Borel measurable function $f:\R_+\to(0,\infty)$ that is nondecreasing and $y\in\R_+$, we define 
\begin{equation}\label{f inverse}
f^{-1}(y)=\inf\{x\in\R_+:f(x)>y\}.
\end{equation}
Observe that $f^{-1}\in\D([0,\infty),\R_+)$ is nondecreasing and unbounded. Moreover, $f(f^{-1}(y))\geq y$ for all $y\in\R_+$ and $f^{-1}(f(x))\geq x$ for all $x\in\R_+$. However, if $f$ is continuous, $f( f^{-1}(y))= y$ for all $y\in[f(0),\infty)$ and if $f$ is strictly increasing, then $f^{-1}(f(x))=x$ for all $x\in\R_+$.  These facts together with the fact that $F$ is continuous, $R\subset(\lambda,\infty)$ and \eqref{cr is s inverse r} show that the following lemma holds.
\begin{lem}\label{S property a} 
For all $r\in R$, $c^r>0$ and
\begin{equation}
S(c^r)=S(S^{-1}(r))=r.  \label{Scr is r}
\end{equation}
\end{lem}

Theorem 2.4.7 from \cite{Bingham} is restated next for the reader's convenience.
\begin{prop}\label{f inverse slowly varying}
If $f\in\textbf{R}_\infty\cup\textbf{R}_{-\infty}$ is nondecreasing and unbounded, then $f^{-1}\in\textbf{R}_0$.
\end{prop}

In the proof of Lemma \ref{S property} below, we show that $S\in\textbf{R}_\infty$ in order to conclude that $S^{-1}\in\textbf{R}_0$.  We delay this because will show that a slightly stronger property holds so that we may leverage certain uniform convergence results later on.  We briefly summarize these here.  What follows is taken from Sections 2.0, 2.2, and 2.4 in \cite{Bingham}.
\begin{defn}\label{Karamata index}
Let $f:\R_+\to (0,\infty)$ be Borel measurable.  Then $c(f)$ (resp.\ $d(f)$) denotes the {\it upper (resp.\ lower) Karamata index} of $f$, where
\begin{eqnarray*}
c(f)&:=&\inf\left\{c\in\R:\forall T>1,\exists X>0\text{ s.t.\ }\frac{f(t x)}{f(x)}\leq (1+o(1))t^c\,\,\forall x\geq X\ \&\ t\in[1,T]\right\},\\
d(f)&:=&\sup\left\{d\in\R:\forall T>1,\exists X>0\text{ s.t.\ }\frac{f(t x)}{f(x)}\geq (1+o(1))t^d\,\,\forall x\geq X\ \&\ t\in[1,T]\right\}.
\end{eqnarray*}
\end{defn}
Above $o(1)\to0$ as $x\to\infty$.  The class of Borel measurable $f$ such that its Karamata indices $c(f)$ and $d(f)$ are both $+\infty$ (resp.\ $-\infty$) is denoted by $\textbf{KR}_\infty$ (resp $\textbf{KR}_{-\infty})$. By (2.4.3) in \cite{Bingham} we have the following proposition.
\begin{prop}\label{BG inclusion}
$\textbf{KR}_\infty\subseteq \textbf{R}_\infty$.
\end{prop}

Also, part (iv) of Proposition 2.4.4 in \cite{Bingham} implies the following.
\begin{prop}\label{Proposition 2.4.4. in Bingham}
If $f\in \textbf{R}_\infty$ and $f$ is nondecreasing, then $f\in\textbf{KR}_\infty$.
\end{prop}

The next result is a restatement of Proposition 2.2.1 in \cite{Bingham}.
\begin{prop}\label{BG prop 2.2.1}
If $f:\R_+\to(0,\infty)$ is Borel measurable and $-\infty<d(f)$, then for every $d\in(-\infty,d(f))$ and $C\in(0,1)$, there exists a positive constant $X=X(d,C)$ such that 
\begin{equation*}
    f(y)/f(x)\geq C(y/x)^d\qquad\text{for all }y\geq x\geq X.
\end{equation*}
\end{prop}

\begin{lem}\label{S property}
The function $S\in\textbf{KR}_\infty$ and, in particular, $d(S)=\infty$. Hence, $S^{-1}\in\textbf{R}_0$ and
\begin{equation}
    \lim_{r\to\infty}\frac{c^r}{r}=0.\label{cr over r to 0}
\end{equation}
\begin{proof}
First, we shall prove that function $S\in\textbf{R}_\infty$, which follows from the fact that $\overline{F}\in\textbf{R}_{-\infty}$.
To see this, let $G(x) = \int_x^\infty\overline{F}(y)\,dy$ and $H(x) = \E[v\id_{[v>x]}]$ for all $x\in\R_+$. For any $t>1$, by L'Hopital rule,
\[
\lim_{x\to\infty}\frac{G(t x)}{G(x)} = \lim_{x\to\infty}\frac{-t\overline{F}(t x)}{-\overline{F}(x)}=0.
\]
Thus, $G\in\textbf{R}_{-\infty}$.
In addition, by integration by parts, $H(x) = x\overline{F}(x)+G(x)$ for all $x\in \R_+$. Then,
for all $x\in\R_+$ and $t>1$,
\[
0\leq \frac{H(t x)}{H(x)}=\frac{x\overline{F}(x)}{H(x)}\cdot\frac{t\overline{F}(t x)}{\overline{F}(x)}+\frac{G(x)}{H(x)}\cdot\frac{G(t x)}{G(x)}\leq \frac{t\overline{F}(t x)}{\overline{F}(x)}+\frac{G(t x)}{G(x)}.
\]
This together with $\overline{F}\in\textbf{R}_{-\infty}$ and $G\in\textbf{R}_{-\infty}$ implies that
$H\in\textbf{R}_{-\infty}$.
Then, since $S(x)=1/H(x)$ for all $x\in\R_+$, $S\in\textbf{R}_\infty$. Since $S$ is nondecreasing, Proposition \ref{Proposition 2.4.4. in Bingham} implies that $S\in\textbf{KR}_\infty$. Thus, $d(S)=\infty$ by the definition of $\textbf{KR}_\infty$. Also, $S^{-1}\in\textbf{R}_0$ due to Proposition \ref{f inverse slowly varying}.  Finally, since $S^{-1}\in\textbf{R}_0$, \eqref{cr over r to 0} holds due to Proposition \ref{subpolynomial growth of slowly varying functions}.
\end{proof}
\end{lem}

%-----------Section 5----------------
\section{Proof of Theorem  \ref{W convergence}}\label{scn:ProofofWconvergence}
In order to prove Theorem \ref{W convergence}, we will study $a$-truncated versions of $\widehat{V}^r$ and $\widehat{X}^r$.  For this, we define the following.
For $r\in R$ and $a\in\R_+$, we let $\rho_a^r=\lambda^r\E[v\id_{[v\leq a]}]$, and so $\rho^r-\rho_a^r=\lambda^r\E[v\id_{[v>a]}] = \lambda^r/S(a)$.
\noindent Also, for $r\in R$, $a\in\R_+$, and $t\geq0$, let 
\begin{alignat}{3} 
V_a^r(t)&=\sum_{i=1}^{E^r(t)}v_i\id_{[v_i\leq a]}, 
&\widehat{V}_a^r(t)&=\frac{1}{r}\left[\sum_{i=1}^{E^r(r^2t)}v_i\id_{[v_i\leq a]}-\rho_a^rr^2t\right],\nonumber\\
X_a^r(t)&=W_a^r(0)+V_a^r(t)-t,\qquad
&\widehat{X}_a^r(t)&=\frac{X^r_a(r^2t)}{r},\nonumber\\
Y_a^r(t)& = \Gamma [X_a^r](t),
&\widehat{Y}_a^r(t)& = \Gamma[\widehat{X}^r_a](t),\label {YTrunc}\\
\widetilde{V}_a^r(t)&=\widehat{V}_{ac^r}^r(t),
&\widetilde{X}_a^r(t)&=\widehat{X}_{ac^r}^r(t),
\qquad\widetilde{Y}_a^r(t) = \Gamma[\widetilde{X}_a^r](t).\label{Gamma Trunc}
\end{alignat}
Then, for $r\in R$ and $a\in\R_+$, $V_a^r$, $X_a^r$ and $Y_a^r$ ignore tasks that are of size more than $a$ upon their arrival. This is different from the $a$-cutoff processes \eqref{cutoff} that include tasks that may have been larger than size $a$ upon their arrival, but are currently of size less or equal $a$.  Also, for $r\in R$ and $a\in\R_+$, $\widehat{V}_a^r, \widehat{X}_a^r$ and $\widehat{Y}_a^r$ are diffusion scaled process and $\widetilde{V}_a^r$, $\widetilde{X}_a^r$ and $\widetilde{Y}_a^r$ are related to $\widehat{V}_a^r,\widehat{X}_a^r$ and $\widehat{Y}_a^r$ in the same way that $\widetilde{W}_a^r$ is related to $\widehat{W}_a^r$, so that they correspond to distribution dependent scaled processes.
 
Next we restate (5.3) in Proposition 10 of \cite{heavy tails} in Proposition \ref{W bounded by Y} below. This provides a comparison between the $a$-truncated process $Y_a^r(\cdot)$ and $a$-cutoff process $W_a^r(\cdot)$ for any $a\in\R_+$ and $r\in R$, the proof of which is valid for any processing time distribution. Hence, we refer the reader to \cite{heavy tails} for the proof. For this, we make note that for $r\in R$ and $a\in \R_+$, $\widetilde{Y}_a^r$ and $\widetilde{W}_a^r$ here correspond to $Y_a^r$ and $W_a^r$ in \cite{heavy tails} respectively. We intentionally retain the $\sim$ here to reflect the scaling.
\begin{prop}\label{W bounded by Y} 
For all $r\in R$, $a\in\R_+$ and $t\geq0$,
$Y_a^r(t)\leq W_a^r(t)\leq
Y_a^r(t)+a$.
In particular, for all $r\in R$, $a\in \R_+$ and $t\geq0$, $\widehat{Y}_a^r(t) \leq \widehat{W}_a^r(t)\leq \widehat{Y}_a^r(t) + a/r$ and \[\widetilde{Y}_a^r(t)\leq \widetilde{W}_a^r(t)\leq \widetilde{Y}_a^r(t)+\frac{ac^r}{r}.
\]
\end{prop}
\indent Upon combining the above with \eqref{cr over r to 0} and \eqref{Gamma Trunc}, ones sees that a natural first step is to analyze $(\widetilde{X}^r_a(\cdot),a\in\R_+)$ as $r\to\infty$. This is the same overall strategy taken to prove Theorem 1 in \cite{heavy tails}.  However, the drift term behaves differently here because $S\in\textbf{R}_\infty$ (see Lemma \ref{S property}) rather than $\textbf{R}_p$ as it is in \cite{heavy tails}, where $p$ is as in \eqref{ht}.  Due to this, different techniques must be used for $a\in[0,1)$.
To begin, note that by definition, for all $r\in R$, $a\in\R_+$ and $t\geq 0$,
\begin{equation}\label{Xtilde}
    \widetilde{X}_a^r(t) =  \widetilde{W}_a^r(0)+\widetilde{V}_a^r(t)+r(\rho^r_{ac^r}-1)t.
\end{equation}
The next lemma gives a representation for the drift terms of the truncated netput processes that is used crucially throughout the analysis.
\begin{lem}\label{drift} For all $r\in R$ and $a\in\R_+$,
\begin{equation}\label{eq:drift}
    r(\rho_{ac^r}^r-1)=-\frac{\lambda^rS(c^r)}{S(ac^r)}+r(\rho^r-1).
\end{equation}
\begin{proof} For $r\in R$ and $a\in\R_+$,
\[
    r(\rho_{ac^r}^r-1)=r(\rho_{ac^r}^r-\rho^r)+r(\rho^r-1)=-r\lambda^r\E\left[v\id_{[v>ac^r]}\right]+r(\rho^r-1).
\]
Combining these the above with \eqref{S} and  \eqref{Scr is r} completes the proof.
\end{proof}
\end{lem}

Next, we establish a convergence result for the drift terms \eqref{eq:drift}.
\begin{lem}\label{rt rho minus 1}
\[
\lim_{r\to\infty}r(\rho_{ac^r}^r-1)=\begin{cases}
-\infty,&\qquad\text{if }0\leq a<1,\\
\kappa-\lambda,&\qquad\text{if }a=1,\\
\kappa ,&\qquad\text{if }a>1.
\end{cases}
\]
\begin{proof}
The result follows from \eqref{eq:drift}, \eqref{lambda}, $S\in\textbf{R}_\infty$ (see Lemma \ref{S property}),  and \eqref{r rho convergence}.
\end{proof}
\end{lem}

The limiting behavior of the drift terms described above is fundamentally different here in this light tailed case where \eqref{lt} holds than that in \cite{heavy tails}.  In \cite{heavy tails} where \eqref{ht} holds, the limiting drift terms are finite for all $a>0$ and vary continuously with $a$.  

Having established the limiting behavior of the drift terms, we move toward analyzing the remaining terms in \eqref{Xtilde}.
%----------------------------------
\begin{lem}\label{Joint convergence of V}
For any $k\in\N$ and $0< a_1<\dots< a_k<\infty$, as $r\to\infty$,
\[
(\widetilde{V}_{a_1}^r(\cdot),\dots,\widetilde{V}_{a_k}^r(\cdot),\widetilde{V}^r(\cdot))\Rightarrow (\sigma B(\cdot),\dots,\sigma B(\cdot),\sigma B(\cdot)).
\]

\begin{proof}
By Proposition \ref{functional clt convergence}, $\widetilde{V}^r(\cdot)\Rightarrow\sigma B(\cdot)$ as $r\to\infty$. By the convergence together theorem (see \cite[Theorem 11.4.7]{Whitt}), it suffices to show that for any $a\in(0,\infty)$, we have $\widetilde{V}^r(\cdot)-\widetilde{V}_{a}^r(\cdot)\Rightarrow0(\cdot)$ as $r\to\infty$. 
Fix $a\in (0,\infty)$, by the definitions of $\widetilde{V}^r(\cdot)$ and $\widetilde{V}_a^r(\cdot)$, for each $r\in R$ and $t\geq0$,
\[
\widetilde{V}^r(t)-\widetilde{V}_{a}^r(t)=\frac{1}{r}\sum_{i=1}^{E^r(r^2t)}v_i\id_{[ac^r<v_i]}-r\lambda^rt\E[v\id_{[ac^r<v]}].
\]
For $r\in R$ and $i\in\N$, let $\xi^r_i=v_i\id_{\{ac^r< v_i\}}$. By \eqref{cr over r to 0}, $\lim_{r\to\infty}ac^r/r=0$. Due to this and $\E[v^2]<\infty$, we have
$\lim_{r\to\infty}\E[(\xi_1^r)^2:(\xi_1^r)^2>r]=\lim_{r\to\infty}\E[v^2:v^2>r]=0$.
Thus, \eqref{1*} holds. For each $r\in R$, let $\hat\zeta^r(\cdot)$ be as in Proposition \ref{prop:FCLT} for this choice of $\{\xi_i^r\}_{i=1}^\infty$. Then by Proposition \ref{prop:FCLT}, $\hat\zeta^r(\cdot)\Rightarrow B(\cdot)$ as $r\to\infty$. For each $r\in R$, as in the derivation of \eqref{2**}, we have
\[
\widetilde{V}^r(\cdot)-\widetilde{V}_{a}^r(\cdot) = \sqrt{\var(\xi_1^r)}\hat\zeta^r(\overline{E}^r(\cdot))+\E[\xi_1^r]\widehat{E}^r(\cdot).
\]
Thus, as in the proof of Proposition \ref{functional clt convergence}, $\left(\hat\zeta^r(\overline{E}^r(\cdot)),\widehat{E}^r(\cdot)\right)\Rightarrow\left(B(\lambda(\cdot)),E^*(\cdot)\right)$ as $r\to\infty$ with $B(\lambda(\cdot))$ and $E^*(\cdot)$ independent of one another. Since $\E[v^2]<\infty$ and $\lim_{r\to\infty}c^r=\infty$,
$\lim_{r\to\infty}\E[(\xi_1^r)^2]=\lim_{r\to\infty}\E[v^2:v>ac^r]=0$. Thus, $\lim_{r\to\infty}\E[\xi_1^r]=0$ and $\lim_{r\to\infty}\Var[\xi_1^r]=0$.
Therefore, we have 
$
\widetilde{V}^r(\cdot)-\widetilde{V}_{a}^r(\cdot)\Rightarrow0(\cdot)$ as $r\to\infty$, as desired.
\end{proof}
\end{lem}

\begin{cor}\label{W tilde and V tilde joint convergence}
    For any $k\in\N$ and $0<a_1<a_2<\cdots<a_k<\infty$, as $r\to\infty$, 
    \begin{align*}
&\left(\widetilde{W}_{a_1}^r(0), \dots, \widetilde{W}_{a_k}^r(0),\widetilde{W}^r(0),\widetilde{V}_{a_1}^r(\cdot),\dots,\widetilde{V}_{a_k}^r(\cdot),\widetilde{V}^r(\cdot)\right)\\
&\qquad\Rightarrow\left(W_{a_1}^*(0),\dots,W_{a_k}^*(0),W^*(0),\sigma B(\cdot),\dots,\sigma B(\cdot),\sigma B(\cdot)\right)        .
    \end{align*}
\begin{proof}
    This follows from assumption \eqref{W sub joint convergence}, Lemma \ref{Joint convergence of V}, and the mutual independence of $((W_a^*(0),a\in\R_+),W^*(0))$ and $B(\cdot)$.
\end{proof}
\end{cor}

Due to Corollary \ref{W tilde and V tilde joint convergence} and Lemma \ref{rt rho minus 1}, $\lim_{r\to\infty}\P(\|\widetilde{X}_a^r\|_T\leq -M)=1$ for all $M,T>0$ and $a\in[0,1)$, whereas $\widetilde{X}_a^r(\cdot)\Rightarrow X_a^*(\cdot)$ as $r\to\infty$ for $a\geq1$. Thus, the analysis naturally separates into two cases: $a_1\geq 1$ and $a_k<1$.  These two cases correspond to Lemmas \ref{Lemma A} and \ref{Lemma B} below respectively. 
These results are leveraged at the end to proof of Theorem \ref{W convergence}. We begin with the case $a_1\geq 1$ since that follows quickly.

%------------------a_1>=1---------

\begin{lem}\label{joint X convergence}
For any $k\in\N$ and $1\leq a_1< a_2<\cdots< a_k<\infty$, as $r\to \infty$, we have 
$$
\big(\widetilde{X}_{a_1}^r(\cdot),\dots,\widetilde{X}_{a_k}^r(\cdot),\widetilde{X}^r(\cdot)\big)\Rightarrow\big(X^*_{a_1}(\cdot),\dots,X^*_{a_k}(\cdot),X^*(\cdot)\big).
$$
\begin{proof}
The result follows from \eqref{Xtilde}, Corollary \ref{W tilde and V tilde joint convergence}, Lemma \ref{rt rho minus 1}, and the convergence together theorem.
\end{proof}
\end{lem}

\begin{lem}\label{Lemma A}
Suppose the conditions of Theorem \ref{W convergence} hold and $a_1\geq1$. Then the conclusion of Theorem \ref{W convergence} holds.
\begin{proof}
Let $T>0$ and $a\geq1$. By Proposition \ref{W bounded by Y}, for $r\in R$,
$0\leq\|\widetilde{W}_a^r-\widetilde{Y}_a^r\|_T\leq ac^r/r$.
This together with  \eqref{cr over r to 0}, implies that
\begin{equation}
    \widetilde{W}^r_a-\widetilde{Y}_a^r\overset{p}{\to}0,\quad\text{as }r\to\infty.\label{Wa converge in probability}
\end{equation} 
By the almost sure continuity of $X_a^*$, the fact that the continuity points of $\Gamma$ contain the continuous functions, the continuous mapping theorem and Lemma \ref{joint X convergence}, we have
\begin{equation}
\widetilde{Y}_a^r(\cdot)=\Gamma[\widetilde{X}_a^r](\cdot)\Rightarrow\Gamma[X_a^*](\cdot)=W_a^*(\cdot),\quad\text{as }r\to\infty. \label{Gamma X}
\end{equation}
Thus, the joint convergence in Theorem \ref{W convergence} for $a_1\geq 1$ follows from \eqref{Gamma X}, Lemma \ref{joint X convergence}, \eqref{Wa converge in probability}, and the convergence together theorem.
\end{proof}
\end{lem}

\begin{cor}\label{cor:Tailtozero} Suppose $a>1$ and $T>0$. As $r\to\infty$,\\
$\| \langle \chi\id_{(a,\infty)},\widetilde{\Q}^r(\cdot)\rangle\|_T=\| \widetilde{W}^r-\widetilde{W}_a^r\|_T\overset{p}{\to} 0$ and $\|\langle \id_{(a,\infty)},\widetilde{\Q}^r(\cdot)\rangle\|_T=\|\widetilde{\QQ}^r-\widetilde{\QQ}_a^r\|_T\overset{p}{\to} 0$.
\begin{proof}
    Lemma \ref{Lemma A} and $W^*=W_a^*$ imply the first statement. For all $r\in R$ and $t\ge 0$, $\langle \id_{(a,\infty)},\widetilde{\Q}^r(t)\rangle\le\langle \chi\id_{(a,\infty)},\widetilde{\Q}^r(t)\rangle/a$.  Thus, the first statement implies the second.
\end{proof}
\end{cor}

%--------------a_k<1------------------
Next we consider $0\leq a_k<1$.  
For this, we adapt some ideas from Section 5.2 in \cite{Puha}. To begin, we make the following definitions. For $r\in R,a\in \R_+$, and $t\geq0$, let 
\begin{align*}
\tau^r(t,a)&:=\sup\{s\in [0,t]: \langle \chi_a,\Q^r(s)\rangle=0\} = \sup\{s\in[0,t]: W_a^r(s)=0\},\\
\theta^r(t,a) &:=t-\tau^r(t,a),
\end{align*}
where $\tau^r(t,a)=0$ by convention if $W_a^r(s)>0$ for all $s\in[0,t]$. With this convention, for each $a\in\R_+$ and $r\in R$, $\tau^r(t,a)$ is well defined and $0\le \tau^r(t,a)\leq t$ for each $t\geq0$, $a\in\R_+$ and $r\in R$. To see this, fix $a\in\R_+$ and $r\in R$. Observe that $W_a^r(\cdot)$ is a process that jumps upward when a task with a processing time of length $a$ or less arrives or the size of a task becomes $a$ (which can only happen if $W_a^r(\cdot)$ is zero immediately beforehand due to the SRPT prioritization). Otherwise, when $W_a^r(\cdot)$ is positive, it decreases continuously at rate 1. Thus, if $\tau^r(t,a)=0$, then $W_a^r(\tau^r(t,a))=W_a^r(0)$. If $\tau^r(t,a)=t$, then $W_a^r(\tau^r(t,a))=W_a^r(t)=0$ unless an arrival occurs at time $t$ in which case $W_a^r(\tau^r(t,a))=W_a^r(t)\le a$. When $0<\tau^r(t,a)<t$, we have $W_a^r\left(\tau^r(t,a)\right)\le a$ as $W_a^r(\cdot)$ must have jumped up at time $\tau^r(t,a)$.  Thus, for all $r\in R$, $a\in\R_+$ and $t\ge 0$,
\begin{equation}\label{WaBnd}
W_a^r(\tau^r(t,a))\le W_a^r(0)+a.
\end{equation}
We also define scaled versions of $\tau^r$ and $\theta^r$ as follows: 
for $r\in R$, $a\in\R_+$, and $t\geq0$, 
\begin{alignat*}{3}
    \widehat{\tau}^r(t,a)&=\frac{\tau^r(r^2t,a)}{r^2},\quad &\widehat{\theta}^r(t,a) &= t-\widehat{\tau}^r(t,a),\\    \widetilde{\tau}^r(t,a)&=\frac{\tau^r(r^2t,ac^r)}{r^2},\quad &\widetilde{\theta}^r(t,a) &= t-\widetilde{\tau}^r(t,a).
\end{alignat*}
Then for $r\in R,a\in\R_+$ and $t\geq0$, $\widetilde{\tau}^r(t,a)=\widehat{\tau}^r(t,ac^r)$ and $\widetilde{\theta}^r(t,a) = \widehat{\theta}^r(t,ac^r)$. 
 
\begin{lem}\label{a<1 convergence}
Suppose $a\in[0,1)$. As $r\to\infty$,
\begin{enumerate}
    \item\label{1 of a<1 convergence}
    $\widetilde{\theta}^r(\cdot,a)\Rightarrow0(\cdot)$,
    \item\label{2 of a<1 convergence}
    $(\rho^r-\rho_{ac^r}^r)r\widetilde{\theta}^r(\cdot,a)\Rightarrow 0(\cdot)$,
    \item\label{3 of a<1 convergence}
    $\widetilde{W}^r_a(\cdot)\Rightarrow 0(\cdot)$. 
\end{enumerate}
In particular, for all $T>0$, 
$\|\widetilde{W}_a^r\|_T\overset{p}{\to}0$ as  $r\to\infty$.
\end{lem}
Before proving Lemma \ref{a<1 convergence}, we make some remarks.  Versions of the first two statements in Lemma \ref{a<1 convergence} are developed in Lemmas 5.3 and 5.4 of \cite{Puha}, but the scaling there is a bit different.  A direct translation of \cite[Lemma 5.3]{Puha} into the notation here states that $rc^r\widetilde\theta^r(\cdot, a/c^r)=rc^r\widehat\theta^r(\cdot, a)\Rightarrow 0(\cdot)$ as $r\to\infty$ for all $a\in\R_+$.
Previously in \cite{Gromoll}, it was shown that
$r\widehat\theta^r(\cdot, a)\Rightarrow 0(\cdot)$ as $r\to\infty$ for all $a\in\R_+$.
Unfortunately, there is an error term that wasn't accounted for in the proof of \cite[Lemma 5.3]{Puha} ($\Omega_1^r$ is the emptyset otherwise) and it isn't clear how to incorporate this error term and keep the factor of $c^r$ in convergence to zero above. This convergence, with the factor of $c^r$ included, is used in the proof of \cite[Lemma 5.5]{Puha} in an attempt to show that what is referred to here as the $a/c^r$-cutoff distribution dependent queue length process, i.e., $\widetilde{\QQ}_{a/c^r}^r=c^r\widehat{\QQ}_{a}^r$, converges in distribution to zero as $r\to\infty$ for $a\in\R_+$.  Thus, this proof has a gap that seems challenging to resolve using the approach in \cite{Puha}.

In part 1 of Lemma \ref{a<1 convergence}, we state the first version of such a convergence that we can prove, which is equivalent to $\widehat\theta^r(\cdot, ac^r)\Rightarrow 0(\cdot)$ as $r\to\infty$ for $a\in[0,1)$.  In part 2 of Lemma \ref{a<1 convergence}, we see the multiplicative factor 
$(\rho^r-\rho_{ac^r}^r)r=\lambda^rS(c^r)/S(ac^r)$ in front of $\widetilde\theta^r(\cdot, a)=\widehat\theta^r(\cdot, ac^r)$ for $a\in[0,1)$
and $r\in R$, which diverges to positive infinity as $r\to\infty$ since $S\in\mathbf{R}_\infty$ and seems to give the natural rate of convergence to zero.

A direct translation of \cite[Lemma 5.4]{Puha}  into the notation here states the following: for all $\varepsilon>0$,
$(a_\varepsilon^r)^{-1}\widetilde\theta^r\left(\cdot, S^{-1}\left( a_\varepsilon^rr\right)/S^{-1}(r)\right))=(a_\varepsilon^r)^{-1}\widehat\theta^r\left(\cdot, S^{-1}\left(a_\varepsilon^r r\right)\right)\Rightarrow 0(\cdot)$ as $r\to\infty$, where $a_{\varepsilon}^r=(c^r)^{2+\varepsilon}$ for all $\varepsilon>0$ and $r\in R$.  Note that $\lim_{r\to\infty}(a_\varepsilon^r)^{-1}= 0$.  The condition (23) in \cite{Puha} is used there so that $S^{-1}\left(a_\varepsilon^r r\right)/S^{-1}(r)-1$ converges to zero at a certain rate as $r\to\infty$ for each $\varepsilon>0$.  While the proof  given for \cite[Lemma 5.4]{Puha} is valid, the approach taken here does not require this rate of convergence condition and so $a_\varepsilon^r$ is replaced with fixed $a\in[0,1)$ in part 2 of Lemma \ref{a<1 convergence}.  Ultimately, part 2 of Lemma \ref{a<1 convergence} is used in the proof below to show that the $a$-cutoff distribution dependent scaled workload process $\widetilde{W}_{a}^r$ converges in distribution to zero as $r\to\infty$ for $a\in[0,1)$ (see part 3 of Lemma \ref{a<1 convergence}).  Entirely different techniques are used in Section \ref{scn:ProofofQconvergence} to show that the $a$-cutoff distribution dependent scaled queue length process $\widetilde{\QQ}_{a}^r$ converges in distribution to zero as $r\to\infty$ for $a\in[0,1)$.
We provide a complete proof of Lemma \ref{a<1 convergence} next to address all of these issues.

\begin{proof}[Proof of Lemma \ref{a<1 convergence}]
Fix $T>0$.  Due to the work conserving property of SRPT and the fact that the $r$-th system is not empty during the time interval $(\tau^r(t,a),t]$ for all $t\geq 0$ and $r\in R$, we have the following equation: for each $r\in R$ and $t\geq0$, 
\begin{equation}
    W_a^r(t) = \langle \chi_a, \Q^r(t)\rangle =
    W_a^r\big(\tau^r(t,a)\big)+V_a^r(t)-V_a^r\big(\tau^r(t,a)\big) - \theta^r(t,a).\label{W by tau}
\end{equation}
After applying the distribution dependent scaling in \eqref{W by tau} and using \eqref{WaBnd}, we have the following: for each $r\in R$, and $t\geq0$,
\begin{align}
\widetilde{W}^r_a(t) \leq & \widetilde{W}^r_a(0)+\widetilde{V}_{a}^r(t)-\widetilde{V}_{a}^r\big(\widetilde{\tau}^r(t,a)\big)+r(\rho^r_{ac^r}-\rho^r)\widetilde{\theta}^r(t,a)\nonumber\\
&+r(\rho^r-1)\widetilde{\theta}^r(t,a)+\frac{ac^r}{r}.\label{eq:6}
\end{align}
For $r\in R$ and $t\geq 0$, $\widetilde{W}_a^r(t)\geq0$, \eqref{eq:6}, $\rho^r-\rho^r_{ac^r}\geq 0$ and $\widetilde{\theta}^r(t,a)\geq 0$ imply that
\begin{align}
    0&\leq r(\rho^r-\rho^r_{ac^r})\widetilde{\theta}^r(t,a)\nonumber\\
    &\leq \widetilde{W}^r_a(0)+\widetilde{V}^r_a(t)-\widetilde{V}^r_a\big(\widetilde{\tau}^r(t,a)\big)+r(\rho^r-1)\widetilde{\theta}^r(t,a)+\frac{ac^r}{r}.\label{eq:7}
\end{align}
By Corollary \ref{W tilde and V tilde joint convergence}
and $\|\widetilde{\theta}^r(\cdot,a)\|_T\leq T$, $\|\widetilde{W}_a^r(0)+\widetilde{V}^r_a(\cdot)-\widetilde{V}^r_a\big(\widetilde{\tau}^r(\cdot,a)\big)\|_T$ is stochastically bounded as $r\to\infty$, i.e., for all $\theta\in(0,1)$ there exists $M>0$ such that $\liminf\limits_{r\to\infty}\P\left(\|\widetilde{W}_a^r(0)+\widetilde{V}_a^r(\cdot)-\widetilde{V}^r_a\big(\widetilde{\tau}^r(\cdot,a)\big)\|_T\leq M\right)\geq 1-\theta$. Then by assumption \eqref{r rho convergence}, $\|\widetilde{\theta}^r(\cdot,a)\|_T\leq T$ and \eqref{cr over r to 0}, the supremum over $t\in[0,T]$ of right hand side of \eqref{eq:7} is stochastically bounded on $[0,T]$. Note that for all $r\in R$, 
\[
r(\rho^r-\rho^r_{ac^r}) =
\lambda^rS(S^{-1}(r))\E[v\id_{[v>ac^r]}]= \frac{\lambda^r S(c^r)}{S(ac^r)}.
\]
Multiplying both sides of \eqref{eq:7} by $\frac{S(ac^r)}{S(c^r)}$, noting that $\lim\limits_{r\to\infty}\frac{S(ac^r)}{S(c^r)} = 0$ since $a\in[0,1)$ and $S\in\textbf{R}_\infty$ (see Lemma \ref{S property}), and using \eqref{lambda} and the stochastic boundedness of the supremum over $t\in[0,T]$ of the right side of \eqref{eq:7} gives 
\begin{equation}
\|\widetilde{\theta}^r(\cdot,a) \|_T\overset{p}{\to} 0\text{ as $r\to\infty$}.\label{sup of tilde theta goes to 0 in probability}
\end{equation}
Since $T>0$ is arbitrary, \ref{1 of a<1 convergence} holds. Let $\iota(t) =t$ for all $t\geq0$.  From \ref{1 of a<1 convergence}, we conclude that $\iota(\cdot)-\widetilde{\tau}^r(\cdot,a)\Rightarrow 0(\cdot)$ as $r\to\infty$.
Then, by Lemma \ref{Joint convergence of V}, \eqref{sup of tilde theta goes to 0 in probability}, the convergence together theorem, 
assumption \eqref{r rho convergence} and \ref{1 of a<1 convergence}, as $r\to\infty$, 
\begin{align}
    \widetilde{V}_a^r(\cdot)- \widetilde{V}^r_a\big(\widetilde{\tau}^r(\cdot,a)\big)&\Rightarrow
    \sigma B(\cdot)-\sigma B(\iota(\cdot))=0(\cdot)\label{eq:8},\\
    r(\rho^r-1)\widetilde{\theta}^r(\cdot,a)&\Rightarrow 0(\cdot).\label{eq:9}
\end{align}
Thus, due to \eqref{cr over r to 0} and $a\in[0,1)$, \eqref{eq:8}, \eqref{eq:9}, \eqref{W tilde sub a r 0 convergence to 0 in prob} together with \eqref{eq:7} and the convergence together theorem imply \ref{2 of a<1 convergence}. Then, \eqref{eq:8}, \eqref{eq:9}, \ref{1 of a<1 convergence}, \ref{2 of a<1 convergence}, \eqref{eq:6}, assumption \eqref{r rho convergence}, \eqref{cr over r to 0} and the convergence together theorem imply \ref{3 of a<1 convergence}. The in particular statement follows from \ref{3 of a<1 convergence} and the fact that $0(\cdot)$ is deterministic, continuous and identically equal to zero.
\end{proof}

\begin{lem}\label{Lemma B}
Suppose that the conditions of Theorem \ref{W convergence} hold and $a_k<1$. Then the conclusion of Theorem \ref{W convergence} holds.
\begin{proof}
For any $r\in R$ and $t\geq0$, 
$
0\leq \widetilde{W}^r_{a_1}(t)\leq \dots\leq\widetilde{W}^r_{a_k}(t)<\infty$.
By Lemma \ref{a<1 convergence}, $\widetilde{W}^r_{a_k}(\cdot)\Rightarrow0(\cdot)$ as $r\to\infty$, and the result follows.
\end{proof}
\end{lem}

\begin{proof}[Proof of Theorem \ref{W convergence}]
If $a_k<1$ or $a_1\geq1$, the result follows from Lemma \ref{Lemma A} or Lemma \ref{Lemma B} respectively. Otherwise, there exists $1\leq m<k$, such that $0\leq a_m<1\leq a_{m+1}$. 
Then,
$\left\{\left(\widetilde W_{a_1}^r,\dots,\widetilde W_{a_m}^r\right)\right\}_{r\in R}$
and 
$\left\{\left(\widetilde W_{a_{m+1}}^r,\dots,\widetilde W_{a_k}^r,\widetilde W^r\right)\right\}_{r\in R}$
converge jointly as $r\to\infty$ due to Lemmas \ref{Lemma A} and \ref{Lemma B} and the deterministic nature of the limit in Lemma \ref{Lemma B}.
\end{proof}

%-------------------Section 6-------------------------
\section{Proof of Theorem \ref{mv convergence}}\label{scn:ProofofQconvergence}
Thus far, we have shown that the queue mass associated with tasks of size greater than $a$ for any $a>1$ tends to zero under the distribution dependent scaling (see Corollary \ref{cor:Tailtozero}). In order to complete the proof, there are two main steps that remain.  First, in Section \ref{a<1 queue mass to zero}, we prove that the queue mass associated with tasks of size less or equal $a$ for any $a<1$ tends to zero under the distribution dependent scaling (see Theorem \ref{In between 0 and 1}). The verification of this is substantial and involves adapting the arguments in \cite{heavy tails} to the present light tailed setting. Second, in Section \ref{section 5.3}, we show that the total queue mass tends to $W^*$ under the distribution dependent scaling (see Lemma \ref{Lemma D}). These results are leveraged in the balance of Section \ref{section 5.3} to prove Theorem \ref{mv convergence}. 

%---------------Section 6.1-----------------------
\subsection{Convergence to Zero of the Queue Mass Below 1}\label{a<1 queue mass to zero}
In this section, we prove the following theorem.
\begin{thm}\label{In between 0 and 1}
Suppose $T>0$ and $a\in[0,1)$.  Then
$\|\widetilde{\QQ}^r_a\|_T\overset{p}{\to}0$ as $r\to\infty$.
\end{thm}
\noindent The result in Theorem \ref{In between 0 and 1} holds trivially if $a=0$. Let $a\in(0,1)$. In order to prove Theorem \ref{In between 0 and 1} for such an $a$, we must show that for any $\alpha\in(0,1)$
\begin{equation}
    \lim_{r\to\infty}\P\left(\|\widetilde{\QQ}^r_a\|_T>\alpha\right)=0.\label{QQ sub in between 0 and 1 greater theta is 0}
\end{equation}
Let $\Theta$ be the set of all $\vartheta:R\to\R_+$ such that $\lim_{r\to\infty}\vartheta(r)=0$, and let $\Psi$ be the set of all $\varphi:\R_+\to\R_+$ such that $\lim_{\delta\to0+}\varphi(\delta)=0$. We will show that there exists $\theta\in\Theta$, $\beta\in(0,1)$, $\vartheta\in\Theta$ and $\varphi\in\Psi$ such that for all sufficiently large $r\in R$ and $\delta\in[\theta(r),\beta]$,
\begin{equation}
\P\left(\|\widetilde{\QQ}^r_\delta\|_T>\varphi(\delta)+\vartheta(r)+\frac{2c^r}{r}\right)<\varphi(\delta)+\vartheta(r). \label{step to show QQ sub in between 0 and 1 greater theta is 0}
\end{equation}
Indeed, if such objects exist, then given $\alpha\in(0,1)$ and $\epsilon\in(0,1)$, we can choose $r$ sufficiently large and $\beta$ sufficiently small that for all $\delta\in[\theta(r),\beta]$ we have $\varphi(\delta)+\vartheta(r)+2c^r/r\leq\alpha$
and $\varphi(\delta)\le \epsilon$
so that
\[
\P\left(\|\widetilde{\QQ}^r_\delta\|_T>\alpha\right)\leq \P\left(\|\widetilde{\QQ}_\delta^r\|_T>\varphi(\delta)+\vartheta(r)+\frac{2c^r}{r}\right)<\varphi(\delta)+\vartheta(r)\le\epsilon+\vartheta(r).
\]
Letting $r\to\infty$ above and using that $\epsilon\in(0,1)$ is arbitrary gives that \eqref{QQ sub in between 0 and 1 greater theta is 0} holds for $a\in(0,\beta]$. Once we have established \eqref{QQ sub in between 0 and 1 greater theta is 0} for positive, but sufficiently small $a$, we will use the result in Theorem \ref{W convergence}, which implies that for $a\in(0,1)$,
$\|\widetilde{W}^r_a\|_T\overset{p}{\to} 0$ as $r\to\infty$, to extend to all $a\in(0,1)$. See the proof of Theorem \ref{In between 0 and 1} at the end of Section \ref{Application of Lemma 5.2.6 to prove Theorem 5.2.1}.

The remainder of this section is organized as follows. In Section \ref{a-truncated SRPT Queues}, we recall the concept of an $a$-truncated SRPT queue from \cite{heavy tails} and state a comparison result (see Proposition \ref{Prop 10 in SRPT}). This comparison result allows one to transition the analysis to the queue length processes in $a$-truncated SRPT queues (see \eqref{Q tilde theta and Z tilde theta}), which reduces the proof of \eqref{step to show QQ sub in between 0 and 1 greater theta is 0} to showing that \eqref{Z tilde M/c^r} and \eqref{telescope of Z tilde on theta2} hold. The proof of \eqref{Z tilde M/c^r} follows as in \cite{heavy tails} and so the proof is omitted. Establishing \eqref{telescope of Z tilde on theta2} relies on a key lemma, Lemma \ref{J1}, which is stated and proved in Section \ref{compact containment}. While Lemma \ref{J1} is analogous to Lemma 18 in \cite{heavy tails}, one must handle the drift terms differently in this light tailed setting in order to provide a complete a proof. Moreover, we show the main steps needed to prove Lemma \ref{J1} as separate lemmas to highlight the major ideas therein.  In the process, we also address some details overlooked in \cite{heavy tails} pertaining to the asymptotic assumptions on the sequence of initial conditions and arrival times of the sequences of initial conditions and arrival times of the first exogenously arriving tasks. Finally, this is all collected in Section \ref{Application of Lemma 5.2.6 to prove Theorem 5.2.1} to prove \eqref{telescope of Z tilde on theta2} and in turn Theorem \ref{In between 0 and 1}.

%-------Section 6.1.1-----------
\subsubsection{$a$-Truncated SRPT Queues}\label{a-truncated SRPT Queues}
\noindent The proof of \eqref{step to show QQ sub in between 0 and 1 greater theta is 0} makes use of comparisons between and with certain truncated SRPT queues.  For this, we recall from \cite{heavy tails} the definition of the $r$-th $a$-truncated SRPT queue, where $r\in R$ and $a\in\R_+$. In short, this is an SRPT queue that ignores, i.e., does not serve, tasks in the $r$-th SRPT queue with a processing time that is greater than $a$.  We give the details for this in the next paragraph.

Fix $r\in R$ and $a\in \R_+$.
To specify the initial state of the $r$-th $a$-truncated SRPT queue, define 
$I^r = \{-\QQ^r(0)+1, \dots,0\}$ and 
$I_a^r = \{i\in I^r: v_i\leq a\}$.
Then $|I_a^r| = Q_a^r(0)$ and there are $\QQ_a^r(0)$ initial tasks in the $r$-th $a$-truncated SRPT queue at time zero. If $Q_a^r(0)>0$, set $i_a^r(0) = \max_{i\in I_a^r}\,i$ and, for $-\QQ_a^r(0)+1\leq i < 0$, set
$$
i_a^r(i) = \max\{k\in I_a^r:k< i_a^r(i+1)\}.
$$
 For $-\QQ_a^r(0)+1\leq i\leq 0$, the remaining processing time of the $i$th initial task in the $r$-th $a$-truncated SRPT queue at time zero is given by $v^{r,a}_i := v_{i_a^r(i)}$.  The arrival process $E_a^r$ for the $r$-th $a$-truncated SRPT queue satisfies $E_a^r(t) = \sum_{i=1}^{E^r(t)}\id_{[v_i\leq a]}$ for all $t\ge 0$.
Then, for $i\in\N$, $\tau_a^r(i)=\inf\{t\geq0:E_a^r(t)\geq i\}$ denotes the time at which the $i$-th task arrives to the $r$-th $a$-truncated SRPT queue.  For $i\in\N$, the processing time of the $i$-th task to arrive to the $r$-th $a$-truncated SRPT queue is given by
$v_i^{r,a}:=v_{E^r(\tau_a^r(i))}$. In particular, $v_i^{r,a}$ is equal to the processing time of the $E^r(\tau_a^r(i))$-th task to arrive to the $r$-th SRPT queue. The process evolves as an SRPT queue with these initial tasks and primitive inputs.
For $t\ge 0$ and $-\QQ_a^r(0)+1\leq i\leq E_a^r(t)$, we define $v_i^{r,a}(t)$ to be the remaining processing time (size) at time $t$ of the $i$-th task in the $r$-th $a$-truncated SRPT queue. For $t\geq0$, define
\begin{alignat*}{4}
&\ZZ_a^r(t) &&= \sum_{i=-\QQ_a^r(0)+1}^{E_a^r(t)}\delta^+_{v_i^{r,a}(t)},\quad&&
\widehat{\ZZ}_a^r(t) = \frac{1}{r}\ZZ_a^r(r^2t),\quad&&
\widetilde{\ZZ}_a^r(t) = \frac{c^r}{r}\ZZ_{ac^r}^r(r^2t),\\
&\ZZZ_a^r(t) &&= \langle\id, \ZZ_a^r(t)\rangle,\quad&&
\widehat{\ZZZ}_a^r(t) = \langle \id,\widehat{\ZZ}_a^r(t)\rangle,\quad&&
\widetilde{\ZZZ}_a^r(t)=\langle\id,\widetilde{\ZZ}_a^r(t)\rangle.
\end{alignat*}
We use the notation ${\mathcal Z}_a^r(\cdot)$, $\widehat{\mathcal{Z}}_a^r(\cdot)$, $\widetilde{\mathcal{Z}}_a^r(\cdot)$, $\ZZZ_a^r(\cdot)$, $\widehat{\ZZZ}_a^r(\cdot)$ and $\widetilde{\ZZZ}_a^r(\cdot)$ for these $a$-truncated processes to distinguish them from the $a$-cutoff processes ${\mathcal Q}_a^r(\cdot)$, $\widehat{\mathcal{Q}}_a^r(\cdot)$, $\widetilde{\mathcal{Q}}_a^r(\cdot)$, $\QQ_a^r(\cdot)$, $\widehat{\QQ}_a^r(\cdot)$ and $\widetilde{\QQ}_a^r(\cdot)$ respectively. Note that, for $t\ge 0$, ${\mathcal Z}_a^r(t)$ only consists of point masses associated with tasks that either were in the $r$-th SRPT queue at time zero
and of size less or equal $a$ at time zero
or arrived to the $r$-th SRPT queue by time $t$ and had a processing time of length less or equal $a$ upon their arrival. To the contrary, ${\mathcal Q}^r_a(t)$ includes point masses associated with tasks of size less or equal $a$ in the $r$-th SRPT queue at time $t$ that may have had a processing time of length greater than $a$ either at time zero or when the task arrived. Note that the workload process for $r$-th a-truncated SRPT is given by $Y_a^r$ as defined in \eqref{YTrunc}.

The following comparison result is a restatement of (5.4) in Proposition 10 in \cite{heavy tails}, which holds for any processing time distribution. We refer the reader to \cite{heavy tails} for a proof and also mention that $\widetilde{\ZZZ}_a^r$ and $\widetilde{\QQ}_a^r$ here are denoted as $\QQ_a^r$ and $\ZZZ_a^r$ respectively in \cite{heavy tails}. The choice here is deliberate so that the notation is consistent.
\begin{prop}\label{Prop 10 in SRPT}
For any $r\in R$ and $a\in\R_+$, $\ZZZ_a^r(0)=\QQ_a^r(0)$ and for all $t\geq0$,
\[
\ZZZ_a^r(t)\leq\QQ_a^r(t)\leq\ZZZ_a^r(t)+1.
\]
In particular, $\widehat{\ZZZ}_a^r(t)\leq \widehat{\QQ}_a^r(t)\leq \widehat{\ZZZ}_a^r(t)+1/r$ and $\widetilde{\ZZZ}_a^r(t)\leq \widetilde{\QQ}_a^r(t)\leq \widetilde{\ZZZ}_a^r(t)+c^r/r$ for all $r\in R$, $a\in\R_+$ and $t\geq 0$.
\end{prop}
As a consequence of Proposition \ref{Prop 10 in SRPT}, for all $\delta\in(0,1)$, $\varphi\in\Psi$, $\vartheta\in\Theta$ and $r\in R$, 
\begin{align}
\P\left(\|\widetilde{\QQ}^r_\delta\|_T>\varphi(\delta)+\vartheta(r)+\frac{2c^r}{r}\right)\leq\P\left(\|\widetilde{\ZZZ}_\delta^r\|_T>\varphi(\delta)+\vartheta(r)+\frac{c^r}{r}\right).\label{Q tilde theta and Z tilde theta}    
\end{align}
Therefore, in order to show \eqref{step to show QQ sub in between 0 and 1 greater theta is 0}, it suffices to show that there exists $\theta\in\Theta$, $\beta\in(0,1)$, $\vartheta\in\Theta$ and $\varphi\in\Psi$ such that for all sufficiently large $r\in R$ and $\delta\in[\theta(r),\beta]$,
\begin{equation}
    \P\left(\|\widetilde{\ZZZ}_\delta^r\|_T>\varphi(\delta)+\vartheta(r)+\frac{c^r}{r}\right)<\varphi(\delta)+\vartheta(r).\label{Z tilde greater theta smaller than theta}
\end{equation}
The remainder of Section \ref{a<1 queue mass to zero} is devoted to showing that \eqref{Z tilde greater theta smaller than theta} holds in the present light tailed setting. See Corollary \ref{in between 0 and 1} in Section \ref{Application of Lemma 5.2.6 to prove Theorem 5.2.1}.

%--------------Section 6.2.2----------------
\subsubsection{Overview of Remainder of the Proof}
\noindent Before beginning the details of proving \eqref{Z tilde greater theta smaller than theta}, we continue with a few additional remarks. For this, let $\delta\in(0,1)$, $M>0$ and $r\in R$ be such that $\delta c^r\geq M$. Then, for all $t\geq0$,
\[
\widetilde{\ZZZ}_\delta^r(t)=\widetilde{\ZZZ}_\delta ^r(t)-\widetilde{\ZZZ}_{M/c^r}^r(t)+\widetilde{\ZZZ}_{M/c^r}^r(t).
\]
There exists $K\in\Z_+$ such that $\delta c^r2^{-K-1}< M\leq \delta c^r2^{-K}$. Fix such a $K$. Then, for all $t\geq0$,
\begin{align}
    \widetilde{\ZZZ}^r_\delta(t)&=\widetilde{\ZZZ}_\delta^r(t)-\widetilde{\ZZZ}^r_{2^{-K-1}\delta}(t)+\left(\widetilde{\ZZZ}^r_{2^{-K-1}\delta}(t)-\widetilde{\ZZZ}^r_{M/c^r}(t)\right)+\widetilde{\ZZZ}^r_{M/c^r}(t)\label{Z delta r t}
\end{align}
The next proposition is a restatement of Lemma 11 in \cite{heavy tails}, which also holds for any processing time distribution. For this, recall that, for $r\in R$ and $a\in \R_+$, $\widetilde{\ZZZ}_a^r$ and $\widetilde{Y}_a^r$ here are respectively denoted as $\QQ_a^r$ and $Y_a^r$ in \cite{heavy tails}.

\begin{prop}\label{tilde QQ}
For all $r\in R$, $0< x\leq y\leq\infty$ and $t\geq0$,
\begin{equation}
0\leq \widetilde{\ZZZ}_y^r(t)-\widetilde{\ZZZ}_x^r(t)\leq\frac{c^r}{r}+x^{-1}\widetilde{Y}_y^r(t).\label{z tilde y minus z tilde x}
\end{equation}
\end{prop}

Thus, for the choice of $\delta,M,r$ and $K$ above, due to the first inequality in \eqref{z tilde y minus z tilde x} and \eqref{Z delta r t}, for all $t\geq0$, we have , $\widetilde{\ZZZ}^r_{\delta2^{-K-1}}(t)-\widetilde{\ZZZ}^r_{M/c^r}(t)\leq0$ and 
\begin{equation}
    \widetilde{\ZZZ}_\delta^r(t)\leq \widetilde{\ZZZ}_\delta^r(t) - \widetilde{\ZZZ}^r_{\delta2^{-K-1}}(t)+ \widetilde{\ZZZ}^r_{M/c^r}(t).\label{telescope of Z tilde}
\end{equation}
By \eqref{telescope of Z tilde} and a union of events bound, for the choice of $\delta,M,r$ and $K$ above, $\vartheta\in\Theta$ and $\varphi\in\Psi$, we have
\begin{align*}    \P\left(\|\widetilde{\ZZZ}_\delta^r\|_T>\varphi(\delta)+\vartheta(r)+\frac{c^r}{r}\right)&\leq\P\left(\|\widetilde{\ZZZ}_{M/c^r}^r\|_T>\vartheta(r)\right)\\
    +\P&\left(\|\widetilde{\ZZZ}_\delta^r - \widetilde{\ZZZ}^r_{\delta2^{-K-1}}\|_T>\varphi(\delta)+\frac{c^r}{r}\right).
\end{align*}
Hence, in order to show \eqref{Z tilde greater theta smaller than theta}, it suffices to show that there exists $M>0$, $\beta\in(0,1)$, $\varphi\in\Psi$, and $\vartheta\in \Theta$ such that for sufficiently large $r\in R$, $\delta\in[2M/c^r,\beta]$ and
$K\in\Z_+$ such that $\delta c^r2^{-K-1}< M\leq \delta c^r2^{-K}$,
\begin{align}
\P\left(\|\widetilde{\ZZZ}_{M/c^r}^r\|_T>\vartheta(r)\right)&<\vartheta(r),\label{Z tilde M/c^r}\\
    \P\left(\|\widetilde{\ZZZ}_\delta^r - \widetilde{\ZZZ}^r_{\delta2^{-K-1}}\|_T>\varphi(\delta)+\frac{c^r}{r}\right)&<\varphi(\delta).\label{telescope of Z tilde on theta2} 
\end{align}
The proof \eqref{Z tilde greater theta smaller than theta} proceeds by showing that \eqref{Z tilde M/c^r} and \eqref{telescope of Z tilde on theta2} hold.\\
\indent It was shown in Lemma 16 in \cite{heavy tails} that \eqref{Z tilde M/c^r} holds for any $M>0$. The proof of Lemma 16 given in \cite{heavy tails} does not rely on the heavy tailed processing times condition that is assumed to hold in \cite{heavy tails}. In particular due to assumption \eqref{q tilde initial 0}, which implies that assumption (2.19) in \cite{heavy tails} holds, it is valid in the present light tailed setting. We restate this here in the next proposition and refer the reader to \cite{heavy tails} for a proof. Recall that for all $r\in R$ and $a\in \R_+$, $\widetilde{\ZZZ}_a^r$ here is denoted as $\QQ_a^r$ in \cite{heavy tails} .

\begin{prop}\label{Lemma 16 in SRPT}
For any $M>0$ and $T>0$, there exists $\vartheta\in\Theta$ such that for all $r\in R$ 
\[
\P\left(\|\widetilde{\ZZZ}^r_{M/c^r}\|_T>\vartheta(r)\right)\leq \vartheta(r).
\]
\end{prop}

The main effort in what follows is to show that there exists $M>0$ such that a version of \eqref{telescope of Z tilde on theta2} holds. This is done in  Lemma \ref{Lemma New} in Section \ref{Application of Lemma 5.2.6 to prove Theorem 5.2.1} shortly before Theorem \ref{In between 0 and 1} is proved.

%---------------------Section 6.1.3-----------------------------------
\subsubsection{A Key Technical Lemma Used To Prove \eqref{telescope of Z tilde on theta2}}
\noindent To begin, note that for $r\in R$, $\delta>0$, $K\in\Z_+$ and $t\geq0 $, $\widetilde{\ZZZ}_\delta^r(t) - \widetilde{\ZZZ}_{\delta2^{-K-1}}(t) = \sum_{i=0}^K\left(\widetilde{\ZZZ}^r_{\delta2^{-i}}(t)-\widetilde{\ZZZ}_{\delta2^{-i-1}}^r(t)\right)$. On inspection, one might note that each term in the summation above is of the form $\widetilde{\ZZZ}_\delta^r(t)-\widetilde{\ZZZ}_{\delta/2}^r(t)$, where $r\in R,\delta>0$ and $t\geq0$. Therefore, a main step in proving \eqref{telescope of Z tilde on theta2} is to establish the next lemma. For this, we recall \eqref{boundedness of tilde W by a} and define finite, positive constants $C$, $B$, $D_0$ and $D$ as follows:
\begin{alignat}{3}
C&=\left(\E[v^2]+\frac{128\sigma_A^2}{99}\right)\vee1,&\quad B&=282C ,\label{definition of C and B}\\ 
D_0&=\sup_{a\in[M_0/c^r,1]}a^{-(1+\eta_0)}\E[\widetilde{W}_a^r(0)],&\quad D&=\frac{D_0}{2B}+9.\label{definition of D and D0}
\end{alignat}

\begin{lem}\label{Lemma 18 of paper}
Suppose $T>0$ and $\eta\in(0,\eta_0)$. There exist $r(\eta)\in R$, $M(\eta)\geq M_0$ and $\delta(\eta)\in(0,1)$ such that for all $r\geq r(\eta)$ and $\delta\in[M(\eta)/c^r,\delta(\eta)]$,
\begin{equation}
    \P\left(\|\widetilde{\ZZZ}_\delta^r-\widetilde{\ZZZ}_{\delta/2}^r\|_T>96B\delta^\eta\log(1/\delta)+\frac{c^r}{r}\right)\leq D\delta^{\eta_0-\eta}.\label{4.4.16}
\end{equation}
\end{lem}
The proof of Lemma \ref{Lemma 18 of paper} is delicate. It follows closely the structure of the proof of Lemma 18 in \cite{heavy tails}. The work here extends these techniques to the light tailed case. In addition, it highlights the major steps in the proof of Lemma \ref{Lemma 18 of paper} by presenting them as six lemmas and two corollaries (Lemmas \ref{a lot of assumptions}, \ref{lemma 1 in original paper}, \ref{new lemma for biased random walk}, \ref{SRPT 5.82} and \ref{Lemma 3 in SRPT} and Corollaries \ref{lemma 2 in original paper} and \ref{SRPT Lemma 18 5.92} below) specific to this light tailed setting before ultimately proving Lemma \ref{Lemma 18 of paper} at the end of this section.

We begin by noting that for $r\in R,\delta>0$ and $T>0$, the upper bound in Proposition \ref{tilde QQ} implies that 
\[
\|\widetilde{\ZZZ}_\delta^r-\widetilde{\ZZZ}^r_{\delta/2}\|_T\leq \frac{2}{\delta}\|\widetilde{Y}_\delta^r\|_T+\frac{c^r}{r}.
\]
Thus, in order to prove Lemma \ref{Lemma 18 of paper}, it suffices to show that for $r\in R$ and $\delta>0$ satisfying the conditions of Lemma \ref{Lemma 18 of paper},
\begin{equation}
    \P\left(\|\widetilde{Y}_\delta^r\|_T>48B\delta^{1+\eta}\log{(1/\delta)}\right)\leq D\delta^{\eta_0-\eta}.\label{probability of sup tilde Y greater 3DB bounded above by 34}
\end{equation}
The analysis that follows focuses on this. As mentioned above, we partition the major steps of this proof into several lemmas.

In order to prove Lemma \ref{Lemma 18 of paper}, we need to establish values of $M>0$, $r\in R$ and $\delta\in(0,1)$ such that a number of basic estimates hold. We do this in the next lemma. But first we define an additional constant $\epsilon_*\in(0,1]$ as follows: 
\begin{equation}
    \epsilon_*=\frac{4}{3\lambda(5+9(\E[v])^2+8\sigma_A^2)}\wedge1.\label{definition of epsilon sub *}
\end{equation}

\begin{lem}\label{a lot of assumptions}
Suppose $\eta\in(0,\eta_0)$. There exist $r_*(\eta)\in R$, $M_*(\eta)\geq M_0$, and $\delta_*(\eta)\in(0,1/e]$ such that for all $r\geq r_*(\eta)$ and $\delta\in[M_*(\eta)/c^r,\delta_*(\eta)]$ the following hold:
\begin{align}
    \frac{2^3-2}{2^{32\log(1/\delta)}-2}&\leq\delta^{16},\label{4.4.1}\\
    \frac{9}{2}&\leq48\log{(1/\delta)},\label{corollary of 4.4.1}\\
    \frac{2\lambda}{3}\leq\lambda^r&\leq\frac{8\lambda}{7},\label{4.4.2}\\
    \frac{\sigma_A^2}{2}\leq (\sigma_A^r)^2&\leq 2\sigma_A^2,\label{4.4.4}\\
    \E\left[\left(v\id_{[v\leq \delta c^r]}-\lambda^r\E\left[v\id_{[v\leq\delta c^r]}\right]T^r\right)^2\right]&\leq C,\label{4.4.6}\\
    c^r&<r^{\frac{1}{2(1+\eta)}},\label{4.4.7}\\
    \max\left\{\frac{3^22^{10}\lambda\sigma_A^2}{\epsilon_*},2^{14}\lambda^2\sigma_A^2,1\right\}&\leq r,\label{4.4.8}\\
    -\frac{\lambda^r(Sc^r)}{S(\delta c^r)}+r(\rho^r-1)&\leq -\frac{5\lambda }{12\delta^{1+\eta}},\label{liminf and inf swap}\\
    \E[(T_1^r)^2]&\leq C_1,\label{4.4.3}\\
\P\left(T_1^r>r^2B\lambda^{-1}\delta^{2(1+\eta)}\right)&<\frac{1}{282},\label{C B}\\
    \P\left(T_1^r>r^2\epsilon_*\delta^{2(1+\eta)}/3\right)&<\frac{1}{18},\label{condition involving epsilon *}\\
    \P\left(\delta^{-1-\eta}\widetilde{W}_\delta^r(0)>2B\right)&\leq\frac{D_0}{2B}\delta^{\eta_0-\eta},\label{tilde W bounded at initial}
\end{align}
where $C$, $B$, and $D_0$ are as in \eqref{definition of C and B} and \eqref{definition of D and D0} and $C_1$ is as in \eqref{u convergence}.  Moreover, for any $b_0>0$, there exists $\widetilde{r}(\eta,b_0)\geq r_*(\eta)$ such that $\widetilde{r}(\eta,b_0)$ is nondecreasing as $b_0$ increases and for any $b\geq b_0$, $r\geq \widetilde{r}(\eta,b_0)$ and $\delta\in[M_*(\eta)/c^r,\delta_*(\eta)]$,
\begin{equation}
    \P\left(E^r\left(\frac{3}{4}br^2\delta^{2(1+\eta)}\right)>\lfloor \lambda br^2\delta^{2(1+\eta)}\rfloor\right)\leq\frac{2^9\lambda\sigma_A^2}{br}.\label{4.4.9}
\end{equation}
\begin{proof} 
To show \eqref{4.4.1} and \eqref{corollary of 4.4.1}, we first show that there exists $\delta_1(\eta)\in (0,e^{-1/2})$ such that $\frac{6}{4^{-\log{\delta}}-2}\leq\delta$ for all $\delta\in(0,\delta_1(\eta)]$. For this, we let
\begin{align*}
    f(\delta):=\begin{cases}
    0,\quad&\text{if }\delta=0,\\
    \frac{6}{4^{-\log{\delta}}-2},\quad&\text{if }\delta\in(0,e^{-1/2}),
    \end{cases}\quad\text{and}\quad
    g(\delta):=\delta\text{ for }\delta\in[0,e^{-1/2}].
\end{align*}
Then $f$ and $g$ are continuously differentiable on $[0,e^{-1/2})$, $f(0) = g(0) = 0$, $f'(0)=0<1 = g'(0)$, and $\lim_{\delta\uparrow e^{-1/2}}g(\delta) = e^{-1/2}<\infty = \lim_{\delta\uparrow e^{-1/2}}f(\delta)$. Hence, there exists $\delta_1(\eta)\in(0,e^{-1/2})$ such that $f(\delta_1(\eta))=g(\delta_1(\eta))$ and when $\delta\in(0,\delta_1(\eta))$, we have $f(\delta)<g(\delta)$. 
% and $f(0)=g(0)=0$. Moreover, $\lim\limits_{\delta\to (e^{-1/2})-}f(\delta)=+\infty$ and $\lim\limits_{\delta\to e^{-1/2}}g(\delta)=e^{-1/2}$. Therefore, either $g(\delta)<f(\delta)$ for all $\delta\in(0,e^{-1/2})$ or there exists $\delta_1(\eta)$ as mentioned above. To verify the existence of such a $\delta_1(\eta)$, it suffices to show that $f'(0)$ exists, $f'(\delta)$ is continuous at $\delta=0$, and $f'(0)<g'(0)=1$. Note that when $\delta\in(0,e^{-1/2})$, we have $f'(\delta) = \frac{6\times4^{-\log{\delta}}}{(4^{-\log{\delta}}-2)^2}$. Hence, by L'Hopital's Rule,
% \begin{align*}
% f'(0+)=\lim_{\delta\to0+}\frac{6\times 4^{-\log{\delta}}}{(4^{-\log{\delta}}-2)^2}=\lim_{\delta\to0+}\frac{3}{4^{-\log{\delta}}-2}=0.
% \end{align*}
% Also, by L'Hopital's Rule,
% \begin{align*}
% f'(0)&=\lim_{\delta\to0+}\frac{f(\delta)-f(0)}{\delta}=\lim_{\delta\to0+}\frac{6}{\delta(4^{-\log{\delta}}-2)}=\lim_{\delta\to0+}\frac{6}{\delta((e^{\log{4}})^{-\log{\delta}}-2)}\\
% &=\lim_{\delta\to0+}\frac{6}{\delta(\delta^{-\log{4}}-2)}=\lim_{\delta\to0+}\frac{6\delta^{\log{4}}}{\delta-2\delta^{1+\log{4}}}=\lim_{\delta\to0+}\frac{6\log{4}\delta^{\log{4}-1}}{1-2(1+\log{4})\delta^{\log{4}}}=0.
% \end{align*}
Also, for $\delta\in[0,e^{-1/2}]$, we must have $0\leq\delta^{16}\leq\delta$ and so $\delta^{16}\in[0,e^{-1/2}]$. Therefore, by the above argument, when $\delta\in[0,\delta_1(\eta)]$, we must have $f(\delta^{16})\leq g(\delta^{16})$, i.e., the inequality \eqref{4.4.1} holds for all $\delta\in(0,\delta_1(\eta)]$. To see that \eqref{corollary of 4.4.1} holds for all $\delta\in (0,\delta_1(\eta)]$, fix $\delta\in(0,\delta_1(\eta)]$. By what was just proved, we have 
\begin{equation}
\frac{6}{4^{16\log{(1/\delta)}}-2}\leq\delta^{16}\leq1.\label{upper bound as delta D}
\end{equation}
Therefore, since $4^{16\log{(1/\delta)}}-2>0$, \eqref{upper bound as delta D} implies that $2^3\leq2^{32\log{(1/\delta)}}$ and so $\frac{9}{2}\leq48\log{(1/\delta)}$. Hence \eqref{corollary of 4.4.1} holds for all $\delta\in (0,\delta_1(\eta)]$. In what follows, we choose $\delta_*(\eta)\leq \delta_1(\eta)$ so that each of the remaining statements in Lemma \ref{a lot of assumptions} holds.

Note that by \eqref{lambda}, when $r$ is large enough, \eqref{4.4.2} holds. Similarly, due to  \eqref{sigma convergence}, \eqref{4.4.4} also holds when $r$ is large enough. 

To prove \eqref{4.4.6}, note that when $r\in R$ and $i\in\N\setminus\{1\}$, $T_i^r\overset{d}{=}T^r$, and by assumption $T^r$ is independent of $v$. Hence, for $r\in R$ and $\delta\in \R_+$,
\begin{align*}
&\E\left[\left(v\id_{[v\leq \delta c^r]}-\lambda^r\E\left[v\id_{[v\leq\delta c^r]}\right]T^r\right)^2\right]\\
&= \E\left[v^2\id_{[v\leq\delta c^r]}\right]+ \left(\lambda^r\E[v\id_{v\leq\delta c^r}]\right)^2\left(\E[(T^r)^2] - 2\E[T^r]^2\right)\leq \E[v^2]+\left(\frac{\lambda^r}{\lambda}\right)^2(\sigma_A^r)^2,
\end{align*}
where we used $(\lambda^r)^{-1} = \E[T^r]$ to obtain the equality and $(\lambda)^{-1}=\E[v]$ to obtain the inequality. This together with \eqref{4.4.2} and \eqref{4.4.4} imply that  \eqref{4.4.6} holds 
for all $r\in R$ sufficiently large and $\delta\in\R_+$.

To prove that \eqref{4.4.7} holds for $r$ sufficiently large, we use Lemma \ref{S property} (which says $S^{-1}\in\textbf{R}_0$), $c^r=S^{-1}(r)$ for all $r\in R$ and Proposition \ref{subpolynomial growth of slowly varying functions} with $\gamma=\frac{1}{2(1+\eta)}$.

Next \eqref{4.4.8} clearly holds for $r$ sufficiently large.

To prove \eqref{liminf and inf swap}, we claim the following statement holds: there exists $r_1(\eta)\in R$ and $M_*(\eta)\geq M_0$ such that for all $r\geq r_1(\eta)$ and $\delta\in\left[M_*(\eta)/c^r,1\right]$, we have
\begin{equation}
\frac{S(c^r)}{S(\delta c^r)}\geq \frac{3}{4}\left(\frac{1}{\delta}\right)^{1+\eta}.\label{BG proposition}
\end{equation}
To see this, observe that by Lemma \ref{S property} (which says that $S\in\textbf{R}_\infty$) and Proposition \ref{BG prop 2.2.1}, there exists $M_*(\eta)\geq M_0$ such that 
\begin{equation}
\frac{S(y)}{S(x)}\geq\frac{3}{4}\left(\frac{y}{x}\right)^{1+\eta}\qquad\text{ for all $y\geq x\geq M_*(\eta)$.} \label{Sy over Sx geq 3 over 4}
\end{equation}
Let $r_1(\eta)\in R$ be such that $c^r\geq M_*(\eta)$ for all $r\geq r_1(\eta)$. Then for $r\geq r_1(\eta)$ and $\delta\in[M^*(\eta)/c^r,1]$, we consider $y=c^r$ and $x=\delta c^r$ in \eqref{Sy over Sx geq 3 over 4} and we obtain \eqref{BG proposition}.
Due to \eqref{BG proposition} and \eqref{4.4.2}, there exists $r_2(\eta)\geq r_1(\eta)$ such that for any $r\geq r_2(\eta)$ and for any $\delta\in[M_*(\eta)/c^r,1]$,
\[
\frac{\lambda^rS(c^r)}{S(\delta c^r)}\geq\frac{2\lambda}{3}\cdot\frac{3}{4}\left(\frac{1}{\delta}\right)^{1+\eta} = \frac{\lambda}{2}\left(\frac{1}{\delta}\right)^{1+\eta}\qquad\hbox{and}\qquad
r(\rho^r-1)\leq |\kappa|+\eta.
\]
Thus, for any $r\geq r_2(\eta)$ and $\delta\in[M_*(\eta)/c^r,1]$,
\begin{equation}
-\lambda^r\frac{S(c^r)}{S(\delta c^r)}+r(\rho^r-1)\leq -\frac{\lambda }{2\delta^{1+\eta}}+|\kappa|+\eta=-\frac{5\lambda }{12\delta^{1+\eta}}-\frac{\lambda }{12\delta^{1+\eta}}+|\kappa|+\eta.\label{smaller equal than absolute kappa and eta}
\end{equation}
Choose $\delta_2(\eta)\in(0,1)$ such that
\[
\left(\delta_2(\eta)\right)^{1+\eta}\leq\frac{\lambda }{12(|\kappa|+\eta)}.
\]
Then, for all $r\geq r_2(\eta)$ and $\delta\in[M_*(\eta)/c^r,\delta_2(\eta)]$, we have
\begin{equation}
|\kappa|+\eta\leq\frac{\lambda }{12\left(\delta_2(\eta)\right)^{1+\eta}}\leq\frac{\lambda}{12\delta^{1+\eta}},\quad\text{and so}\quad -\frac{\lambda }{12\delta^{1+\eta}}+|\kappa|+\eta\leq0.\label{sum of absolute kappa and eta bounded by 0}
\end{equation}
Set $\delta_*(\eta) = \delta_1(\eta)\wedge \delta_2(\eta)\wedge 1/e$. Then, by \eqref{smaller equal than absolute kappa and eta} and \eqref{sum of absolute kappa and eta bounded by 0}, \eqref{liminf and inf swap} holds for all $r\geq r_2(\eta)$ and $\delta\in[M_*(\eta)/c^r,\delta_*(\eta)]$.

Condition \eqref{4.4.3} holds for all $r\in R$ due to assumption \eqref{u convergence}.

To verify \eqref{C B}, observe that by Markov's inequality and assumption \eqref{u convergence} (which implies that $\E[T_1^r]\leq 1+C_1$ for all $r\in R$), for $r\in R$ and $\delta\in [M_*(\eta)/c^r,\delta_*(\eta)]$,
\[\P\left(T_1^r\geq r^2 B\lambda^{-1}\delta^{2(1+\eta)}\right)\leq \frac{(c^r)^{2(1+\eta)}(1+C_1)}{r^2B\lambda^{-1}M_*(\eta)^{2(1+\eta)}}.
\]
Thus, condition \eqref{C B} holds for all $r\in R$ sufficiently larger than $r_2(\eta)$ and $\delta\in [M_*(\eta)/c^r,\delta_*(\eta)]$ as a consequence of the above, $S^{-1}\in\textbf{R}_0$ (see Lemma \ref{S property}), $c^r=S^{-1}(r)$, and Lemma \ref{T1 and r convergence}.

Condition \eqref{condition involving epsilon *} holds for all $r\in R$ sufficiently larger than $r_2(\eta)$ by a line of reasoning similar to that in the previous paragraph.

To see that $\eqref{tilde W bounded at initial}$ holds, note that by Markov's inequality, for all $r\in R$ and $\delta\in[M_*(\eta)/c^r,\delta_*(\eta)]$, 
\begin{align*}
\P\left(\widetilde{W}_\delta^r(0)>2B\delta^{1+\eta}\right)&\leq
\frac{\E[\widetilde{W}_\delta^r(0)]}{2B\delta^{1+\eta}}
=\frac{\delta^{-1-\eta_0}\E[\widetilde{W}_\delta^r(0)]\delta^{\eta_0-\eta}}{2B}\leq \frac{D_0}{2B}\delta^{\eta_0-\eta}.
\end{align*}

Let $r_*(\eta)\in R$ be such that \eqref{4.4.1}-\eqref{tilde W bounded at initial} hold for all $r\geq r_*(\eta)$ and $\delta\in[M_*(\eta)/c^r,\delta_*(\eta)]$.

To prove \eqref{4.4.9}, fix $b_0>0$ and $b\geq b_0$. For each $r\in R$ and $\delta>0$, by definition of $E^r$, 
\begin{equation}
    \P\left(E^r\left(\frac{3}{4}br^2\delta^{2(1+\eta)}\right)>\lfloor \lambda br^2\delta^{2(1+\eta)}\rfloor\right)\leq \P\left(\sum_{i=2}^{\lfloor \lambda br^2\delta^{2(1+\eta)}\rfloor}T_i^r<\frac{3}{4}br^2\delta^{2(1+\eta)}\right).\label{4.4.10}
\end{equation}
For each $r\in R$, $\E[T_i^r] = 1/\lambda^r$ for $i\ge 2$, and so, for each $r\in R$ and $\delta>0$, the right side of \eqref{4.4.10} becomes
\begin{equation}
    \P\left(\sum_{i=2}^{\lfloor \lambda br^2\delta^{2(1+\eta)}\rfloor}(T_i^r-\E[T_i^r])<\frac{3br^2\delta^{2(1+\eta)}}{4}-\frac{\lfloor \lambda br^2\delta^{2(1+\eta)}\rfloor-1}{\lambda^r}\right).\label{4.4.11}
\end{equation}
Let $\widetilde{r}(\eta,b_0)\in R$ be such that $\widetilde{r}(\eta, b_0)\geq r_*(\eta)$ and $\lambda b_0 r^2(M_*(\eta)/c^r)^{2(1+\eta)}>28$ for all $r\geq\widetilde{r}(\eta,b_0)$. Then, by \eqref{4.4.10}, \eqref{4.4.11} and \eqref{4.4.2} for $r\geq\widetilde{r}(\eta,b_0)$ and $\delta\in[M_*(\eta)/c^r,\delta_*(\eta)]$, we have 
\begin{align*}
    &\P\left(E^r\left(\frac{3}{4}br^2\delta^{2(1+\eta)}\right)>\lfloor \lambda br^2\delta^{2(1+\eta)}\rfloor\right)\\
   &\ \leq \P\left(\sum_{i=2}^{\lfloor \lambda br^2\delta^{2(1+\eta)}\rfloor}(T_i^r-\E[T_i^r])<\frac{3br^2\delta^{2(1+\eta)}}{4}-\frac{7}{8\lambda}\left(\lambda br^2\delta^{2(1+\eta)}-2\right)\right)\\
    &\ =\P\left(\sum_{i=2}^{\lfloor \lambda br^2\delta^{2(1+\eta)}\rfloor}(T_i^r-\E[T_i^r])<-\frac{br^2\delta^{2(1+\eta)}}{8}+\frac{14}{8\lambda}\right)\\
    &\ \leq \P\left(\sum_{i=2}^{\lfloor \lambda br^2\delta^{2(1+\eta)}\rfloor}(T_i^r-\E[T_i^r])<-\frac{br^2\delta^{2(1+\eta)}}{16}\right)\\%\label{4.4.12}
    &\ \leq\P\left(\left|\sum_{i=2}^{\lfloor \lambda br^2\delta^{2(1+\eta)}\rfloor}(T_i^r-\E[T_i^r])\right|>\frac{br^2\delta^{2(1+\eta)}}{16}\right).
\end{align*}
By Chebyshev's inequality and the mutual independence of the $T_i^r$'s for $i\geq2$, for $r\geq \widetilde{r}(\eta,b_0)$ and $\delta\in[M_*(\eta)/c^r,\delta_*(\eta)]$, we have 
\begin{align}
    \P\left(E^r\left(\frac{3}{4}br^2\delta^{2(1+\eta)}\right)>\lfloor \lambda br^2\delta^{2(1+\eta)}\rfloor\right)
    &\leq \frac{2^8\lambda br^2\delta^{2(1+\eta)}(\sigma_A^r)^2}{\left(br^2\delta^{2(1+\eta)}\right)^2}\nonumber= \frac{2^8\lambda(\sigma_A^r)^2}{br^2\delta^{2(1+\eta)}}.\nonumber
\end{align}
Then, by \eqref{4.4.4} and \eqref{4.4.7}, for $r\geq\widetilde{r}(\eta,b_0)$ and $\delta\in[M_*(\eta)/c^r,\delta_*(\eta)]$, we have $\delta\geq M_*(\eta)/c^r>1/c^r$ (since $M_*(\eta)\geq M_0>1$), and so
\[
\P\left(E^r\left(\frac{3}{4}br^2\delta^{2(1+\eta)}\right)>\lfloor \lambda br^2\delta^{2(1+\eta)}\rfloor\right)\leq \frac{2^9\lambda\sigma_A^2(c^r)^{2(1+\eta)}}{br^2}\leq \frac{2^9\lambda\sigma_A^2}{br}.
\]
\end{proof}
\end{lem}

Recall our objective given in \eqref{probability of sup tilde Y greater 3DB bounded above by 34} and also recall that $\widetilde{Y}_\delta^r(\cdot)=\Gamma[\widetilde{X}_\delta^r](\cdot)$ for all $r\in R$ and $\delta>0$. Thus, the analysis begins by developing properties of $\widetilde{X}_\delta^r(\cdot)$ for suitable $r\in R$ and $\delta>0$. The next lemma shows that when $r$ is sufficiently large, $\delta$ is in the range specified in Lemma \ref{a lot of assumptions}, and $\widetilde{X}_\delta^r(\cdot)$ starts from a positive value, then $\widetilde{X}_\delta^r(\cdot)$ is at least twice as likely to decrease by $2B\delta^{1+\eta}$ units as it is to increase by $B\delta^{1+\eta}$ units. This is due to the negative drift as quantified by \eqref{liminf and inf swap}.

\begin{lem}\label{lemma 1 in original paper}
Suppose $\eta\in(0,\eta_0)$.  Let $M_*(\eta)$, $ \delta_*(\eta)$ and $\widetilde{r}(\eta,\E[v])$ be as in Lemma \ref{a lot of assumptions}. For all $r \geq \widetilde{r}(\eta,\E[v])$, $\delta\in[M_*(\eta)/c^r,\delta_*(\eta)]$, $k\in\N$, $i\in\N\setminus\{1\}$, and $x\in\R_+$ such that $0\leq x\leq iB\delta^{1+\eta}$, the following inequality holds:
\begin{equation*}
    \P\left(\widetilde{X}_\delta^r(\cdot+u_{k}^r/r^2)\text{ crosses }(i+1)B\delta^{1+\eta}\text{ before }(i-2)B\delta^{1+\eta}\middle\vert \widetilde{X}_\delta^r(u_{k}^r/r^2)=x\right)\leq\frac{1}{3}.
\end{equation*}
\begin{proof}
Fix $r\geq\widetilde{r}(\eta,\E[v])$ and $\delta\in[M_*(\eta)/c^r,\delta_*(\eta)]$. The main step is to show that the following holds:
\begin{align}
\P\left(\widetilde{X}_\delta^r(\cdot+T_1^r/r^2)-\widetilde{X}_\delta^r(T_1^r/r^2)\text{ crosses $B\delta^{1+\eta}$ before $-2B\delta^{1+\eta}$}\right)\leq\frac{1}{3}\label{after 387}
\end{align}
The result in Lemma \ref{lemma 1 in original paper} follows from this. Indeed, by the strong Markov property and \eqref{after 387}, for $k\in\N,i\in\N\setminus\{1\}$ and $0\leq x\leq iB\delta^{1+\eta}$, we have that
\begin{align*}
&\P\left(\widetilde{X}_\delta^r(\cdot+u_k^r/r^2)\text{ crosses }(i+1)B\delta^{1+\eta}\text{ before }(i-2)B\delta^{1+\eta}\middle\vert \widetilde{X}_\delta^r(u_k^r/r^2)=x\right)\\
    \leq&\P\left(\widetilde{X}_\delta^r(\cdot+u_k^r/r^2)\text{ crosses }(i+1)B\delta^{1+\eta}\text{ before }(i-2)B\delta^{1+\eta}\middle\vert \widetilde{X}_\delta^r(u_k^r/r^2)=iB\delta^{1+\eta}\right)\\
    &\ =\P\left(\widetilde{X}_\delta^r(\cdot+u_k^r/r^2)-\widetilde{X}_\delta^r(u_k^r/r^2)\text{ crosses }B\delta^{1+\eta}\text{ before }-2B\delta^{1+\eta}\right)\\
    &\ = \P\left(\widetilde{X}_\delta^r(\cdot+T_1^r/r^2)-\widetilde{X}_\delta^r(T_1^r/r^2)\text{ crosses }B\delta^{1+\eta}\text{ before }-2B\delta^{1+\eta}\right).
\end{align*}
Thus, it suffices to show that \eqref{after 387} holds.

By \eqref{Xtilde}, \eqref{eq:drift}, and \eqref{liminf and inf swap}, for any $t\geq0$,
\[
\widetilde{X}_\delta^r(t+T_1^r/r^2)-\widetilde{X}_\delta^r(T_1^r/r^2)\leq \widetilde{V}_\delta^r(t+T_1^r/r^2)-\widetilde{V}_\delta^r(T_1^r/r^2)-\frac{5\lambda }{12\delta^{1+\eta}}t.
\]
If 
\begin{equation}
\sup_{t\in[0,6B\E[v]\delta^{2(1+\eta)}]}\widetilde{V}_\delta^r(t+T_1^r/r^2)-\widetilde{V}_\delta^r(T_1^r/r^2)\leq\frac{B\delta^{1+\eta}}{2},\label{6BEVdelta}
\end{equation}
then, for all $t\in[0,6B\E[v]\delta^{2(1+\eta)}]$,
\[
\widetilde{X}_\delta^r(t+T_1^r/r^2)-\widetilde{X}_\delta^r(T_1^r/r^2)\leq\frac{B\delta^{1+\eta}}{2}-\frac{5\lambda }{12\delta^{1+\eta}}t\leq \frac{B\delta^{1+\eta}}{2},
\]
and at time $t=6B\E[v]\delta^{2(1+\eta)}$, using $\E[v]=1/\lambda$,
\begin{align*}
\widetilde{X}_\delta^r(6B\E[v]\delta^{2(1+\eta)}+T_1^r/r^2)-\widetilde{X}_\delta^r(T_1^r/r^2)&
\leq\frac{B\delta^{1+\eta}}{2}-\frac{5B\delta^{1+\eta}}{2}
=-2B\delta^{1+\eta}.
\end{align*}
Therefore, on the event that \eqref{6BEVdelta} holds, $\widetilde{X}_\delta^r(\cdot+T_1^r/r^2)-\widetilde{X}_\delta^r(T_1^r/r^2)$ crosses $-2B\delta^{1+\eta}$ by time $6 B\E[v]\delta^{2(1+\eta)}$. Thus,
\begin{align}
    &\P\left(\widetilde{X}_\delta^r(\cdot+T_1^r/r^2)-\widetilde{X}_\delta^r(T_1^r/r^2)\text{ crosses $B\delta^{1+\eta}$ before $-2B\delta^{1+\eta}$}\right)\nonumber\\
    \leq&\P\left(\sup_{t\in[0,6B\E[v]\delta^{2(1+\eta)}]}\widetilde{V}_\delta^r(t+T_1^r/r^2)-\widetilde{V}_\delta^r(T_1^r/r^2)>B\delta^{1+\eta}/2 \right)\label{before ck}.
\end{align}
In what follows, we focus on upper bounding the right side of \eqref{before ck}.\\
\indent For $k\in\N$, let
\[
M_\delta^r(k)=\sum_{i=1}^k\left(v_i\id_{[v_i\leq \delta c^r]}-\rho_{\delta c^r}T_i^r\right)=r\widetilde{V}_\delta^r\left(\frac{u_k^r}{r^2}\right) =r\widetilde{V}_\delta^r\left(\sum_{i=1}^k \frac{T_i^r}{r^2}\right).
\]
In particular, for $k\in\N$, we have 
\[
M_\delta^r(k)- M_\delta^r(1) =r\widetilde{V}_\delta^r\left(\frac{u_k^r}{r^2}\right)- r\widetilde{V}_\delta^r\left(\frac{T_1^r}{r^2}\right).
\]
Since $\widetilde{V}_\delta^r(\cdot+T_1^r/r^2)-\widetilde{V}_\delta(T_1^r/r^2)$ can only cross $B\delta^{1+\eta}/2$ at jump times, we have 
\begin{align}
     &\P\left(\sup_{t\in[0,6B\E[v]\delta^{2(1+\eta)}]}\widetilde{V}_\delta^r(t+T_1^r/r^2)-\widetilde{V}_\delta^r(T_1^r/r^2)> B\delta^{1+\eta}/2 \right)\nonumber\\
     &\quad =\P\left(\sup_{2\leq k\leq E^r\left(6r^2B\E[v]\delta^{2(1+\eta)}+T_1^r\right)}M_\delta^r(k)-M_\delta^r(1)> Br\delta^{1+\eta}/2\right)\nonumber\\
    &\quad\leq\P\left(\sup_{2\leq k\leq \left\lfloor\frac{28}{3}Br^2\delta^{2(1+\eta)}\right\rfloor}M_\delta^r(k)-M_\delta^r(1)>Br\delta^{1+\eta}/2\right)\nonumber\\
    &\qquad+\P\left(T_1^r> r^2B\E[v]\delta^{2(1+\eta)}\right)\nonumber\\
&\qquad+\P\left(E^r(7r^2B\E[v]\delta^{2(1+\eta)})>\left\lfloor\frac{28}{3}Br^2\delta^{2(1+\eta)}\right\rfloor\right).\label{before doob}
\end{align}
Note that \eqref{C B} gives an upper bound of $1/282$ on the second term in \eqref{before doob}. Next we upper bound the other two terms in \eqref{before doob}.\\
\indent We begin with upper bounding the first term in \eqref{before doob}. Note that for $i\geq 2$, $\E[v_i\id_{[v_i\leq \delta c^r]}-\rho^r_{\delta c^r}T_i^r]=\E[v\id_{[v\leq \delta c^r]}](1-\lambda^r\E[T^r])=0$. This together with the independence assumptions implies that $\{M_\delta^r(k)-M_\delta^r(1)\}_{k\in\N}$ is a martingale \big(with respect to the filtration $\{\F_n^r\}_{n\in\N}$, where $\F_1^r$ is the trivial $\sigma$-algebra and $\F_n^r:=\sigma(v_i,T_i^r:2\leq i\leq n)$ for $n\geq2$\big). Also, by the independence assumptions and \eqref{4.4.6}, for each $ k\in\N$, we have 
\begin{equation}
\E[\left(M_\delta^r(k)-M_\delta^r(1)\right)^2]\leq C(k-1)\leq Ck.    \label{ck}
\end{equation}
Thus, by Doob's maximal inequality, \eqref{ck} and $B=282C$, we have 
\begin{align}
&\P\left(\sup_{2\leq k\leq \left\lfloor\frac{28}{3}Br^2\delta^{2(1+\eta)}\right\rfloor}M_\delta^r(k)-M_\delta^r(1)>Br\delta^{1+\eta}/2\right)\nonumber\\
\leq &\frac{4\E\left[\left(M_\delta^r(\lfloor\frac{28}{3}Br^2\delta^{2(1+\eta)}\rfloor)-M_\delta^r(1)\right)^2\right]}{\left(Br\delta^{1+\eta} \right)^2}
\leq \frac{112CBr^2\delta^{2(1+\eta)}}{3B^2r^2\delta^{2(1+\eta)}}\leq\frac{38}{282}.\label{208C 3B}
\end{align}

To upper bound the last term of \eqref{before doob}, we apply \eqref{4.4.9} with $b=\frac{28}{3}B\E[v]\geq\E[v]$, noting that $\frac{3b}{4}=7B\E[v]$ and $\lambda\E[v]=1$ so that $b\lambda=\frac{28}{3}B$ and using that $B=282 C$, \eqref{4.4.8} and $C\geq1$, we obtain
\begin{align}
    \P\left(E^r(7B\E[v]r^2\delta^{2(1+\eta)})>\left\lfloor\frac{28}{3}Br^2\delta^{2(1+\eta)}\right\rfloor\right)&\leq\frac{3\times 2^7\lambda^2\sigma_A^2}{7\times 282 r} \leq \frac{55}{282}.\label{176C 3B}
\end{align}
Combining \eqref{before ck}, \eqref{before doob}, \eqref{208C 3B}, \eqref{C B}, and \eqref{176C 3B} completes the proof that \eqref{after 387} holds and the desired result follows from this.
\end{proof}
\end{lem}

The result in Lemma \ref{lemma 1 in original paper} allows for a comparison with a biased random walk as stated in the next lemma.  For this for $\delta>0$, let $\{S_\delta(k)\}_{k\in\Z_+}$ be a random walk such that $S_\delta(0) = 9B\delta^{1+\eta}/2$ and for $k\in\Z_+$
\begin{equation}\label{defS}
S_\delta(k+1)-S_\delta(k)=
\begin{cases}
\frac{3B\delta^{1+\eta}}{2},&\text{with probability 1/3,}\\
-\frac{3B\delta^{1+\eta}}{2},&\text{with probability 2/3}.
\end{cases}
\end{equation}

\begin{lem}\label{new lemma for biased random walk}
Suppose $\eta\in(0,\eta_0)$. Let $\widetilde{r}(\eta,\E[v]),M_*(\eta)$ and $\delta_*(\eta)$ be as in Lemma \ref{a lot of assumptions}. 
Then for $r\geq\widetilde{r}(\eta,\E[v])$, $\delta\in[M_*(\eta)/c^r$, $\delta_*(\eta)]$, and $k\in\N$,
\begin{align}
    &\P\left(\widetilde{X}_\delta^r(\cdot+u_k^r/r^2)\text{ crosses }48B\delta^{1+\eta}\log{(1/\delta)}\text{ before }\frac{3\delta^{1+\eta}}{2}\middle\vert\widetilde{X}_\delta^r(u_k^r/r^2)=\frac{9\delta^{1+\eta}}{2}\right)\nonumber\\
    \leq&\P\left(S_\delta(\cdot)\text{ crosses }48B\delta^{1+\eta}\log{(1/\delta)}\text{ before }\frac{3\delta^{1+\eta}}{2}\middle\vert S_\delta(0)=\frac{9\delta^{1+\eta}}{2}\right)\leq\delta^{16}.\label{probability of widetilde X ukr over r squared bounded above by delta to the D}
\end{align}
\begin{proof}
Fix $r\geq \widetilde{r}(\eta,\E[v])$ and $\delta\in[M_*(\eta)/c^r,\delta_*(\eta)]$. We claim that $\delta c^r/r<B\delta^{(1+\eta)}/2$. By assumption, $M_*(\eta)/c^r\leq\delta$ and $M_*(\eta)>1$. Thus,
\[
\frac{\delta c^r}{r}=\frac{\delta^{1+\eta}c^r}{r\delta^n}\leq\frac{\delta^{1+\eta}}{r}\times\frac{c^r}{\left(M_*(\eta)/c^r\right)^\eta}\leq \frac{\delta^{1+\eta}}{r}\times\frac{c^r}{\left(1/c^r\right)^\eta}=\frac{\delta^{1+\eta}(c^r)^{1+\eta}}{r}.
\]
By \eqref{4.4.7} and \eqref{4.4.8} (which implies that $\sqrt{r}\geq1\geq2/B$) the above inequality becomes
\begin{equation}
\frac{\delta c^r}{r}\leq \frac{\delta^{1+\eta}(c^r)^{1+\eta}}{r}\leq \frac{\delta^{1+\eta}}{r^{1/2}}<\frac{B\delta^{1+\eta}}{2}.\label{delta c^r / r}
\end{equation}
Next it suffices to prove that the first inequality in \eqref{probability of widetilde X ukr over r squared bounded above by delta to the D} holds for $k=1$. For this, we define stopping times $\{\beta_k^r\}_{k\in\Z_+}$ with respect to the filtration $\{\mathcal{H}^r_t\}_{t\geq0}$, where $\mathcal{H}^r_t := \sigma\left(\widetilde{V}_\delta^r(s),E^r(r^2s):s\leq t\right)$ for $t\geq0$, such that $\beta_0^r=T_1^r/r^2$ and for $k\in\N$,
\begin{align*}
\beta_k^r:= 
&\inf\big\{t\geq\beta_{k-1}^r:\widetilde{X}_\delta^r(t)-\widetilde{X}_\delta^r(\beta_{k-1}^r)\geq B\delta^{1+\eta}\\
&\qquad\text{ or }E^r(r^2t)-E^r(r^2t-)>0\text{ and }\widetilde{X}_\delta^r(t-)-\widetilde{X}_\delta^r(\beta_{k-1}^r)\leq-2B\delta^{1+\eta}\big\}.
\end{align*}
Thus, for any $k\in\N$, observing that the maximum jump up of $\widetilde{X}_\delta^r(\cdot)$ is $\delta c^r/r$ and combining this with \eqref{delta c^r / r}, we obtain two possibilities:
\begin{itemize}
    \item If $\widetilde{X}_\delta^r(\beta_{k+1}^r)-\widetilde{X}_\delta^r(\beta_{k}^r)\geq B\delta^{1+\eta}$, then 
    \[
    \widetilde{X}_\delta^r(\beta_{k+1}^r)-\widetilde{X}_\delta^r(\beta_{k}^r)\leq B\delta^{1+\eta}+\frac{\delta c^r}{r}\leq\frac{3B\delta^{1+\eta}}{2}.
    \]
    \item If $\widetilde{X}_\delta^r(\beta_{k+1}^r-)-\widetilde{X}_\delta^r(\beta_{k}^r)\leq -2B\delta^{1+\eta}$, then 
    \[
    \widetilde{X}_\delta^r(\beta_{k+1}^r)-\widetilde{X}_\delta^r(\beta_{k}^r)\leq -2B\delta^{1+\eta}+\frac{\delta c^r}{r}\leq -\frac{3B\delta^{1+\eta}}{2}.
    \]
\end{itemize}
This together with the result in Lemma \ref{lemma 1 in original paper} implies that the first inequality in \eqref{probability of widetilde X ukr over r squared bounded above by delta to the D} holds.
To see that the second inequality in \eqref{probability of widetilde X ukr over r squared bounded above by delta to the D} holds, note that
\begin{align}
    &\P\left(S_\delta(\cdot)\text{ crosses }48B\delta^{1+\eta}\log{(1/\delta)}\text{ before }\frac{3B\delta^{1+\eta}}{2}\middle\vert S_\delta(0)=\frac{9B\delta^{1+\eta}}{2}\right)\nonumber\\
    &\qquad\leq\frac{2^3-2}{2^{32\log(1/\delta)}-2}\leq\delta^{16},\label{S delta crosses}
\end{align}
where the first inequality follows by a standard calculation for the biased random walk $\{S_\delta(k)\}_{k\in\Z_+}$ (see Appendix \ref{Appendix B}) and the second inequality follows from \eqref{4.4.1}.
\end{proof}
\end{lem}

Recall that for $r\in R$ and $\delta>0$, $\widetilde{Y}_\delta^r(\cdot)= \Gamma[\widetilde{X}_\delta^r](\cdot)$. Thus, we obtain the following corollary.
\begin{cor}\label{lemma 2 in original paper}
Suppose $\eta\in(0,\eta_0)$. Let $\widetilde{r}(\eta,\E[v]),M_*(\eta)$, and $\delta_*(\eta)$ be as in Lemma \ref{a lot of assumptions}. Then for any $r\geq\widetilde{r}(\eta,\E[v])$, $\delta\in[M_*(\eta)/c^r,\delta_*(\eta)]$, $t\geq0$, $k\in\N$, and $x_0\in[4B\delta^{1+\eta},\frac{9B\delta^{1+\eta}}{2}]$,
\[
\P\left(\widetilde{Y}_\delta^r(\cdot+u_k^r/r^2)\text{ crosses } 48B\delta^{1+\eta}\log(1/\delta)\text{ before }\frac{3B\delta^{1+\eta}}{2}\middle\vert\widetilde{Y}_\delta^r(u_k^r/r^2)=x_0\right)\leq\delta^{16}.
\]
\begin{proof}
Fix $r\geq\widetilde{r}(\eta,\E[v])$, $\delta\in[M_*(\eta)/c^r,\delta^*(\eta)]$, $k\in\N$ and $x_0\in[4B\delta^{1+\eta},\frac{9B\delta^{1+\eta}}{2}]$.  If $\widetilde{Y}_\delta^r(u_k^r/r^2)=\widetilde{X}_\delta^r(u_k^r/r^2)=x_0$, then $\widetilde{Y}_\delta^r(t+u_k^r/r^2)=\widetilde{X}_\delta^r(t+u_k^r/r^2)$ for all $t\in[0,t_*]$, where $t_*=\inf\{t\geq0:\widetilde{X}_\delta^r(t+u_k^r/r^2)=0\}$, and $\widetilde{X}_\delta^r(\cdot+u_k^r/r^2)$ must cross $\frac{3B\delta^{1+\eta}}{2}$ before time $t_*$.
Thus,
\begin{align}
    &\P\left(\widetilde{Y}_\delta^r(\cdot+u_k^r/r^2)\text{ crosses }3DB\delta^{1+\eta}\log(1/\delta)\text{ before }\frac{3B\delta^{1+\eta}}{2}\middle\vert\widetilde{Y}_\delta^r(u_k^r/r^2)=x_0\right)\nonumber\\
    =&\P\left(\widetilde{X}_\delta^r(\cdot+u_k^r/r^2)\text{ crosses }3DB\delta^{1+\eta}\log(1/\delta)\text{ before }\frac{3B\delta^{1+\eta}}{2}\middle\vert\widetilde{X}_\delta^r(u_k^r/r^2)=x_0\right)\nonumber\\
    \leq&\P\left(\widetilde{X}_\delta^r(\cdot+u_k^r/r^2)\text{ crosses }3DB\delta^{1+\eta}\log(1/\delta)\text{ before }\frac{3B\delta^{1+\eta}}{2}\middle\vert\widetilde{X}_\delta^r(u_k^r/r^2)=\frac{9B\delta^{1+\eta}}{2}\right).\nonumber
\end{align}
The result follows from this and Lemma \ref{new lemma for biased random walk}.
\end{proof}
\end{cor}

In the next lemma, $r$ sufficiently large and $\delta$ in the range specified by Lemma \ref{a lot of assumptions} are considered. Time intervals are constructed where $\widetilde{Y}_\delta^r(\cdot)$ starts above $4B\delta^{1+\eta}$ and eventually drops to $2B\delta^{1+\eta}$ during which $\widetilde{Y}_\delta^r(\cdot)$ attempts to reach level $48B\delta^{1+\eta}\log{(1/\delta)}$. The result states that if $\widetilde{Y}_\delta^r(\cdot)$ is to reach level $48B\delta^{1+\eta}\log{(1/\delta)}$, it will take a very large number of attempts with very high probability.

\begin{lem}\label{SRPT 5.82}\label{definition of tau}
Suppose $\eta\in(0,\eta_0)$. Let $\widetilde{r}(\eta,\E[v])$, $M_*(\eta)$ and $\delta_*(\eta)$ be as in Lemma \ref{a lot of assumptions}. For $r\geq \widetilde{r}(\eta,\E[v])$ and $\delta\in[M_*(\eta)/c^r,\delta_*(\eta)]$, let the filtration $\{\mathcal{H}_t^r\}_{t\geq0}$ be as in the proof of Lemma \ref{new lemma for biased random walk} and define stopping times $\{\tau_{k}^r\}_{k=-1}^{\infty}$ with respect to the filtration $\{\mathcal{H}_t^r\}_{t\geq0}$ such that $\tau_{-1}^r=0$ and for $k\in\Z_+$,
\begin{align}
    \tau_{2k}^r &:= \inf\{t\geq \tau_{2k-1}^r:E^r(r^2t)-E^r(r^2t-)>0\text{ and }\widetilde{Y}_\delta^r(t-)\leq 2B\delta^{1+\eta}\},\label{tau2k}\\
    \tau_{2k+1}^r &:=\inf\{t\geq \tau_{2k}^r:\widetilde{Y}_\delta^r(t)\geq 4B\delta^{1+\eta}\},\qquad\text{and let}\label{tau2k+1}\\
    \mathcal{N}^r&:=\inf\left\{k\in\Z_+:\sup_{t\in[\tau_{2k-1}^r,\tau_{2k}^r)}\widetilde{Y}_\delta^r(t)> 48B\delta^{1+\eta}\log(1/\delta)\right\}.\label{definition of mathcal N}
\end{align}
Then, for $r\geq\widetilde{r}(\eta,\E[v])$ and $\delta\in[M_*(\eta)/c^r,\delta_*(\eta)]$, 
\begin{alignat*}{3}
\sup\limits_{t\in[\tau_0^r,\tau^r_{2\mathcal{N}^r-1})}\widetilde{Y}_\delta^r(t)&\leq 48B\delta^{1+\eta}\log{(1/\delta)},&\qquad &\text{and }\\
\P\left(\NN^r\leq \lfloor\delta^{-8}\rfloor+1\right)&\leq D_1\delta^{\eta_0-\eta},&\qquad&\text{where }D_1=\frac{D_0}{2B}+1.
\end{alignat*}
\begin{proof}
Fix $r\geq\widetilde{r}(\eta,\E[v])$ and $\delta\in[M_*(\eta)/c^r,\delta_*(\eta)]$.
To begin, we consider a time interval
$[\tau_{2k-1}^r,\tau_{2k}^r)$ with $k\in\N$.
For this, recall that $\widetilde{Y}_\delta^r$ has upward jumps of size at most $c^r\delta/r$. In addition, applying the inequalities \eqref{delta c^r / r} and \eqref{corollary of 4.4.1}, for each $k\in\N$, we have $\widetilde{Y}_\delta^r(\tau_{2k-1}^r)\in[4B\delta^{1+\eta},9B\delta^{1+\eta}/2)$, which implies that $\widetilde{Y}^r_\delta(\tau_{2k-1}^r)<48B\delta^{1+\eta}\log{(1/\delta)}$.
These facts together with definitions \eqref{tau2k} - \eqref{definition of mathcal N} imply that the first statement of the lemma holds. 
For the second statement, by a union bound,
\[
\P\left(\NN^r\leq\lfloor\delta^{-8}\rfloor+1\right)\leq\sum_{k=0}^{\lfloor\delta^{-8}\rfloor+1}\P\left(\sup_{t\in[\tau_{2k-1}^r,\tau_{2k}^r)}\widetilde{Y}_\delta^r(t)>48B\delta^{1+\eta}\log(1/\delta)\right).
\]
If $\widetilde{Y}_\delta^r(0)\le 2B\delta^{1+\eta}$, then $\tau_{0}^r=T_1^r/r^2$ and $\sup_{t\in[0, \tau_0^r)}\widetilde{Y}_\delta^r(t)\le 2B\delta^{1+\eta}$.  This together with Markov's inequality, $\widetilde{Y}_\delta^r(0)=\widetilde{W}_\delta^r(0)$ and \eqref{tilde W bounded at initial} gives that
\[
\P\left(\sup_{t\in[\tau_{-1}^r,\tau_{0}^r)}\widetilde{Y}_\delta^r(t)>48B\delta^{1+\eta}\log(1/\delta)\right)\leq 
\P\left(\widetilde{Y}_\delta^r(0)>2B\delta^{1+\eta}\right)
\le \frac{D_0}{2B}\delta^{\eta_0-\eta}.
\]
Combining the above with Corollary \ref{lemma 2 in original paper} and $\lfloor\delta^{-8}\rfloor\delta^{16}\le \delta^8 $ completes the proof.
\end{proof}
\end{lem}

Recall our objective \eqref{probability of sup tilde Y greater 3DB bounded above by 34}.
For $T>0$, $\eta\in(0,1)$, $r\in R$ and $\delta>0$,
\begin{eqnarray}
    &\P\left(\|\widetilde{Y}_\delta^r\|_T>48B\delta^{1+\eta}\log{(1/\delta)}\right)\leq
    \P\left(\NN^r\leq \lfloor\delta^{-8}\rfloor+1\right)\nonumber\\ &\qquad\qquad+
    \P\left(\NN^r>\lfloor\delta^{-8}\rfloor+1\text{ and }\|\widetilde{Y}_\delta^r\|_T>48B\delta^{1+\eta}\log{(1/\delta)}\right)\nonumber\\
    &\qquad\leq
    \P\left(\NN^r\leq \lfloor\delta^{-8}\rfloor+1\right)+\P\left(\sum_{k=0}^{\lfloor\delta^{-8}\rfloor+1}\left(\tau_{2k+1}^r-\tau_{2k}^r\right)\le T\right),\label{ineq:Nr}
\end{eqnarray}
where the last line follows because $\mathcal{N}^r>\lfloor\delta^{-8}\rfloor+1$ and $\|\widetilde{Y}_\delta^r\|_T>48B\delta^{1+\eta}\log{(1/\delta)}$ together imply that 
$
\sum_{k=0}^{\lfloor\delta^{-8}\rfloor+1}\left(\tau_{2k+1}^r-\tau_{2k}^r\right)\leq\tau_{2\lfloor\delta^{-8}\rfloor+3}^r\leq\tau_{2\NN^r-1}^r\leq T$.
The quantity $
\sum_{k=0}^{\lfloor\delta^{-8}\rfloor+1}\left(\tau_{2k+1}^r-\tau_{2k}^r\right)$ equals the total time spent crossing from below level $2B\delta^{1+\eta}$ to above level $4B\delta^{1+\eta}$ over the first $\lfloor\delta^{-8}\rfloor+1$ such upcrossings, which should be large with high probability when $\delta$ is small.

The next two lemmas establish preliminary results that are used to show that for $\eta\in(0,\eta_0)$, $r$ sufficiently large, suitable $\delta>0$ and $T>0$, $\P\left(\sum_{k=0}^d\left(\tau_{2k+1}^r-\tau_{2k}^r\right)\leq T\right)$ decays exponentially as $d$ increases. The reader may wish to preview the statement of Corollary \ref{SRPT Lemma 18 5.92} before continuing with Lemmas \ref{probability of the sum of chi} and \ref{Lemma 3 in SRPT}. A key idea toward this is to study $\Gamma$ applied to the driftless process $3B\delta^{\eta_0-\eta}+\widetilde{V}_\delta^r(\cdot+\tau_{2k}^r)-\widetilde{V}_\delta^r(\tau_{2k}^r)$ and to show that it crosses level $4B\delta^{1+\eta}$ before $\widetilde{Y}_\delta^r(\cdot+\tau_{2k}^r)$ for any $k\in\Z_+$. See \eqref{tilde Y the sum of t and tau2j} below.  This leads one to define random variables $\{\chi_k^r\}_{k=0}^\infty$ as in \eqref{probability of chij smaller equal s} below
and gives rise to the upper bound \eqref{sum of difference tau2j+1 and tau2j}. Once \eqref{sum of difference tau2j+1 and tau2j} is established, it is shown that the probability that $\chi_0^r$ is less or equal $\epsilon_*\delta^{1+\eta}$ is at most $1/2$,
for $\eta\in(0,\eta_0)$, $r$ sufficiently large, and suitable $\delta>0$. See Lemma \ref{Lemma 3 in SRPT} below. The proof of Lemma \ref{Lemma 3 in SRPT} uses martingale arguments similar to those used in the proof Lemma \ref{lemma 1 in original paper}. The upper bound of $1/2$ is used to prove the exponential decay bound in \eqref{probability of sum of tau 2j+1 and tau 2j bounded above by e-d/32} for $T=d\epsilon_*\delta^{1+\eta}/4$, where $d\in\N$ is the number of attempts to cross level $48B\delta^{1+\eta}\log{(1/\delta)}$.
\begin{lem}\label{probability of the sum of chi}
Suppose $\eta\in(0,\eta_0)$. Let $\widetilde{r}(\eta,\E[v])$, $M_*(\eta)$ and $\delta_*(\eta)$ be as in Lemma \ref{a lot of assumptions}. For $r\geq\widetilde{r}(\eta,\E[v])$ and $\delta\in[M_*(\eta)/c^r,\delta_*(\eta)]$, let $\{\tau_k^r\}_{k=0}^\infty$ be as in Lemma \ref{SRPT 5.82}, and define $\{\chi_k^r\}_{k\in\Z_+}$ to be a sequence of $\R_+$ valued random variables such that for any $k\in\Z_+$ and $t\geq0$,
\begin{equation}
    \left\{\chi_k^r\leq t\right\}:=\left\{\sup_{s\in[0,t]}\Gamma[3B\delta^{1+\eta}+\widetilde{V}_\delta^r(\cdot+\tau_{2k}^r)-\widetilde{V}_\delta^r(\tau_{2k}^r)](s)\geq4B\delta^{1+\eta}\right\}.\label{probability of chij smaller equal s}
\end{equation}
Then for each $r\geq \widetilde{r}(\eta,\E[v])$ and $\delta\in[M_*(\eta)/c^r,\delta_*(\eta)]$, $\{\chi_k^r\}_{k\in\Z_+}$ is i.i.d.\ and for each $r\geq \widetilde{r}(\eta,\E[v])$, $\delta\in[M_*(\eta)/c^r,\delta_*(\eta)]$, $d\in\N$ and $t\geq0$,
\begin{equation}
\P\left(\sum_{k=0}^d(\tau_{2k+1}^r-\tau_{2k}^r)\leq t\right)\leq \P\left(\sum_{k=0}^d\chi_k^r\leq t\right).\label{sum of difference tau2j+1 and tau2j}
\end{equation}
\begin{proof}
Fix $r\geq\widetilde{r}(\eta,\E[v])$ and $\delta\in[M_*(\eta)/c^r,\delta_*(\eta)]$. By the strong Markov property, $\{\chi_k^r\}_{k\in\Z_+}$ is i.i.d. In what follows we will show that for each $k\in\Z_+$ and $t\geq0$,
\begin{equation}
    \widetilde{Y}^r(t+\tau_{2k}^r)\leq\Gamma\left[3B\delta^{1+\eta}+\widetilde{V}_\delta^r(\cdot+\tau_{2k}^r)-\widetilde{V}_\delta^r(\tau_{2k}^r)\right](t).\label{tilde Y the sum of t and tau2j}
\end{equation}
This together with \eqref{tau2k}, \eqref{tau2k+1}, and \eqref{probability of chij smaller equal s} implies that $\tau_{2k+1}^r-\tau_{2k}^r\geq \chi_k^r$
for each $k\in\Z_+$ and \eqref{sum of difference tau2j+1 and tau2j} follows. So it remains to prove \eqref{tilde Y the sum of t and tau2j}. Fix $k\in\Z_+$ and define $f:\R_+\to\R$ as
\[
f(t):=(\widetilde{V}_\delta^r(t+\tau_{2k}^r)-\widetilde{V}_\delta^r(\tau_{2k}^r))-(\widetilde{X}_\delta^r(t+\tau_{2k}^r)-\widetilde{X}_\delta^r(\tau_{2k}^r)),\quad t\geq0.
\]
By \eqref{Xtilde} and \eqref{eq:drift}, $f(t) = \frac{\lambda^rS(c^r)}{S(\delta c^r)}t-r(\rho^r-1)t$ for all $t\ge 0$.  Then, by \eqref{liminf and inf swap},
\begin{align*}
f'(t) &=\frac{\lambda^rS(c^r)}{S(\delta c^r)}-r(\rho^r-1)\geq\frac{5\lambda}{12\delta^{1+\eta}}>0,\quad t\geq0.
\end{align*}
Thus, $f$ is nondecreasing and so $f(s)\leq f(t)$ for all $s\leq t$. Thus, $\widetilde{X}_\delta^r(t+\tau_{2k}^r)-\widetilde{X}_\delta^r(s+\tau_{2k}^r)\leq\widetilde{V}_\delta^r(t+\tau_{2k}^r)-\widetilde{V}_\delta^r(s+\tau_{2k}^r)$, for all $s\leq t$. By the shift and monotonicity properties of the Skorokhod Map from Propositions \ref{shift property of Gamma} and \ref{monotonicity of skorohod} respectively, for each $t\geq0$, 
\begin{align*}
    \widetilde{Y}_\delta^r(t+\tau_{2k}^r)&=\Gamma[\widetilde{Y}_\delta^r(\tau_{2k}^r)+(\widetilde{X}_\delta^r(\cdot+\tau_{2k}^r)-\widetilde{X}_\delta^r(\tau_{2k}^r))](t)\nonumber\\
    &\leq \Gamma[3B\delta^{1+\eta}+(\widetilde{V}_\delta^r(\cdot+\tau_{2k}^r)-\widetilde{V}_\delta^r(\tau_{2k}^r))](t),
\end{align*}
where we used the fact that $\widetilde{Y}_\delta^r(\tau_{2k}^r)\leq 3B\delta^{1+\eta}$ for all $k\in\Z_+$ which follows by the definition of $\widetilde{Y}_\delta^r$, and that the maximum jump up is $\delta c^r/r$ and bounded above by $B\delta^{1+\eta}/2$ (see \eqref{delta c^r / r}). Thus, \eqref{tilde Y the sum of t and tau2j} holds.
\end{proof}
\end{lem}

\begin{lem}\label{Lemma 3 in SRPT}
Suppose $\eta\in(0,\eta_0)$. Let $\widetilde{r}(\eta,\E[v]\wedge4\epsilon_*/3),M_*(\eta)$ and $\delta_*(\eta)$ be as in Lemma \ref{a lot of assumptions}. For $r\geq\widetilde{r}(\eta,\E[v]\wedge4\epsilon_*/3)$ and $\delta\in[M_*(\eta)/c^r,\delta_*(\eta)]$, let $\{\tau_k^r\}_{k\in\Z_+}$ and $\{\chi_k^r\}_{k\in\Z_+}$ be as in Lemmas \ref{SRPT 5.82} and \ref{probability of the sum of chi} respectively. 
Then for all $r\geq\widetilde{r}(\eta,\E[v]\wedge4\epsilon_*/3)$, $\delta\in[M_*(\eta)/c^r,\delta_*(\eta)]$, and $\epsilon\in[0,\epsilon_*]$,
\begin{equation}
\P\left(\chi_0^r\leq \epsilon\delta^{2(1+\eta)}\right)\leq \frac{1}{2}.\label{chi0 smaller equal to epsilon delta smaller equal to one half}
\end{equation}
\begin{proof}
Fix $r\geq\widetilde{r}\left(\eta,\E[v]\wedge4\epsilon_*/3\right)$ and $\delta\in[M_*(\eta)/c^r,\delta_*(\eta)]$. It suffices to prove \eqref{chi0 smaller equal to epsilon delta smaller equal to one half} for $\epsilon=\epsilon_*$ since $\P\left(\chi_0^r\leq \epsilon\delta^{2(1+\eta)}\right)\leq\P\left(\chi_0^r\leq \epsilon_*\delta^{2(1+\eta)}\right)$ for all $\epsilon\in[0,\epsilon_*]$. 
Observe that for any positive constant function $f$, $\Gamma[f](\cdot)=f(\cdot)$. This together with Proposition \ref{Lipschitz property of Skorokhod}, the Lipschitz property of $\Gamma$, gives us that 
\begin{align}
    &\P\left(\chi_0^r\leq\epsilon_*\delta^{2(1+\eta)}\right)\nonumber\\
    &\qquad= \P\left(\sup_{t\in[0,\epsilon_*\delta^{2(1+\eta)}]}\Gamma[3B\delta^{1+\eta}+\widetilde{V}_\delta^r(\cdot+\tau_0^r)-\widetilde{V}_\delta^r(\tau_0^r)](t)\geq4B\delta^{1+\eta}\right)\nonumber\\
   &\qquad =\P\left(\sup_{t\in[0,\epsilon_*\delta^{2(1+\eta)}]}\Gamma[3B\delta^{1+\eta}+\widetilde{V}_\delta^r(\cdot+\tau_0^r)-\widetilde{V}_\delta^r(\tau_0^r)](t)-3B\delta^{1+\eta}\geq B\delta^{1+\eta}\right)\nonumber\\
    &\qquad\leq\P\left(\sup_{t\in[0,\epsilon_*\delta^{2(1+\eta)}]}\left|\Gamma[3B\delta^{1+\eta}+\widetilde{V}_\delta^r(\cdot+\tau_0^r)-\widetilde{V}_\delta(\tau_0^r)](t)-\Gamma[3B\delta^{1+\eta}](t)\right|\geq B\delta^{1+\eta}\right)\nonumber\\
    &\qquad\leq \P\left(\sup_{t\in[0,\epsilon_*\delta^{2(1+\eta)}]}|\widetilde{V}_\delta^r(t+\tau_0^r)-\widetilde{V}_\delta^r(\tau_0^r)|\geq B\delta^{1+\eta}/2\right)\nonumber\\
    &\qquad\leq \P\left(\sup_{t\in(0,\epsilon_*\delta^{2(1+\eta)}]}\widetilde{V}_\delta^r(t+\tau_0^r)-\widetilde{V}_\delta^r(\tau_0^r)\geq B\delta^{1+\eta}/2\right)\nonumber\\
    &\qquad\qquad+\P\left(\inf_{t\in[0,\epsilon_*\delta^{2(1+\eta)}]}\widetilde{V}_\delta^r(t+\tau_0^r)-\widetilde{V}_\delta^r(\tau_0^r)\leq -B\delta^{1
    +\eta}/2\right).\label{chi0 bound}
\end{align}
Next we upper bound each of the terms in \eqref{chi0 bound}. For this, let $\{M^r_\delta(k)\}_{k\in\N}$ be as in the proof of Lemma \ref{lemma 1 in original paper} and $b=4\epsilon_*/3$ in \eqref{4.4.9}. Then by arguing similarly to the proof of Lemma \ref{lemma 1 in original paper} and noting that $\sup_{t\in[0,\epsilon_*\delta^{2(1+\eta)}]}\widetilde{V}_\delta^r(t+\tau_0^r)-\widetilde{V}_\delta^r(\tau_0^r)$ is equal in distribution to
$\sup_{t\in[0,\epsilon_*\delta^{2(1+\eta)}]}\widetilde{V}_\delta^r(t+T_1^r/r^2)-\widetilde{V}_\delta^r(T_1^r/r^2)
$, we have
\begin{align}
&\P\left(\sup_{t\in[0,\epsilon_*\delta^{2(1+\eta)}]}\widetilde{V}_\delta^r(t+\tau_0^r)-\widetilde{V}_\delta^r(\tau_0^r)\geq B\delta^{1+\eta}/4\right)\nonumber\\
&\qquad=\P\left(\max\left\{M_\delta^r(k)-M_\delta^r(1) :2\leq k\leq E^r\left(\epsilon_*r^2\delta^{2(1+\eta)}+T_1^r\right)\right\}\geq B\delta^{1+\eta}/4\right)\nonumber\\
&\qquad\leq\P\left(\max\left\{M_\delta^r(k)-M_\delta^r(1): 2\leq k\leq \left\lfloor\lambda 16\epsilon_* r^2\delta^{2(1+\eta)}/9\right\rfloor\right\}
\geq B\delta^{1+\eta}r/4\right)\nonumber\\
&\qquad\qquad+\P\left(E^r\left(4\epsilon_* r^2\delta^{2(1+\eta)}/3\right)>\left\lfloor 16\epsilon_*\lambda r^2\delta^{2(1+\eta)}/9\right\rfloor\right) + \P\left(T_1^r/r^2>\epsilon_*\delta^{2(1+\eta)}/3\right).\nonumber
\end{align}
Then by Doob's maximal inequality and \eqref{4.4.9} in Lemma \ref{a lot of assumptions} with $b=16\epsilon_*/9$, and by using $C/B=1/282$, $1/B\leq1$ and $256/(9\cdot282)\leq 1/8$, we have 
\[
\P\left(\max_{2\leq k\leq \left\lfloor\lambda 16\epsilon_*r^2\delta^{2(1+\eta)}/9\right\rfloor}M_\delta^r(k)-M_\delta^r(1)\geq B\delta^{1+\eta}r/4\right)\leq\frac{C\times 16\epsilon_*\lambda r^2\delta^{2(1+\eta)}}{9B^2\delta^{2(1+\eta)}r^2/16}\leq\frac{\lambda\epsilon_*}{8},
\]
and
\[
\P\left(E^r\left(4\epsilon_* r^2\delta^{2(1+\eta)}/3\right)>\left\lfloor\lambda 16 \epsilon_* r^2\delta^{2(1+\eta)}/9\right\rfloor\right)\leq \frac{9\cdot2^8\lambda\sigma_A^2}{16 \epsilon_* r}\leq \frac{2^9\lambda\sigma_A^2}{\epsilon_*r}.
\]
Thus, by the above and \eqref{condition involving epsilon *},
\begin{equation}
\P\left(\sup_{t\in[0,\epsilon_*\delta^{2(1+\eta)}]}\widetilde{V}_\delta^r(t+\tau_0^r)-\widetilde{V}_\delta^r(\tau_0^r)\geq B\delta^{1+\eta}/4\right)\leq\frac{\lambda\epsilon_*}{8}+\frac{2^9\lambda\sigma_A^2}{\epsilon_*r}+\frac{1}{18}.\label{widetilde V sup}    
\end{equation}
Next we bound the probability of the event in \eqref{chi0 bound} involving the infimum. Observe that on the time interval $[0,\epsilon_*\delta^{2(1+\eta)}]$, we have 
\begin{align}
    &\inf_{t\in[0,\epsilon_*\delta^{2(1+\eta)}]}\widetilde{V}_\delta^r(t+T_1^r/r^2)-\widetilde{V}_\delta^r(T_1^r/r^2)\nonumber\\
    &\qquad\ge \frac{1}{r}\min\left\{\sum_{i=2}^k v_i\id_{\{v_i\le \delta c^r\}}-\rho_{\delta c^r}^r\sum_{i=2}^{k+1}T_i^r : 1\le k\le E^r\left(r^2\epsilon_* \delta^{2(1+\eta)}+T_1^r\right)\right\}\nonumber\\
    &\qquad=\frac{1}{r}\min\left\{M_\delta^r(k)-M_\delta^r(1)-\rho_{\delta c^r}^rT_{k+1}^r : 1\leq k\leq E^r\left(r^2\epsilon_* \delta^{2(1+\eta)}+T_1^r\right)\right\}\nonumber\\
    &\qquad\geq-\frac{1}{r}\max\left\{M_\delta^r(1)-M_\delta^r(k) : 1\leq k\leq E^r\left(r^2\epsilon_* \delta^{2(1+\eta)}+T_1^r\right)\right\}\nonumber\\
    &\qquad\qquad-\frac{8}{7r}\max\left\{ T_k^r : 2\leq k\leq E^r\left(r^2\epsilon_* \delta^{2(1+\eta)}+T_1^r\right)+1\right\},\label{widetilde V inf}
\end{align}
where the first inequality is from the definition of $\widetilde{V}_\delta^r$ and the fact that the infimum is achieved immediately prior to one of the jump times, the equality is from
the definition of $\{M^r_\delta(k)\}_{k\in\N}$ and the last inequality is from $\rho_{\delta c^r}^r\le\lambda^r\E[v]$, \eqref{4.4.2} and $\lambda \E[v]=1$. Above it is understood that $\sum_{i=2}^1v_i\id_{[v_i\leq\delta c^r]}=0$. Following an argument similar to the one that leads to \eqref{widetilde V sup}, we obtain 
\begin{align}
    &\P\left(\frac{1}{r}\max\left\{M_\delta(1)-M_\delta^r(k) : 1\leq k\leq E^r\left(r^2\epsilon_*\delta^{2(1+\eta)}+T_1^r\right)\right\}\geq \frac{B\delta^{1+\eta}}{8}\right)\nonumber\\
    &\qquad\leq\frac{\lambda\epsilon_*}{2}+\frac{2^9\lambda\sigma_A^2}{\epsilon_*r}+\frac{1}{18}.\label{probability 1 over r sup Mdelta(1) minums Mdelta(k)}
\end{align}
Moreover, we have 
\begin{align}
    &\P\left(\frac{8}{7r}\max\left\{T_k^r : 2\leq k\leq E^r\left(r^2\epsilon_* \delta^{2(1+\eta)}+T_1^r\right)+1\right\}\geq \frac{B\delta^{1+\eta}}{8}\right) \nonumber\\
    &\qquad\leq \P\left(\max\left\{T_k^r : 2\leq k\leq \left\lfloor\lambda 16 r^2\epsilon_*\delta^{2(1+\eta)}/9\right\rfloor+1\right\}\geq\frac{7B\delta^{1+\eta}r}{64}\right)\label{before union bound}\\
    &\qquad\quad+\P\left(E^r\left(4\epsilon_*r^2 \delta^{2(1+\eta)}/3\right)>\left\lfloor\lambda 16 \epsilon_* r^2\delta^{2(1+\eta)}/9\right\rfloor\right)+\P\left(\frac{T_1^r}{r^2}>\epsilon_*r^2 \delta^{2(1+\eta)}/3\right).\nonumber
\end{align}
By a union bound, Markov's inequality, $\E[(T^r)^2]=(\lambda^r)^{-2}+(\sigma_A^r)^2$, \eqref{corollary of 4.4.1}, \eqref{4.4.2}, and $B^2\geq282^2$ $\left(\text{so that } \frac{2^{16}}{3^2 \cdot 7^2\cdot 282^2}\leq\frac{1}{2}\right)$, we have 
\begin{align}
    &\P\left(\max\left\{T_k^r : 2\leq k\leq \left\lfloor\frac{4}{3}\epsilon_*\lambda r^2\delta^{2(1+\eta)}\right\rfloor+1\right\}\geq\frac{7B\delta^{1+\eta}r}{64}\right)\nonumber\\
    &\qquad\leq \lambda\frac{16\epsilon_*}{9} r^2\delta^{2(1+\eta)}\P\left(T^r\geq\frac{7B\delta^{1+\eta}r}{64}\right)\leq\lambda\frac{16\epsilon_*}{9} r^2\delta^{2(1+\eta)}\left(\frac{64}{7B\delta^{1+\eta}r}\right)^2\E\left[(T^r)^2\right]\nonumber\\
    &\qquad = \frac{2^{16}\epsilon_*\lambda[(\lambda^r)^{-2}+(\sigma_A^r)^2]}{9\cdot49B^2}\leq \frac{\epsilon_*\lambda}{2}\left(\frac{9(\E[v])^2 }{4}+2\sigma_A^2\right).\label{after Markov}
\end{align}
Thus, by \eqref{before union bound}, \eqref{after Markov}, using the inequality from \eqref{4.4.9} with $b=16\epsilon_*/9$ and \eqref{condition involving epsilon *},
\begin{eqnarray}
    &\P\left(\frac{8}{7r}\max\left\{T_k^r : 2\leq i\leq E^r\left(r^2\epsilon_* \delta^{2(1+\eta)}+T_1^r\right)+1\right\}\geq \frac{B\delta^{1+\eta}}{8}\right)\nonumber\\
    &\qquad \leq\frac{\epsilon_*\lambda}{2}\left(\frac{9(\E[v])^2 }{4}+2\sigma_A^2\right)+\frac{2^9\lambda\sigma_A^2}{\epsilon_*r}+\frac{1}{18}.\label{second term of widetilde V inf} 
\end{eqnarray}
Combining \eqref{chi0 bound}, \eqref{widetilde V sup}, \eqref{widetilde V inf}, \eqref{probability 1 over r sup Mdelta(1) minums Mdelta(k)}, and \eqref{second term of widetilde V inf} yields that
\begin{align*}
\P\left(\chi_0^r\leq\epsilon_*\delta^{2(1+\eta)}\right)&\leq\P\left(\sup_{t\in[0,\epsilon_*\delta^{2(1+\eta)}]}|\widetilde{V}_\delta^r(t+\tau_0^r)-\widetilde{V}_\delta^r(\tau_0^r)|\geq B\delta^{2(1+\eta)}\right)\\
&\leq\frac{\epsilon_*\lambda}{8}\left(5+9(\E[v])^2 +8\sigma_A^2\right)+\frac{3\cdot2^9\lambda\sigma_A^2}{\epsilon_*r}+\frac{3}{18}.
\end{align*}
Using the definition of $\epsilon_*$ from \eqref{definition of epsilon sub *} in the first term above and then \eqref{4.4.8} in the second term above completes the proof.
\end{proof}
\end{lem}

\begin{cor}\label{SRPT Lemma 18 5.92}
Suppose $\eta\in(0,\eta_0)$. Let $\widetilde{r}\left(\eta,\E[v]\wedge4\epsilon_*/3\right)$, $M_*(\eta)$ and $\delta_*(\eta)$ as in Lemma \ref{a lot of assumptions}. For $r\geq\widetilde{r}\left(\eta,\E[v]\wedge4\epsilon_*/3\right)$ and $\delta\in[M_*(\eta)/c^r,\delta_*(\eta)]$, let $\{\tau_k^r\}_{k\in\Z_+}$ and $\{\chi_k^r\}_{k\in\Z_+}$ be as in Lemmas \ref{SRPT 5.82} and \ref{probability of the sum of chi} respectively. For all $r\geq\widetilde{r}\left(\eta,\E[v]\wedge4\epsilon_*/3\right)$, $\delta\in[M_*(\eta)/c^r,\delta_*(\eta)]$, $\epsilon\in[0,\epsilon_*]$ and $d\in\Z_+$, we have
\begin{equation}
\P\left(\sum_{k=0}^d(\tau_{2k+1}^r-\tau_{2k}^r)\leq (d+1)\epsilon\delta^{2(1+\eta)}/4\right)\leq e^{-(d+1)/8}.\label{probability of sum of tau 2j+1 and tau 2j bounded above by e-d/32}    
\end{equation}
\begin{proof}
Fix $r\geq\widetilde{r}\left(\eta,\E[v]\wedge4\epsilon_*/3\right)$, $\delta\in[M_*(\eta)/c^r,\delta_*(\eta)]$, $\epsilon\in[0,\epsilon_*]$ and $d\in\Z_+$. By Lemma \ref{probability of the sum of chi}, we have 
\begin{align*}
    &\P\left(\sum_{k=0}^d(\tau_{2k+1}^r-\tau_{2k}^r)\leq (d+1)\epsilon\delta^{2(1+\eta)}/4\right)\leq \P\left(\sum_{k=0}^d\chi_k^r\leq (d+1)\epsilon\delta^{2(1+\eta)}/4\right) \\
    &\qquad \leq \P\left(\sum_{k=0}^d\id_{[\chi_k^r>\epsilon\delta^{2(1+\eta)}]}\leq (d+1)/4\right)\\
    &\qquad = \P\left(M_d^{\chi^r}+(d+1)\P\left(\chi_0^r>\epsilon\delta^{2(1+\eta)}\right)\leq (d+1)/4\right),
\end{align*}
where $M_d^{\chi^r} = \sum_{k=0}^d\id_{[\chi^r_j>\epsilon\delta^{2(1+\eta)}]}-(d+1)\P(\chi_0^r>\epsilon\delta^{2(1+\eta)})$. By Lemma \ref{Lemma 3 in SRPT}, we have\\ $\P\left(M_d^{\chi^r}+(d+1)\P\left(\chi_0^r>\epsilon\delta^{2(1+\eta)}\right)\leq (d+1)/4\right)\leq \P(M_d^{\chi^r}\leq-(d+1)/4)$. Observe that $\{M_d^{\chi^r}\}_{d\in\Z_+}$ is a martingale with respect to the filtration $\{\mathcal{H}^r_t\}_{t\in\Z_+}$, where $\{\mathcal{H}^r_t\}_{t\in\Z_+}$ is as in the proof of Lemma \ref{new lemma for biased random walk}. Then by the Azuma-Hoeffding Inequality in Proposition \ref{Azuma Hoeffding} with $c_i=1$ for $i\in\N$ and $\gamma=(d+1)/4$, we have 
\[
\P(M_d^{\chi^r}\leq-(d+1)/4)\leq\exp{\left(-\frac{2\times((d+1)/4)^2}{\sum_{k=0}^d1^2}\right)}=\exp{(-(d+1)/8)}.
\]
\end{proof}
\end{cor}

We are now ready to put the results in these technical lemmas together to prove Lemma \ref{Lemma 18 of paper}.
\begin{proof}[Proof of Lemma \ref{Lemma 18 of paper}]
Fix $\eta\in(0,\eta_0)$. Let $\widetilde{r}\left(\eta,\E[v]\wedge4\epsilon_*/3\right)$, $M_*(\eta)$ and $\delta_*(\eta)$ be as in Lemma \ref{a lot of assumptions}. 
Let $0<\delta_1(\eta)\leq\delta_*(\eta)$ be such that $T<\epsilon_*\delta_1(\eta)^{-2(1+\eta)}/4$. 
Fix $r\geq \widetilde{r}\left(\eta,\E[v]\wedge4\epsilon_*/3\right)$ and $\delta\in[M_*(\eta)/c^r,\delta_1(\eta)]$. By the choice of $\delta$, Proposition \ref{tilde QQ} with $x=\delta/2$ and $y=\delta$, \eqref{ineq:Nr}, and Lemma \ref{definition of tau}, we have 
\begin{align}
    &\P\left(\|\widetilde{\ZZZ}^r_\delta-\widetilde{\ZZZ}^r_{\delta/2}\|_T>96B\delta^\eta\log(1/\delta)+\frac{c^r}{r}\right)\nonumber\\
    &\qquad \leq\P\left(\|\widetilde{\ZZZ}^r_\delta-\widetilde{\ZZZ}^r_{\delta/2}\|_{\epsilon_*\delta^{-2(1+\eta)}/4}>96B\delta^\eta\log(1/\delta)+\frac{c^r}{r}\right)\nonumber\\
    &\qquad \leq \P\left(\|\widetilde{Y}_\delta^r\|_{\epsilon_*\delta^{-2(1+\eta)}/4}>48B\delta^{1+\eta}\log(1/\delta)\right)\nonumber\\
    &\qquad \leq D_1\delta^{\eta_0-\eta}+\P\left(\sum_{k=0}^{\lfloor\delta^{-8}\rfloor+1}(\tau_{2k+1}^r-\tau_{2k}^r)\leq \epsilon_*\delta^{-2(1+\eta)}/4\right).\label{SRPT Lemma 18 5.94}
\end{align}
Since $\delta^8\leq\delta^{4(1+\eta)}$, we have that $\delta^{8}\epsilon_*\delta^{-2(1+\eta)}\leq \epsilon_*\delta^{2(1+\eta)}$. Using this fact and Corollary \ref{SRPT Lemma 18 5.92},
\begin{align}
    &\P\left(\sum_{k=0}^{\lfloor\delta^{-8}\rfloor+1}(\tau_{2k+1}^r-\tau_{2k}^r)\leq \epsilon_*\delta^{-2(1+\eta)}/4\right)\nonumber\\ 
    &\qquad\leq\P\left(\sum_{k=0}^{\lfloor\delta^{-8}\rfloor+1}(\tau_{2k+1}^r-\tau_{2k}^r)\leq \delta^{-8}\epsilon_*\delta^{2(1+\eta)}/4\right)\nonumber\\
    &\qquad\leq \P\left(\sum_{k=0}^{\lfloor\delta^{-8}\rfloor+1}(\tau_{2k+1}^r-\tau_{2k}^r)\leq\left(\lfloor \delta^{-8}\rfloor+2\right)\epsilon_*\delta^{2(1+\eta)}/4\right)\nonumber\\
    &\qquad\leq \exp{\left(-\frac{\delta^{-8}+2}{8}\right)}\leq \exp{\left(-\frac{\delta^{-8}}{8}\right)}\leq 8\delta^{8},\label{1st term of 5.96}
\end{align}
where the last step uses the fact that $x\exp(-x/8)\leq 8$ for all $x\geq1$.
Thus, combining \eqref{SRPT Lemma 18 5.94}, \eqref{1st term of 5.96} and $\delta^8\leq\delta^{\eta_0-\eta},$  we have that \eqref{4.4.16} holds. Since $r\geq\widetilde{r}\left(\eta,\E[v]\wedge4\epsilon_*/3\right)$ and $\delta\in[M_*(\eta)/c^r,\delta_1(\eta)]$ are arbitrary, the result holds for $r(\eta)=\widetilde{r}\left(\eta,\E[v]\wedge4\epsilon_*/3\right),M(\eta)=M_*(\eta)$ and $\delta(\eta)=\delta_1(\eta)$. 
\end{proof}
%--------------Section 6.1.4----------------------
\subsubsection{Application of Lemma \ref{Lemma 18 of paper} to Prove Theorem \ref{In between 0 and 1}}\label{Application of Lemma 5.2.6 to prove Theorem 5.2.1}
We begin with a basic estimate.
%------------------------------------
\begin{lem}\label{a basic estimate before Lemma New}
Suppose $\eta\in(0,\eta_0)$. Let $r(\eta)$, $M(\eta)$ and $\delta(\eta)$ be as in Lemma \ref{Lemma 18 of paper}. For $r\geq r(\eta)$ and $\delta\in[M(\eta)/c^r,\delta(\eta)]$, let $K(\eta,r,\delta)\in\Z_+$ be such that 
\begin{equation}
    2^{-\left(K(\eta,r,\delta)+1\right)}\delta<\frac{M(\eta)}{c^r}\leq2^{-K(\eta,r,\delta)}\delta.\label{M(eta) bounded below and above by expentional of 2}
\end{equation}
Then, for $r\geq r(\eta)$ and $\delta\in[M(\eta)/c^r,\delta(\eta)]$,
\begin{equation}
    K(\eta,r,\delta)<\alpha(\eta)\log{r}+\beta(\eta),\label{theta sub eta of r is bounded below by K plus 1}
    \end{equation}
where $\alpha(\eta)=1/(2(1+\eta)\log{2})$ and $\beta(\eta)=\log\left(\frac{\delta(\eta)}{M(\eta)}\right)/\log{2}$.
Moreover, $\theta_\eta\in\Theta$, where $\theta_\eta(r)=\left(\alpha(\eta)\log{r} + \beta(\eta)\right)\frac{c^r}{r}$ for $r\in R$.
\begin{proof}
Fix $r\geq r(\eta)$ and $\delta\in[M(\eta)/c^r,\delta(\eta)]$. By \eqref{M(eta) bounded below and above by expentional of 2},
$
\log\left(\frac{c^r\delta}{M(\eta)}\right)/\log{2}\geq K(\eta,r,\delta)$.
Since $\delta\leq\delta(\eta)$ and $c^r<r^{\frac{1}{2(1+\eta)}}$ (see \eqref{4.4.7}), \eqref{theta sub eta of r is bounded below by K plus 1} holds.
By writing, $\frac{c^r\log{r}}{r}=\frac{c^r}{r^{1/2}}\cdot\frac{\log{r}}{r^{1/2}}$, noting that $c^r=S^{-1}(r)$, $S^{-1}\in\textbf{R}_0$ and $\log{(\cdot)}\in\textbf{R}_0$, and applying Proposition \ref{subpolynomial growth of slowly varying functions} with $\gamma = 1/2$, $\lim_{r\to\infty}\theta_\eta(r)=0$. So $\theta_\eta\in\Theta$.
\end{proof}
\end{lem}

\begin{lem}\label{Lemma New}
Suppose $T>0$ and $\eta\in(0,\eta_0)$. Let $r(\eta)$, $M(\eta)$ and $\delta(\eta)$ be as in Lemma \ref{Lemma 18 of paper}. Set $D(\eta)=96B\sum_{k=0}^\infty2^{-k(\eta_0-\eta)}(1+k\ln{2})$ and $\widetilde{D}(\eta) = \frac{D}{1-2^{\eta-1}}$ and let $\theta_\eta\in\Theta$ be as in Lemma \ref{a basic estimate before Lemma New}. For all $r\geq r(\eta)$ and $\delta\in[M(\eta)/c^r,\delta(\eta)]$,
\[
\P\left(\|\widetilde{\ZZZ}_\delta^r-\widetilde{\ZZZ}^r_{M(\eta)/c^r}\|_T>D(\eta)\delta^\eta\left[1+\log{(1/\delta)}\right]
+\theta_\eta(r)+\frac{c^r}{r}\right)\leq\widetilde{D}(\eta)\delta^{\eta_0-\eta}.
\]
\begin{proof}
Fix $r\geq r(\eta)$ and $\delta\in[M(\eta)/c^r,\delta(\eta)]$ and let $K(\eta,r,\delta)\in\Z_+$ be such that $2^{-K(\eta,r,\delta)-1}\delta<M(\eta)/c^r\leq 2^{-K(\eta,r,\delta)}\delta$. By a union bound and Lemma \ref{Lemma 18 of paper}, we have
\begin{align}
    &\P\left(\left\|\widetilde{\ZZZ}_\delta^r-\widetilde{\ZZZ}_{M(\eta)/c^r}^r\right\|_T>D(\eta)\delta^{\eta}\left[1+\log{(1/\delta)}\right]
    +\theta_\eta(r)+\frac{c^r}{r}\right)\nonumber\\
    &\qquad \leq \P\left(\left\|\widetilde{\ZZZ}_\delta^r-\widetilde{\ZZZ}_{2^{-K(\eta,r,\delta)-1}\delta}^r\right\|_T>D(\eta)\delta^{\eta}\left[1+\log{(1/\delta)}\right]
    +\theta_\eta(r)+\frac{c^r}{r}\right)\nonumber\\  
    &\qquad \leq\P\left(\sum_{k=0}^{K(\eta,r,\delta)}\left\|\widetilde{\ZZZ}_{2^{-k}\delta}^r-\widetilde{\ZZZ}_{2^{-k-1}\delta}^r\right\|_T>D(\eta)\delta^\eta\left[1+\log{(1/\delta)}\right]+
    \theta_\eta(r)+\frac{c^r}{r}\right).\nonumber\\
&\qquad\leq\sum_{k=0}^{K(\eta,r,\delta)}\P\left(\left\|\widetilde{\ZZZ}_{2^{-k}\delta}^r-\widetilde{\ZZZ}_{2^{-k-1}\delta}^r\right\|_T>96B(2^{-k}\delta)^{\eta}\log{(2^k/\delta)}+\frac{c^r}{r}\right)\nonumber\\
    &\qquad\leq\sum_{k=0}^{K(\eta,r,\delta)}D(\delta/2^k)^{\eta_0-\eta}\leq D\delta^{\eta_0-\eta}\sum_{k=0}^\infty(1/2^{\eta_0-\eta})^k = \widetilde{D}(\eta)\delta^{\eta_0-\eta},\nonumber
\end{align}
which completes the proof.
\end{proof}
\end{lem}
%------------------------------------
\begin{lem}\label{Lemma New 2}
Suppose $T>0$ and $\eta\in(0,\eta_0)$. Let $r(\eta)$, $M(\eta)$ and $\delta(\eta)$ be as in Lemma \ref{Lemma 18 of paper}, $D(\eta)$ and $\widetilde{D}(\eta)$ be as in Lemma \ref{Lemma New}, $\theta_\eta\in\Theta$ be as in Lemma \ref{a basic estimate before Lemma New}, and $\vartheta_\eta\in\Theta$ be as in Proposition \ref{Lemma 16 in SRPT} with $M=M(\eta)$. Then, for all $r\geq r(\eta)$ and $\delta\in[M(\eta)/c^r,\delta(\eta)]$,
\[
\P\left(\|\widetilde{\ZZZ}_\delta^r\|_T>D(\eta)\delta^\eta\left[1+\log(1/\delta)\right]+
\theta_\eta(r)+\vartheta_\eta(r)+\frac{c^r}{r}\right)\leq \widetilde{D}(\eta)\delta^{\eta_0-\eta}+\vartheta_\eta(r).
\]
\begin{proof}
Observe that for any $r\in R$ and $\delta>0$,
\begin{align*}
    &\P\left(\|\widetilde{\ZZZ}_\delta^r\|_T>D(\eta)\delta^{\eta}\left[1+\log(1/\delta)\right]+\theta_\eta(r)
    +\vartheta_\eta(r)+\frac{c^r}{r}\right)\\
    &\qquad\leq\P\left(\|\widetilde{\ZZZ}_\delta^r-\widetilde{\ZZZ}^r_{M(\eta)/c^r}\|_T>D(\eta)\delta^\eta\left[1+\log(1/\delta)\right]+\theta_\eta(r)+\frac{c^r}{r}
    \right)\\
    &\qquad\qquad+\P\left(\|\widetilde{\ZZZ}_{M(\eta)/c^r}^r\|_T>\vartheta_\eta(r)\right).
\end{align*}
Thus the result follows from Lemma \ref{Lemma New} and Proposition \ref{Lemma 16 in SRPT}.
\end{proof}
\end{lem}

\begin{cor}\label{in between 0 and 1}
Suppose $T>0$ and $\eta\in(0,\eta_0)$. Let $r(\eta)$, $M(\eta)$ and $\delta(\eta)$ be as in Lemma \ref{Lemma 18 of paper}, $D(\eta)$ and $\widetilde{D}(\eta)$ be as in Lemma \ref{Lemma New}, $\theta_\eta\in\Theta$ be as in Lemma \ref{a basic estimate before Lemma New}, and $\vartheta_\eta\in\Theta$ be as in Proposition \ref{Lemma 16 in SRPT} with $M=M(\eta)$. Then, for all $r\geq r(\eta)$ and $\delta\in[M(\eta)/c^r,\delta(\eta)]$,
\[
\P\left(\|\widetilde{\QQ}_\delta^r\|_T>D(\eta)\delta^\eta[1+\log{(1/\delta)}]+\theta_\eta(r)+\vartheta_\eta(r)+\frac{2c^r}{r}\right)\leq \widetilde{D}(\eta)\delta^{\eta_0-\eta}+\vartheta_\eta(r).
\]
\begin{proof}
This follows from Lemma \ref{Lemma New 2} and \eqref{Q tilde theta and Z tilde theta}.
\end{proof}
\end{cor}

\begin{proof}[Proof of Theorem \ref{In between 0 and 1}]
Fix $a\in(0,1)$.  For all $t\in[0,T]$, $r\in R$, and $\delta\in(0,a)$,\\
$\delta[\widetilde{\QQ}_a^r(t)-\widetilde{\QQ}_{\delta}^r(t)]\leq \widetilde{W}_a^r(t)-\widetilde{W}_{\delta}^r(t)\leq \widetilde{W}_a^r(t)$.
Hence, for all $r\in R$ and $\delta\in(0,a)$,
\begin{equation}
\|\widetilde{\QQ}_a^r\|_T
\leq \|\widetilde{\QQ}_a^r-\widetilde{\QQ}_\delta^r\|_T+\|\widetilde{\QQ}_\delta^r\|_T
\leq\delta^{-1}\|\widetilde{W}_a^r\|_T+\|\widetilde{\QQ}^r_{\delta}\|_T.    \label{upper bound of tilde Q a r by W and Q}
\end{equation}
Let $\theta\in(0,1)$ and $\epsilon\in(0,1)$. Also, let $D(\eta)$ and $\widetilde{D}(\eta)$ be as in Lemma \ref{Lemma New}, let $\theta_\eta\in \Theta$ be as in Lemma \ref{a basic estimate before Lemma New}, and let $\vartheta_\eta\in\Theta$ be as in Lemma \ref{Lemma New 2}. Observe that
$
\lim_{\delta\to0+}D(\eta)\delta^\eta\left(1+\log{(1/\delta)}\right)=0\text{ and }\lim_{\delta\to0+}\widetilde{D}(\eta)\delta^{\eta_0-\eta}=0$.
Hence, there exists $\widehat{\delta}\in(0,a)$ such that 
\[
D(\eta)\delta^\eta(1+\log{(1/\delta)})<\theta/3
\qquad\text{ and }\qquad\widetilde{D}(\eta)\delta^{D/2}<\epsilon/2\quad\forall \delta\in(0,\widehat{\delta}],
\]
and $\widehat{r}\in R$ such that 
\[
\theta_\eta(r)+\vartheta_\eta(r)+\frac{2c^r}{r}<\theta/3
\qquad\text{ and }\qquad\vartheta_\eta(r)<\epsilon/2\quad \forall r\geq \widehat{r}.
\]
Then, for $\delta\in(0,\widehat{\delta})$ and $r\geq \widehat{r}$,
\begin{align}
\P\left(\|\widetilde{\QQ}_\delta^r\|_T>2\theta/3\right)&\leq \P\left(\|\widetilde{\QQ}_\delta^r\|_T>D(\eta)\delta^\eta(1+\log{(1/\delta)})+\theta_\eta(r)+\vartheta_\eta(r)+\frac{2c^r}{r}\right), \label{proof of 5.2.1 with probability on tilde Q}   \\
\vartheta_\eta(r)+\widetilde{D}(\eta)\delta^{\eta_0-\eta}&<\epsilon\label{proof of 5.2.1 with vartheta tilde D and epsilon}
\end{align}
Let $r(\eta)$, $M(\eta)$ and $\delta(\eta)$ be as in Lemma \ref{Lemma New}. By the choice of $\widehat{\delta}$ and $\widehat{r}$, \eqref{proof of 5.2.1 with probability on tilde Q}, Corollary \ref{in between 0 and 1}, and \eqref{proof of 5.2.1 with vartheta tilde D and epsilon}, for $r\geq r(\eta)\vee\widehat{r}$ and $\delta\in[M(\eta)/c^r,\delta(\eta)\wedge\widehat{\delta}]$, we have 
\[
\P\left(\|\widetilde{\QQ}_\delta^r\|_T>2\theta/3\right)<\epsilon.
\]
From this, \eqref{upper bound of tilde Q a r by W and Q} and a union bound, it follows that for $r\geq r(\eta)\vee\widehat{r}$ and $\delta\in[M(\eta)/c^r,\delta(\eta)\wedge\widehat{\delta}\wedge a)$,
\[
\P\left(\|\widetilde{\QQ}_a^r\|_T>\theta\right)\leq \P\left(\frac{1}{\delta}\|\widetilde{W}_a^r\|_T>\frac{\theta}{3}\right)+\epsilon.
\]
This together with Lemma \ref{a<1 convergence} yields that
\[
\limsup_{r\to\infty}\P\left(\|\widetilde{\QQ}_a^r\|_T>\theta\right)\leq \epsilon.
\]
But $\epsilon\in(0,1)$ was arbitrary. So letting $\epsilon\to0+$ completes the proof.
\end{proof}

%---------------Section 6.2--------------------
\subsection{Convergence of the Measured-Valued Processes}\label{section 5.3}
\noindent In order to prove convergence in distribution of the measure valued process $\widetilde{\mathcal Q}^r(\cdot)$ as $r\to\infty$ to the measure valued process $W^*(\cdot)\delta_1^+$ and prove Theorem \ref{mv convergence}, it suffices to show that $\{\widetilde{\mathcal Q}^r(\cdot)\}_{r\in R}$ is tight and then to show that any convergent subsequence converges in distribution to $W^*(\cdot)\delta_1^+$. In order to verify tightness, we use Jakubowski's criteria (see \cite[Theorem 10]{Jakubowski}). Specifically, in order to show that $\{\widetilde{\mathcal Q}^r(\cdot)\}_{r\in R}$ is tight, it suffices to show that for each $T>0$ and $\eta\in(0,1)$

\vspace{0.1cm}

\begin{enumerate}
    \item[(J1)] there exists a compact set ${\mathcal K}_T^\eta\subseteq {\mathcal M}$
    such that
    \[
    \liminf_{r\to\infty}{\mathbb P}\left( \widetilde{\mathcal Q}^r(t)\in {\mathcal K}_T^\eta\hbox{ for all }t\in[0,T]\right)\ge 1-\eta,
    \]
    \item[(J2)] and the sequence $\{\langle g,\widetilde{\mathcal Q}^r(\cdot)\rangle\}_{r\in R}$ of real valued process  is tight for each $g\in \C_b({\mathbb R}_+)$.
\end{enumerate}
Condition (J1) is often referred to as the compact containment condition.  It is verified in Section \ref{compact containment} below. Condition (J2) requires a verifying tightness of certain real valued processes.  We will verify the slightly stronger condition of {\bf C}-tightness, which implies tightness as well as the continuity of all limit points. By \cite[Theorem 7.2 in Chapter 3]{Ethier and Kurtz}, for $g\in\C_b({\mathbb R}_+)$, $\{\langle g,\widetilde{\mathcal Q}^r(\cdot)\rangle\}_{r\in R}$ is {\bf C}-tight if for all $T>0$ and $\eta\in(0,1)$

\begin{enumerate}
    \item[(T1)] there exists $M>0$ (depending on $g$, $T$ and $\eta$) such that 
    $$
    \liminf_{r\to\infty}{\mathbb P}\left( \|\langle g,\widetilde{\mathcal Q}^r(\cdot)\rangle\|_T\le M\right)\ge 1-\eta,
    $$
    \item[(T2)] and there exists $h>0$ (depending on $g$, $T$ and $\eta$) such that
    $$
    \liminf_{r\to\infty}{\mathbb P}\left( \sup_{t\in[0,T-h]}\sup_{s\in[0,h]}\left| \langle g,\widetilde{\mathcal Q}^r(t+s)\rangle-\langle g,\widetilde{\mathcal Q}^r(t)\rangle\right|<\eta\right)\ge 1-\eta.
    $$
\end{enumerate}
We verify (T1) and (T2) for all $g\in\C_b(\R_+)$ in the proof of Lemma \ref{tight} in Section \ref{verification of T1 and T2} below.

%-------------Section 6.2.1---------------------
\subsubsection{Compact Containment of $\{\widetilde{\mathcal Q}^r(\cdot)\}_{r\in R}$}\label{compact containment}
\begin{lem}\label{Lemma D} As $r\to\infty$,
$\widetilde{\QQ}^r-\widetilde{W}^r\overset{p}{\to}0$.
\begin{proof}
Fix $\epsilon\in(0,1/2)$, $T>0$, $\eta\in(0,1)$, and $\delta\in(0,1)$.
For each $r\in R$ and $t\in[0,T]$,
\begin{equation}\label{eps_decomp}
\widetilde{\QQ}^r(t)= 
\widetilde{\QQ}_{1-\varepsilon}^r(t)
+\left(\widetilde{\QQ}_{1+\varepsilon}^r(t)-\widetilde{\QQ}_{1-\varepsilon}^r(t)\right)
+ \left(\widetilde{\QQ}^r(t)-\widetilde{\QQ}_{1+\varepsilon}^r(t)\right).
\end{equation}
Note that
$(1-\epsilon)\id_{(1-\epsilon,1+\epsilon]}\leq \chi\id_{(1-\epsilon,1+\epsilon]}\leq (1+\epsilon)\id_{(1-\epsilon,1+\epsilon]}$. Thus, for each $r\in R$ and $t\in[0,T]$, 
\[
(1-\epsilon)
\left(\widetilde{\QQ}_{1+\varepsilon}^r(t)-\widetilde{\QQ}_{1-\varepsilon}^r(t)\right)
\leq
\widetilde{W}_{1+\varepsilon}^r(t)-\widetilde{W}_{1-\varepsilon}^r(t)
\leq (1+\epsilon)\left(\widetilde{\QQ}_{1+\varepsilon}^r(t)-\widetilde{\QQ}_{1-\varepsilon}^r(t)\right),
\]
which implies that
\begin{equation}
   \frac{1}{1+\epsilon}\left(\widetilde{W}_{1+\varepsilon}^r(t)-\widetilde{W}_{1-\varepsilon}^r(t)\right)
   \leq \widetilde{\QQ}_{1+\varepsilon}^r(t)-\widetilde{\QQ}_{1-\varepsilon}^r(t)
   \leq\frac{1}{1-\epsilon}\left(\widetilde{W}_{1+\varepsilon}^r(t)-\widetilde{W}_{1-\varepsilon}^r(t)\right).\label{1 over 1 + epsilon and 1 over 1 - epsilon}
\end{equation}
For $r\in R$, define
\begin{align*}
   \Omega_1^r(\epsilon) &:= \left\{\|\widetilde{\QQ}_{1-\varepsilon}^r\|_T\leq\delta/3\right\}\bigcap \left\{\|\widetilde{W}_{1-\varepsilon}^r\|_T\leq\delta/3\right\}\\
  \Omega_2^r(\epsilon) &:=\left\{\|\widetilde{\QQ}^r-\widetilde{\QQ}_{1+\varepsilon}^r\|_T\leq\delta/3\right\}\bigcap \left\{\|\widetilde{W}^r-\widetilde{W}_{1+\varepsilon}^r\|_T\leq\delta/3\right\}.
\end{align*} 
By Corollary \ref{cor:Tailtozero}, Lemma \ref{a<1 convergence} and Theorem \ref{In between 0 and 1}, we have
\begin{equation}
\lim_{r\to\infty}\P\left(\Omega_1^r(\epsilon)\cap\Omega_2^r(\epsilon)\right)=1.    \label{intersection of Omega}
\end{equation}
By \eqref{eps_decomp} and \eqref{1 over 1 + epsilon and 1 over 1 - epsilon}, for each $r\in R$ and $t\in[0,T]$ on $\Omega_1^r(\epsilon)\cap\Omega_2^r(\epsilon)$, we have
\begin{align*}
\widetilde{\QQ}^r(t)
&\leq\frac{\delta}{3}+\frac{1}{1-\epsilon}\widetilde{W}^r(t)+\frac{\delta}{3} = \frac{1}{1-\epsilon}\widetilde{W}^r(t)+\frac{2\delta}{3}.
\end{align*}
Similarly, for each $r\in R$ and $t\in[0,T]$ on $\Omega_1^r(\epsilon)\cap\Omega_2^r(\epsilon)$, we have 
\begin{align*}
\widetilde{W}^r(t)&=\widetilde{W}_{1-\varepsilon}^r(t) + \widetilde{W}_{1+\varepsilon}^r(t)-\widetilde{W}_{1-\varepsilon}^r(t)+\widetilde{W}^r(t)-\widetilde{W}_{1+\varepsilon}^r(t)\\
&\leq \frac{\delta}{3}+(1+\epsilon)\widetilde{\QQ}^r(t)+\frac{\delta}{3} = (1+\epsilon)\widetilde{\QQ}^r(t)+\frac{2\delta}{3},
\end{align*}
which implies that
\[
\widetilde{\QQ}^r(t)\geq\frac{1}{1+\epsilon}\widetilde{W}^r(t)-\frac{2\delta}{3(1+\epsilon)}\geq \frac{1}{1+\epsilon}\widetilde{W}^r(t)-\frac{2\delta}{3}.
\]
Hence, for each $r\in R$ and $t\in[0,T]$ on $\Omega_1^r(\epsilon)\cap\Omega_2^r(\epsilon)$, we have
\[
\frac{1}{1+\epsilon}\widetilde{W}^r(t)-\frac{2\delta}{3}\leq\widetilde{\QQ}^r(t)\leq \frac{1}{1-\epsilon}\widetilde{W}^r(t)+\frac{2\delta}{3}.
\]
Subtracting $\widetilde{W}^r(t)$ from each side of the above inequality and using the fact that $\epsilon\in(0,1/2)$, for each $r\in R$ and $t\in[0,T]$ on $\Omega_1^r(\epsilon)\cap\Omega_2^r(\epsilon)$, we have
\begin{equation}
   -\epsilon\widetilde{W}^r(t)-\frac{2\delta}{3}\leq\widetilde{\QQ}^r(t)-\widetilde{W}^r(t)\leq 2\epsilon\widetilde{W}^r(t)+\frac{2\delta}{3}.\label{difference of Q and W}
\end{equation}
Given $M\in\N$, let 
\[
\Omega^r(M) := \left\{\|\widetilde{W}^r\|_T<M\right\}\quad\text{and}\quad \Omega(M):=\left\{\left\|W^*\right\|_T<M\right\}.
\]
Since $W^*$ is continuous almost surely, 
$\P\left(\bigcup_{M\in\N}\Omega(M)\right)=1$.
Hence, there exists $M_\eta\in\N$ such that $
\P(\Omega(M_\eta))\geq 1-\eta$.
Since $\left\{f
\in\D([0,\infty),\R_+) : \|f\|_T<M_\eta \right\}$ is open in the $J_1$-topology and $\widetilde{W}^r(\cdot)\Rightarrow W^*(\cdot)$ as $r\to\infty$ (see \eqref{diffusion of W}), by the Portmanteau theorem,
$\liminf_{r\to\infty}\P(\Omega^r(M_\eta))\geq 1-\eta$.  Thus,
\begin{equation}
    \liminf_{r\to\infty}\P\left(\Omega_1^r(\epsilon)\cap\Omega_2^r(\epsilon)\cap\Omega^r(M_\eta)\right)\geq 1-\eta.\label{intersection of three Omegas}
\end{equation}
Furthermore, by \eqref{difference of Q and W}, for each $r\in R$, $
\|\widetilde{\QQ}^r-\widetilde{W}^r\|_T\leq 2\epsilon M_\eta+2\delta/3$ on $\Omega_1^r(\epsilon)\cap\Omega_2^r(\epsilon)\cap\Omega^r(M_\eta)$.
Take $\epsilon = \frac{1}{2}\left(\frac{\delta}{3M_\eta}\wedge1\right)$. Thus, for each $r\in R$, $\|\widetilde{\QQ}^r-\widetilde{W}^r\|_T\leq\delta$ on $\Omega_1^r(\epsilon)\cap\Omega_2^r(\epsilon)\cap\Omega^r(M_\eta)$.
Hence, by \eqref{intersection of three Omegas}, we have 
$\liminf_{r\to\infty}\P\left(\|\widetilde{\QQ}^r-\widetilde{W}^r\|_T\leq \delta\right)\geq 1-\eta$.
Since $T>0$, $\eta\in(0,1)$ and $\delta\in(0,1)$ are chosen arbitrarily, the proof is complete.
\end{proof}
\end{lem}

\begin{cor}\label{QQ to W}
As $r\to\infty$, $\widetilde{\QQ}^r(\cdot)\Rightarrow W^*(\cdot)$.
\begin{proof}
For each $r\in R$ and $t\geq0$, $\widetilde{\QQ}^r(t) = \widetilde{\QQ}^r(t) - \widetilde{W}^r(t)+\widetilde{W}^r(t)$. The result follows from this, Lemma \ref{Lemma D}, \eqref{diffusion of W} and the convergence together theorem.\end{proof}
\end{cor}

\begin{lem}\label{Lemma E}
For all $a>1$, we have $\widetilde{\QQ}_a^r(\cdot)\Rightarrow W^*(\cdot)$ as $r\to\infty$. 
\begin{proof}
The result follows from Corollaries \ref{cor:Tailtozero} and \ref{QQ to W}, and the convergence together theorem.
\end{proof}
\end{lem}
For any $M>0$, we define $\K(M):=\left\{\zeta\in\M:\langle\id,\zeta\rangle\vee\langle\chi,\zeta\rangle\leq M\right\}$. For each $M>0$, $\K(M)$ is a relatively compact subset of $\M$ (see \cite{Prokhorov}).
\begin{lem}\label{J1}
The sequence of processes $\{\widetilde{\Q}^r(\cdot)\}_{r\in R}$ satisfies (J1).
\begin{proof}
Fix $T>0$ and $\eta\in(0,1)$.  It suffices to show that there exists $M_{T}^\eta>0$ such that 
\begin{equation}\label{Qtight}
\liminf_{r\to\infty}\P\left(\widetilde{\Q}^r(t)\in\K(M_{T}^\eta)\text{ for all }t\in[0,T]\right)\geq 1-\eta.
\end{equation}
As $r\to\infty$, $\widetilde{\QQ}^r(\cdot)\Rightarrow W^*(\cdot)$ by Corollary \ref{QQ to W} and $\widetilde{W}^r(\cdot)\Rightarrow W^*(\cdot)$ by \eqref{diffusion of W}. Thus, $\{\widetilde{\QQ}^r(\cdot)\}_{r\in R}$ and $\{\widetilde{W}^r(\cdot)\}_{r\in R}$ are tight. Hence, $M_{T}^\eta>0$ exists such that 
\[
\liminf_{r\to\infty}\P\left(\|\widetilde{\QQ}^r\|_T\leq M_{T}^\eta\right)\geq 1-\eta/2\quad\text{and}\quad
\liminf_{r\to\infty}\P\left(\|\widetilde{W}^r\|_T\leq M_{T}^\eta\right)\geq 1-\eta/2,
\]
which implies \eqref{Qtight}.
\end{proof}
\end{lem}
%------------------Section 6.2.2-----------------------
\subsubsection{Verification of (T1) and (T2)}\label{verification of T1 and T2}

\begin{lem}\label{tight}
Suppose $g\in\C_b(\R_+)$.  Then $\{\langle g ,\widetilde{\mathcal{Q}}^r(\cdot)\rangle\}_{r\in R}$ is {\bf C}-tight.
\begin{proof} The result holds trivially if $\Vert g\Vert_{\infty}=0$.
Henceforth we assume that $\Vert g\Vert_{\infty}>0$.
Fix $T>0$ and $\eta\in(0,1)$.
We begin by noting that
$\|\langle g ,\widetilde{\mathcal{Q}}^r(\cdot)\rangle\|_T\le \Vert g\Vert_{\infty}\|\widetilde{\QQ}^r\|_T$,
In addition, by Lemma \ref{Lemma D}, $\{{\widetilde{\QQ}}^r(\cdot)\}_{r\in R}$ is tight.
Hence, there exists $M>0$ (depending on $g$, $T$ and $\eta$) 
\begin{equation}
\liminf_{r\to\infty}{\mathbb P}\left( \|\langle g ,\widetilde{\mathcal{Q}}^r(\cdot)\rangle\|_T\le M\right)\ge 1-\eta/3,\label{tilde QQ M bounded above by 1 - eta/3}
\end{equation}
and so $\{\langle g ,\widetilde{\mathcal{Q}}^r(\cdot)\rangle\}_{r\in R}$ satisfies (T1).

Next we show that (T.2) holds. Let $M$ be as in \eqref{tilde QQ M bounded above by 1 - eta/3} and set $\delta=\eta/8M$.  There exists $\varepsilon\in(0,1)$ such that for all $\vert x-1\vert<\varepsilon$ we have $\vert g(x)-g(1)\vert<\delta$. Fix such an $\varepsilon\in(0,1)$. For each $r\in R$ and $t\ge 0$,
$$
\langle g ,\widetilde{\mathcal{Q}}^r(t)\rangle
=
\langle g \id_{[0,1-\varepsilon]} ,\widetilde{\mathcal{Q}}^r(t)\rangle
+
\langle g \id_{(1-\varepsilon,1+\varepsilon]} ,\widetilde{\mathcal{Q}}^r(t)\rangle
+
\langle g \id_{(1+\varepsilon,\infty)},\widetilde{\mathcal{Q}}^r(t)\rangle.
$$
Hence, for each $r\in R$ and $t,s\ge 0$,
\begin{eqnarray}\label{eq:3_parts}
\langle g ,\widetilde{\mathcal{Q}}^r(t+s)\rangle&-&\langle g ,\widetilde{\mathcal{Q}}^r(t)\rangle
=
\left(\langle g \id_{[0,1-\varepsilon]} ,\widetilde{\mathcal{Q}}^r(t+s)\rangle
-
\langle g \id_{[0,1-\varepsilon]} ,\widetilde{\mathcal{Q}}^r(t)\rangle\right)\label{eq:part1}\\
&+&
\left(\langle g \id_{(1-\varepsilon,1+\varepsilon]} ,\widetilde{\mathcal{Q}}^r(t+s)\rangle-\langle g \id_{(1-\varepsilon,1+\varepsilon]} ,\widetilde{\mathcal{Q}}^r(t)\rangle\right)\label{eq:part2}\\
&+&
\left(\langle g \id_{(1+\varepsilon,\infty)},\widetilde{\mathcal{Q}}^r(t+s)\rangle-\langle g \id_{(1+\varepsilon,\infty)},\widetilde{\mathcal{Q}}^r(t)\rangle\right).\label{eq:part3}
\end{eqnarray}
We will work with each of the three types terms above separately.

To begin we consider the terms on the right side in lines \eqref{eq:part1} and \eqref{eq:part3}.  For all $r\in R$, we have
\begin{eqnarray*}
\|
\langle g \id_{[0,1-\varepsilon]} ,\widetilde{\mathcal{Q}}^r(\cdot)\rangle\|_T
&\le& 
\Vert g\Vert_{\infty}\|
\langle \id_{[0,1-\varepsilon]} ,\widetilde{\mathcal{Q}}^r(\cdot)\rangle\|_T\\
\|
\langle g \id_{(1+\varepsilon,\infty)} ,\widetilde{\mathcal{Q}}^r(\cdot)\rangle\|_T
&\le& 
\Vert g\Vert_{\infty}\|
\langle \id_{(1+\varepsilon,\infty)} ,\widetilde{\mathcal{Q}}^r(\cdot)\rangle\|_T
\end{eqnarray*}
This together with Theorem \ref{In between 0 and 1} and Corollary \ref{cor:Tailtozero} gives that as $r\to\infty$.
$$
\|\langle g \id_{[0,1-\varepsilon]} ,\widetilde{\mathcal{Q}}^r(\cdot)\rangle\|_T\overset{p}{\to} 0
\qquad\hbox{and}\qquad
\|\langle g \id_{(1+\varepsilon,\infty)} ,\widetilde{\mathcal{Q}}^r(\cdot)\rangle\|_T\overset{p}{\to} 0.
$$
Hence,
$
\liminf_{r\to\infty}{\mathbb P}\left(\|
\langle g \id_{[0,1-\varepsilon]} ,\widetilde{\mathcal{Q}}^r(\cdot)\rangle\|_T\vee\|
\langle g \id_{(1+\varepsilon,\infty)} ,\widetilde{\mathcal{Q}}^r(\cdot)\rangle\|_T\le \eta/8\right)=1$.
For $r\in R$, let
$$
B^r_T:=\left\{ \|\widetilde{\QQ}^r\|_T\le M\right\}\cap 
\left\{ \|\langle g \id_{[0,1-\varepsilon]} ,\widetilde{\mathcal{Q}}^r(\cdot)\rangle\|_T\vee\|\langle g \id_{(1+\varepsilon,\infty)]} ,\widetilde{\mathcal{Q}}^r(\cdot)\rangle\|_T\le \eta/8\right\}.
$$
Then, by the above and \eqref{tilde QQ M bounded above by 1 - eta/3},
\begin{equation}\label{eq:B}
\liminf_{r\to\infty}{\mathbb P}\left( B_T^r\right)\ge 1-\eta/3.
\end{equation}

Next we work with the term in \eqref{eq:part2}.
For each $r\in R$ and $t\ge 0$, by the choice of $\varepsilon$,
$$
(g(1)-\delta)
\langle \id_{(1-\varepsilon,1+\varepsilon]} ,\widetilde{\mathcal{Q}}^r(t)\rangle
\le 
\langle g \id_{(1-\varepsilon,1+\varepsilon]} ,\widetilde{\mathcal{Q}}^r(t)\rangle
\le (g(1)+\delta)
\langle \id_{(1-\varepsilon,1+\varepsilon]} ,\widetilde{\mathcal{Q}}^r(t)\rangle.
$$
And so, on the event $\left\{ \|\widetilde{\QQ}^r\|_T\le M\right\}$, by the choice of $\delta$ and $\varepsilon$,
for any $r\in R$ and $s,t\ge 0$ such that $s+t\leq T$,
\begin{eqnarray*}
&&\langle g \id_{(1-\varepsilon,1+\varepsilon]} ,\widetilde{\mathcal{Q}}^r(t+s)\rangle
-
\langle g \id_{(1-\varepsilon,1+\varepsilon]} ,\widetilde{\mathcal{Q}}^r(t)\rangle\\
&&\qquad\le (g(1)+\delta)
\langle \id_{(1-\varepsilon,1+\varepsilon]} ,\widetilde{\mathcal{Q}}^r(t+s)\rangle
-
(g(1)-\delta)
\langle \id_{(1-\varepsilon,1+\varepsilon]} ,\widetilde{\mathcal{Q}}^r(t)\rangle\\
&&\qquad \le g(1)\left(\langle \id_{(1-\varepsilon,1+\varepsilon]} ,\widetilde{\mathcal{Q}}^r(t+s)\rangle-
\langle \id_{(1-\varepsilon,1+\varepsilon]} ,\widetilde{\mathcal{Q}}^r(t)\rangle\right)+\delta\left(\widetilde{\QQ}^r(t+s)+\widetilde{\QQ}^r(t)\right)\\
&&\qquad \le g(1)\left(\langle \id_{(1-\varepsilon,1+\varepsilon]} ,\widetilde{\mathcal{Q}}^r(t+s)\rangle-
\langle \id_{(1-\varepsilon,1+\varepsilon]} ,\widetilde{\mathcal{Q}}^r(t)\rangle\right)
+\eta/4\\
&&\qquad = g(1)\left(\widetilde{\QQ}_{1+\epsilon}^r(t+s)-\widetilde{\QQ}^r_{1+\epsilon}(t)\right)-g(1)\left(\widetilde{\QQ}_{1-\epsilon}^r(t+s)-\widetilde{\QQ}_{1-\epsilon}^r(t)\right)+\eta/4.
\end{eqnarray*}
Similarly, on $\left\{ \|\widetilde{\QQ}^r\|_T\le M\right\}$, for any $r\in R$ and $s,t\ge 0$ such that $s+t\leq T$ 
\begin{eqnarray*}
&&\langle g \id_{(1-\varepsilon,1+\varepsilon]} ,\widetilde{\mathcal{Q}}^r(t+s)\rangle
-
\langle g \id_{(1-\varepsilon,1+\varepsilon]} ,\widetilde{\mathcal{Q}}^r(t)\rangle\\
&&\qquad \geq g(1)\left(\widetilde{\QQ}_{1+\epsilon}^r(t+s)-\widetilde{\QQ}^r_{1+\epsilon}(t)\right)-g(1)\left(\widetilde{\QQ}_{1-\epsilon}^r(t+s)-\widetilde{\QQ}_{1-\epsilon}^r(t)\right)-\eta/4.
\end{eqnarray*}
Thus, on $\left\{ \|\widetilde{\QQ}^r\|_T\le M\right\}$, for any $r\in R$ and $h\in(0,T)$,
\begin{eqnarray}
&&\sup_{t\in[0,T-h]}\sup_{s\in[0,h]}\lvert \langle g \id_{(1-\varepsilon,1+\varepsilon]} ,\widetilde{\mathcal{Q}}^r(t+s)\rangle
-
\langle g \id_{(1-\varepsilon,1+\varepsilon]} ,\widetilde{\mathcal{Q}}^r(t)\rangle\rvert\nonumber\\
&&\quad\le
\lvert g(1)\rvert \sup_{t\in[0,T-h]}\sup_{s\in[0,h]}
 \lvert\widetilde{\QQ}_{1+\varepsilon}^r(t+s)-\widetilde{\QQ}_{1+\varepsilon}^r(t)\rvert\label{eq:part2Est}\\
 &&\qquad+
 \lvert g(1)\rvert\sup_{t\in[0,T-h]}\sup_{s\in[0,h]}
 \lvert\widetilde{\QQ}_{1-\varepsilon}^r(t+s)-\widetilde{\QQ}_{1-\varepsilon}^r(t)\rvert
+\eta/4.\nonumber
\end{eqnarray}
For $r\in R$ and $h\in(0,T)$, let
\begin{eqnarray*}
A_T^r(h)&=&B_T^r\cap
\left\{ \lvert g(1)\rvert\sup_{t\in[0,T-h]}\sup_{s\in[0,h]}
 \lvert\widetilde{\QQ}_{1+\varepsilon}^r(t+s)(\omega)-\widetilde{\QQ}_{1+\varepsilon}^r(t)(\omega)\rvert\le \frac{\eta}{4}\right\}\\
&&\quad \cap\left\{
\lvert g(1)\rvert\sup_{t\in[0,T-h]}\sup_{s\in[0,h]}
 \lvert\widetilde{\QQ}_{1-\varepsilon}^r(t+s)(\omega)-\widetilde{\QQ}_{1-\varepsilon}^r(t)(\omega)\rvert\le \frac{\eta}{4 }
 \right\}.
\end{eqnarray*}
Then \eqref{eq:part2Est} together with \eqref{eq:part1}-\eqref{eq:part3} implies
that for all $r\in R$ and $h\in(0,T)$
\begin{equation}\label{eq:A}
A_T^r(h)\subseteq
\left\{\sup_{t\in[0,T-h]}\sup_{s\in[0,h]} \lvert
\langle g ,\widetilde{\mathcal{Q}}^r(t+s)\rangle-\langle g ,\widetilde{\mathcal{Q}}^r(t)\rangle\rvert\le \eta
\right\}.
\end{equation}
If $g(1)=0$, then $A_T^r(h)=B_T^r$ for all $h\in(0,T)$ and $r\in R$, in which case
(T.2) follows from the above and \eqref{eq:B}. Otherwise, $\lvert g(1)\rvert>0$.
By Theorem \ref{In between 0 and 1} and Lemma \ref{Lemma E}, $\left\{\widetilde{\QQ}_{1-\varepsilon}^r(\cdot)\right\}_{r\in R}$
and
$\left\{\widetilde{\QQ}_{1+\varepsilon}^r(\cdot)\right\}_{r\in R}$
are {\bf C}-tight. Hence, there exists $h\in(0,T)$ such that
\begin{eqnarray*}
\limsup_{r\to\infty}{\mathbb P}\left(
\sup_{t\in[0,T-h]}\sup_{s\in[0,h]}\lvert\widetilde{\QQ}_{1-\varepsilon}^r(t+s)(\omega)-\widetilde{\QQ}_{1-\varepsilon}^r(t)(\omega)\rvert\le \frac{\eta}{8\lvert g(1)\rvert}\right)\ge 1-\eta/3,\\
\limsup_{r\to\infty}{\mathbb P}\left(
\sup_{t\in[0,T-h]}\sup_{s\in[0,h]}\lvert\widetilde{\QQ}_{1+\varepsilon}^r(t+s)(\omega)-\widetilde{\QQ}_{1+\varepsilon}^r(t)(\omega)\rvert\le \frac{\eta}{8\lvert g(1)\rvert}
\right)\ge 1-\eta/3.
\end{eqnarray*}
Thus, for this choice of $h$, the above together with \eqref{eq:B} implies that
$\limsup_{r\to\infty}{\mathbb P}\left(A_T^r(h)\right)\ge 1-\eta$.
In this case, (T.2) follows from the above and \eqref{eq:A}.
\end{proof}
\end{lem}
%-------------Section 6.2.3--------------
\subsubsection{Proof of Theorem \ref{mv convergence}}
\begin{proof}[Proof of Theorem \ref{mv convergence}]
By Lemmas \ref{J1} and \ref{tight}, the sequence $\{\widetilde{\mathcal{Q}}^r(\cdot)\}_{r\in R}$ is tight. Therefore, there exists a convergent subsequence $\{r_j\}_{j\in \N}$. We denote the limit process by $\Q^*(\cdot)$, i.e., $\widetilde{\mathcal{Q}}^{r_j}(\cdot)\Rightarrow\mathcal{Q}^*(\cdot)$ as $j\to\infty$. By re-indexing, we may assume $\widetilde{\Q}^r(\cdot)\Rightarrow \Q^*(\cdot)$ as $r\to\infty$. Also, by the Skorokhod representation theorem, without loss of generality, we may assume that $\widetilde{\Q}^r(\cdot)\to\Q^*(\cdot)$ almost surely as $r\to\infty$.  Then, by Corollary \ref{QQ to W}, $\langle \id ,\Q^*(\cdot) \rangle\overset{d}{=}W^*(\cdot)$. Also, by Theorem \ref{In between 0 and 1} and Lemma \ref{Lemma E}, 
$\langle\id_{[0,a]},\Q^*(\cdot)\rangle \overset{d}{=}0(\cdot)$ for all $a\in[0,1)$ and $\langle\id_{[0,a]},\Q^*(\cdot)\rangle\overset{d}{=}W^*(\cdot)$ for all $a\in(1, \infty)$. Hence,
$\Q^*(\cdot)\overset{d}{=}W^*(\cdot)\delta_1$.
\end{proof}
%-----------------------------------
\begin{appendices}
\section{}
%------------------------------A.1---------------------------------------------
\subsection{Properties of the Skorohod Map}\label{ap:SM}

\begin{prop}[Lipschitz property of $\Gamma$]\label{Lipschitz property of Skorokhod}
For all $f_1,f_2\in \D_0([0,\infty):\R)$ and $T\geq0$, 
$\|\Gamma[f_1]-\Gamma[f_2]\|_T\leq 2\|f_1-f_2\|_T$.
In particular, $\Gamma$ is
continuous in the Skorokhod $J_1$-topology 
at points in $\C(\R_+)$.
\begin{proof}
See proof of Lemma 13.5.1 in \cite{Whitt}.
\end{proof}
\end{prop}

\begin{prop}[Monotonicity Property of $\Gamma$]\label{monotonicity of skorohod}
Suppose $f_1,f_2\in \D_0([0,\infty):\R)$ are such that for all $0\leq s\leq t<\infty$, $f_1(t)-f_1(s)\leq f_2(t)-f_2(s)$ and $f_1(0)\leq f_2(0)$. Then it follows that $\Gamma[f_1](t)\leq \Gamma[f_2](t)$ for all $t\geq0$.
\begin{proof}
    First, we claim that $f_1(t)\leq f_2(t)$ for all $t\geq0$. Fix $t\geq0 $. Letting $s=0$, we have $f_1(t)-f_1(0)\leq f_2(t)-f_2(0)$. Combining this with $f_1(0)\leq f_2(0)$, we have $f_1(t)=f_1(t)-f_1(0)+f_1(0)\leq f_2(t)-f_2(0)+f_2(0)= f_2(t)$. Second, we have $\sup_{0\leq s\leq t}(f_1(t)-f_1(s))\leq\sup_{0\leq s\leq t}(f_2(t)-f_2(s))$, which follows from $f_1(t)-f_1(s)\leq f_2(t)-f_2(s)$ for all $0\leq s\leq t<\infty$ as given in the statement. Thus, using $0\leq f_1(0)\leq f_2(0)$, we have 
\begin{align*}
\Gamma[f_1](t)-\Gamma[f_2](t) &= f_1(t)-\inf_{0\leq s\leq t}(f_1(s)\wedge0)-f_2(t)+\inf_{0\leq s\leq t}(f_2(s)\wedge0)\\
&= \left(f_1(t)+\sup_{0\leq s\leq t}(-f_1(s)\vee 0)\right)-\left(f_2(t)+\sup_{0\leq s\leq t}(-f_2(s)\vee 0)\right)\\
&=\sup_{0\leq s\leq t}\Big[\big(f_1(t)-f_1(s)\big)\vee f_1(t)\Big]-\sup_{0\leq s\leq t}\Big[\big(f_2(t)-f_2(s)\big)\vee f_2(t)\Big]\\
&\leq0.
\end{align*}
\end{proof}
\end{prop}

\begin{prop}[Shift Property of $\Gamma$]\label{shift property of Gamma}
For any $f\in\D_0([0,\infty),\R)$ and $t,s\geq0$,  
\[
\Gamma[f](s+t)=\Gamma\big[\Gamma[f](s)+f(\cdot+s)-f(s)\big](t).
\]
\end{prop}

\begin{proof}
Fix $t,s\geq0$. By the definition of $\Gamma$,
\begin{align*}
\Gamma[f](t+s)&=f(t+s)-\inf_{0\leq u\leq t+s}\big(f(u)\wedge 0\big)\\
&=f(t+s)-f(s)+f(s)-\inf_{0\leq u\leq s}\big(f(u)\wedge 0\big)\\
&\qquad-\Big(\inf_{0\leq u\leq t+s}\big(f(u)\wedge 0\big)-\inf_{0\leq u\leq s}\big(f(u)\wedge 0\big)\Big)\\
&=\Gamma[f](s)+f(t+s)-f(s)\\
&\qquad-\Big(\inf_{0\leq u\leq t+s}\big(f(u)\wedge 0\big)-\inf_{0\leq u\leq s}\big(f(u)\wedge 0\big)\Big).
\end{align*}
Moreover, 
\begin{align*}
    &\Gamma\big[\Gamma[f](s)+f(\cdot+s)-f(s)\big](t)\\
    &\qquad=\Gamma[f](s)+f(t+s)-f(s)-\inf_{0\leq v\leq t}\Big(\big(\Gamma[f](s)+f(v+s)-f(s)\big)\wedge0\Big)\\
    &\qquad=\Gamma[f](s)+f(t+s)-f(s)-\inf_{0\leq v\leq t}\bigg(\Big(f(v+s)-\big(\inf_{0\leq u\leq s}f(u)\wedge 0\big)\Big)\wedge 0\bigg).
\end{align*}
Therefore, it suffices to show that 
\begin{align}
\inf_{0\leq u\leq t+s}\big(f(u)\wedge 0\big)-\inf_{0\leq u\leq s}\big(f(u)\wedge 0\big)=\inf_{0\leq v\leq t}\bigg(\Big(f(v+s)-\big(\inf_{0\leq u\leq s}f(u)\wedge 0\big)\Big)\wedge 0\bigg).\label{the sufficient condition of the Skorohod proof}
\end{align}
If $f(v+s)\geq \inf\limits_{0\leq u\leq s}\big(f(u)\wedge0\big)$ for all $0\leq v\leq t$, then both the left and right side of \eqref{the sufficient condition of the Skorohod proof} are zero. Otherwise, there exists $v_1\in(0,t]$ such that 
    \[
    f(v_1+s)<\inf_{0\leq u\leq s}\left(f(u)\wedge0\right)\leq0.
    \]
    Then $f(v_1+s)-\inf\limits_{0\leq u\leq s}\left(f(u)\wedge0\right)<0$ and $\inf\limits_{0\leq u\leq s+t}\left(f(u)\wedge0\right)=\inf\limits_{0\leq u\leq s+t}f(u)=\inf\limits_{0\leq v\leq t}f(v+s)$. Hence, 
    \begin{align*}
        &\inf_{0\leq v\leq t}\bigg(\Big(f(v+s)-\big(\inf_{0\leq u\leq s}f(u)\wedge 0\big)\Big)\wedge 0\bigg)
        =\inf_{0\leq v\leq t}\left(f(v+s)-\inf_{0\leq u\leq s}\left(f(u)\wedge0\right)\right)\\
        &\qquad=\inf_{0\leq v \leq t}f(v+s)-\inf_{0\leq u\leq s}\left(f(u)\wedge0\right)
        =\inf_{0\leq u \leq t+s}f(u)-\inf_{0\leq u\leq s}\left(f(u)\wedge0\right)\\
        &\qquad=\inf_{0\leq u\leq t+s}\big(f(u)\wedge 0\big)-\inf_{0\leq u\leq s}\big(f(u)\wedge 0\big).
    \end{align*}
    Thus, \eqref{the sufficient condition of the Skorohod proof} holds.
\end{proof}

%------------------------------A.2---------------------------------------------
\subsection{Proofs of Functional Central Limit Theorems}\label{ap:FCLTs}
Let $J\subseteq\R_+$ be an unbounded strictly increasing sequence. For each $j\in J$, let $\{\xi_i^j\}_{i\in\N}$
be a sequence of real valued, independent and identically distributed (i.i.d.) random variables with finite mean and finite, positive variance. For each $j\in J$, $i\in\N$ and $t\geq0$, let
$$\hat\xi_i^j=\frac{\xi_i^j-{\mathbb E}[\xi_i^j]}{\sqrt{{\mathbb E}\left[\left(\xi_i^j-{\mathbb E}[\xi_i^j]\right)^2\right]}}.
\qquad\text{and}\qquad
\hat\zeta^j(t)
=
\frac{1}{j}\sum_{i=1}^{\lfloor j^2t\rfloor}\hat\xi_i^j.
$$
The following proposition is a special case \cite[Theorem 3.1]{Prokhorov}.
\begin{prop}
\label{prop:FCLT}
If 
\begin{equation}
\lim_{j\to\infty}{\mathbb E}\left[ (\xi_1^j)^2 : (\xi_1^j)^2> j\right]=0,  \label{1*}
\end{equation}
then $\hat\zeta^j(\cdot)\Rightarrow B(\cdot)$ as $j\to\infty$, where $B(\cdot)$ is a standard Brownian motion.
\end{prop}

\begin{proof}[Proof of Proposition \ref{functional clt convergence}]
For each $r\in R$, $V^r(\cdot)$ is a composition two more elementary processes. To see this, define $K(x)=\sum_{i=1}^{\lfloor x\rfloor}v_i$ for $x\in\R_+$. Then for each $n\in\N$, $K(n)$ is the total service required by the first $n$ tasks to arrive to the system. For each $r\in R$ and $t\geq0$, $V^r(t)=K(E^r(t))$. For $r\in R$ and $x\in \R_+$, let
\begin{align}
\widehat{K}^r(x)&=\frac{K(r^2x)-r^2x{\mathbb E}[v]}{r}=\frac{1}{r}\left[\sum_{i=1}^{\lfloor r^2x\rfloor}v_i-r^2x\E[v]\right]\nonumber\\
&=\frac{\sigma_S}{r}\sum_{i=1}^{\lfloor r^2x\rfloor}\frac{v_i-\E[v_i]}{\sigma_S}+\frac{(\lfloor r^2x\rfloor-r^2x)\E[v]}{r}\nonumber\\
&=\sigma_S\left[\frac{1}{r}\sum_{i=1}^{\lfloor r^2x\rfloor}\left(\frac{v_i-\E[v_i]}{\sigma_S}\right)\right]+\left(\frac{\lfloor r^2x\rfloor-r^2x}{r}\right)\E[v].\label{2*}
\end{align}
For each $x\in\R_+$, $0\leq \frac{r^2x-\lfloor r^2x\rfloor}{r}\leq\frac{1}{r}$, and so $\lim\limits_{r\to\infty}\sup\limits_{x\in\R_+}\left(\frac{\lfloor r^2x\rfloor-r^2x}{r}\right)\E[v]=0$.  For $i\in\N$, let $\xi_i=v_i$ and $\hat\xi_i=\frac{v_i-\E[v_i]}{\sigma_S}$, and for each $r\in R$, let $\hat\xi_i^r=\hat\xi_i$ for all $i\in\N$. Since $\E[v^2]<\infty$, $\lim\limits_{r\to\infty}\E\left[(\xi_1^r)^2:(\xi_1^r)^2>r\right]=0$. Therefore \eqref{1*} holds.
Combining the above two observations with \eqref{2*}, Proposition \ref{prop:FCLT}, and the convergence together theorem shows that as $r\to\infty$, 
\begin{equation}
\widehat{K}^r(\cdot)\Rightarrow K^*(\cdot)\overset{d}{=} \sigma_SB(\cdot).    \label{1*'}
\end{equation}
In particular, $K^*(\cdot)$ is a Brownian motion with zero drift and standard deviation $\sigma_S$. Since $\widehat{K}^r(\cdot)$ and $\widehat{E}^r(\cdot)$ are mutually independent for all $r\in R$, \eqref{1*'} and Proposition \ref{E hat convergence} imply that as $r\to\infty$,
\begin{equation}
\left(\widehat{E}^r(\cdot),\widehat{K}^r(\cdot)\right)\Rightarrow\left(E^*(\cdot),K^*(\cdot)\right),    \label{3*}
\end{equation}
where $E^*(\cdot)$ and $K^*(\cdot)$ are mutually independent with $E^*(\cdot)$ as in \eqref{E hat convergence} and $K^*(\cdot)$ as in \eqref{1*'}.
\noindent For each $r\in R$ and $t\geq0$, we have that
\begin{equation}
    \widehat{V}^r(t)
=\frac{K(E^r(r^2t))-\rho^r r^2 t}{r}
=\frac{K(r^2\overline{E}^r(t))-\rho^r r^2 t}{r}
=\widehat{K}^r(\overline{E}^r(t))+{\mathbb E}[v]\widehat{E}^r(t).\label{2**}
\end{equation}
This together with the second convergence in \eqref{E hat convergence}, \eqref{3*} and the convergence together theorem implies that as $r\to\infty$, $\widehat{V}^r(\cdot)\Rightarrow V^*(\cdot):=K^*(\lambda(\cdot))+\E[v]E^*(\cdot)$, where $K^*(\lambda(\cdot))$ is a Brownian motion with zero drift and standard deviation $\sigma_S\sqrt{\lambda}$ and $\E[v]E^*(\cdot)$ is a Brownian motion with zero drift and standard deviation $\E[v]\lambda^{3/2}\sigma_A=\lambda^{1/2}\sigma_A$. By the independence of $K^*(\cdot)$ and $E^*(\cdot)$, $V^*(\cdot) \overset{d}{=}\sigma B(\cdot)$.
\end{proof}

\begin{proof}[Proof of Proposition \ref{Def of X*}]
Note that by definition, $\widehat{X}^r(t)=\widehat{W}^r(t)+\widehat{V}^r(t)+(\rho^r-1)rt$ for $t\geq0$ and $r\in R$. By Proposition \ref{functional clt convergence}, $\widehat{V}^r(\cdot)\Rightarrow\sigma B(\cdot)$ as $r\to\infty$ and by \eqref{r rho convergence} $\lim_{r\to\infty}(\rho^r-1)r=\kappa$. Let $W^*(0)\overset{d}{=} \widehat{W}^*(0)$ be defined on the same probability space as $B(\cdot)$ and independent of $B(\cdot)$. Then, since $\widehat{W}^r(0)$ is independent of $\widehat{V}^r(\cdot)$ for each $r\in R$, and $\widehat{W}^r(0)\Rightarrow W^*(0)$ as $r\to\infty$, it follows that $(\widehat{W}^r(0),\widehat{V}^r(\cdot))\Rightarrow(W^*(0), \sigma B(\cdot))$ as $r\to\infty$. This together with the convergence together theorem implies that \eqref{difusion of X} holds.

To see that \eqref{diffusion of W} holds, 
note that,  $\widehat{W}^r(t) = \Gamma[\widehat{X}^r](t)$ for all $t\geq0$ and $r\in R$, by \eqref{I(t)}.
Then \eqref{diffusion of W} follows from this, \eqref{difusion of X}, the almost sure continuity of $X^*(\cdot)$, 
continuity of $\Gamma$ in the Skorokhod $J_1$-topology 
at points in $\C(\R_+)$
and the continuous mapping theorem (see \cite[Section 5.2.2]{Whitt}).
\end{proof}

%------------------------------A.3---------------------------------------------
\subsection{Bias Random Walk Computation}\label{Appendix B}
\noindent Let $\delta>0$ and recall  the definition of the biased random walk $\{S_\delta(k)\}_{k\in\Z_+}$ given in \eqref{defS}.
For $l\in\Z$, we refer to $3lB\delta^{1+\eta}/2$ as level $l$.  For $l\in\N$ and $1\leq j\leq l$, let
\[
h_j^l:=\P\left(S_\delta(\cdot)\text{ hits level $l$ before hitting level 1}\middle\vert S_\delta(0)\text{ is at level }j\right) = \frac{2^j-2}{2^l-2},
\]
where the second equality follows by a standard stopping time calculation, e.g., see \cite[equation (5.13)]{ref:TK}.

\begin{proof}[Proof of \eqref{S delta crosses}]
Let $\delta\in(0,1/e]$. Then $\log(1/\delta)\geq 1$.  Using the fact that the step size is $3B\delta^{1+\eta}/2$, we have 
\begin{align*}
&\P\left(S_\delta(\cdot)\text{ crosses }48B\delta^{1+\eta}\log{(1/\delta)}\text{ before }\frac{3B\delta^{1+\eta}}{2}\middle\vert S_\delta(0)=\frac{9B\delta^{1+\eta}}{2}\right)\\
&\qquad=h_3^{\lceil 32\log{(1/\delta)}\rceil}=\frac{2^3-2}{2^{\lceil 32\log{(1/\delta)}\rceil}-2}\leq\frac{2^3-2}{2^{32\log{(1/\delta)}}-2}.
\end{align*}
\end{proof}

%------------------------------A.4---------------------------------------------
\subsection{Two Martingale Inequalities}

We apply Doob's maximal inequality and the Azuma Hoeffding inequality at various stages of our analysis. We state these here for the reader's convinence.
\begin{prop}[Doob's Maximal Inequality] 
If $X$ is a martingale indexed by the finite set $\{0,1,\dots,N\}$, then for every $\gamma\geq 1$ and $\upsilon>0$
\[
\upsilon^\gamma\P\left(\max_{n\in\{0,1,\dots,N\}}|X_n|\geq\upsilon\right)\leq\E\Big[|X_N|^\gamma\Big].
\]
\begin{proof}
See Corollary II.1.6 of \cite{Revuz}.
\end{proof}
\end{prop}

\begin{prop}[Azuma Hoeffding Inequality]\label{Azuma Hoeffding}
Suppose $\{X_k\}_{k\in\Z_+}$ is a martingale and $\{c_k\}_{k\in\Z_+}$ is a sequence of positive numbers such that $|X_k-X_{k-1}|\leq c_k$ almost surely for all $k\in\N$. Then for any $N\in\Z_+$ and for any $\gamma>0$, we have
\[
\P(X_N-X_0\leq -\gamma)\leq\exp{\left(\frac{-\gamma^2}{2\sum_{k=1}^Nc_k^2}\right)}.
\]
\begin{proof}
See Theorem 2.25 in \cite{Janson}.
\end{proof}
\end{prop}
\end{appendices}

%------------------------------Bib---------------------------------------------


\begin{thebibliography}{99}

\bibitem{Atar}
Atar, R., Biswas, A., Kaspi, H., Ramaman, K.:
A Skorokhod map on measure-valued paths with applications to priority queues. Annals of Applied Probability 28, 418-481 (2018). \url{https://doi.org/10.1214/17-AAP1309}

\bibitem{heavy tails}
Banerjee, S., Budhiraja, A., Puha, A.L.: Heavy traffic scaling limits for shortest remaining processing time queues with heavy tailed processing time distributions. Annals of Applied Probability 32, 2587-2651 (2022). \url{https://doi.org/10.1214/21-AAP1741}

\bibitem{Billingsley}
Billingsley, P.: Convergence of Probability Measures. Wiley, New York (2013). \url{https://doi.org/10.1002/9780470316962}

\bibitem{Bingham}
Bingham, N., Goldie, C., Teugels, J.: Regular Variation. Cambridge University Press, Cambridge (1987). \url{https://doi.org/10.1017/CBO9780511721434}
% Bingham, N., Goldie, C., Teugels, J.: Regular variation. Encyclopedia of mathematics and its applications \textbf{27} (1987). Cambridge Univ. Press, Cambridge. 

\bibitem{Chen}
Chen, Y., Dong, J.: Scheduling with Service-Time Information: The Power of Two Priority Classes. Preprint (2021). \url{https://arxiv.org/abs/2105.10499}

\bibitem{Dong}
Dong, J., Ibrahim, R.: On the SRPT scheduling discipline in many-server queues with impatient customers. Management Science 67, 7291-7950 (2021). \url{https://doi.org/10.1287/mnsc.2021.4110}

\bibitem{Down}
Down, D., Gromoll, H.C., Puha, A.L.,: Fluid limits for shortest remaining processing time queues. Mathematics of Operations Research 34, 880-911 (2009). \url{https://doi.org/10.1287/moor.1090.0409}

\bibitem{Down_Sig}
Down, D., Gromoll, H.C., Puha, A.L.:
State-dependent response times via fluid limits for shortest remaining processing time queues. San Diego ACM-Sigmetrics Performance Evaluation 27, 75-76 (2009). \url{https://doi.org/10.1145/1639562.1639593}

%\bibitem{Durrett}
%R. Durrett, \textit{Probability: Theory and Examples}, Ver 5, 2019.

\bibitem{Ethier and Kurtz}
Ethier, S., Kurtz, T.G.: Markov processes: characterization and convergence. Wiley, New York (1986). \url{https://doi.org/10.1002/9780470316658}

\bibitem{Gromoll}
Gromoll, H.C., Kruk, L., Puha, A.L.: Diffusion limits for shortest remaining processing time queues. Stochastic Systems 1, 1-16 (2011). \url{https://doi.org/10.1214/10-SSY016}

\bibitem{Grosof}
Grosof, I., Scully, Z., Harchol-Balter, M.: SRPT for multiserver systems. Performance Evaluation 127, 154–175 (2018). \url{https://doi.org/10.1145/3308897.3308902}

\bibitem{KrukSoko}
Kruk, L., Sokolowska, E.:
Flud limits for multiple-input shortest remaining processing time queues. Mathematics of Operations Research 41, 1055-1092 (2016). \url{https://doi.org/10.1287/moor.2015.0768}

\bibitem{Lin}
Lin, M., Wierman, A., Zwart, B.:
The heavy-traffic growth rate of shortest remaining processing time queues. Performance Evaluation 68, 955-966 (2011). \url{https://doi.org/10.1016/j.peva.2011.06.001}

\bibitem{MultipleChannel}
Iglehart, D.L., Whitt, W.: Multiple channel queues in heavy traffic. Advances in Applied Probability 2, 150-177 (1970). \url{https://doi.org/10.2307/3518347}

\bibitem{Jakubowski}
Jakubowski, A.: On the Skorokhod topology. Annales De L Institut Henri Poincare-probabilites Et Statistiques 22, 263-285 (1986). 

\bibitem{Janson}
Janson, S., Luczak, T., Rucinski, A.: Random Graphs. Wiley, New York (2000). \url{https://doi.org/10.1002/9781118032718}

%\bibitem{Lindvall}
%Lindvall,T: Lectures on the Coupling Method. Dover Publications, New York (2002). %\url{}

\bibitem{Prokhorov}
Prokhorov, Y.V.: Convergence of random processes and limit theorems in probability theory. Theory of Probability \& Its Applications 1, 157-214 (1956). \url{https://doi.org/10.1137/1101016}

\bibitem{Puha}
Puha, A.L.: Diffusion limits for shortest remaining processing time queues under nonstandard spatial scaling. The Annals of Applied Probability 25, 3381–3404 (2015). \url{https://doi.org/10.1214/14-AAP1076}

%\bibitem{Resnick}
%S.I. Resnick, \textit{Adventures in Stochastic Processes}, Springer, New York, 2002.

\bibitem{Revuz}
Revuz, D., Yor, M.: Continuous Martingales and Brownian Motion. Springer, New York (1998). \url{https://doi.org/10.1007/978-3-662-06400-9}

\bibitem{Old Optimal}
Schrage, L.E.: A proof of the optimality of the shortest remaining processing time discipline. Operations Research 16 687-690 (1968). \url{https://doi.org/10.1287/opre.16.3.687}

\bibitem{Schreiber}
Schreiber, F.: Properties and applications of the optimal queueing strategy SRPT: A survey.
Archiv f\"ur Elektronik und \"Ubertragungstechnik 47 372–378 (1993).

\bibitem{New Optimal}
Smith, D.R.: A new proof of the optimality of the shortest remaining processing time discipline. Operations Research 26 197–199 (1978). \url{https://doi.org/10.1287/opre.26.1.197}

\bibitem{ref:TK}
Taylor, H.E., Karlin, S.: An Introduction to Stochastic Modeling, Third Edition. Academic Press, Cambridge (1998).

\bibitem{Whitt}
Whitt, W.: Stochastic-Processing Limits. Springer, New York (2002).
% . 436.

\end{thebibliography}
\end{document}